\documentclass{amsart}
\usepackage{amssymb,enumerate}
\usepackage[chatter]{rotating}
\usepackage{epsfig}
\usepackage{fancyvrb}
\newtheorem{theorem}{Theorem}
\newtheorem{lemma}{Lemma}[section]
\newtheorem{cor}[lemma]{Corollary}
\newtheorem{proposition}[lemma]{Proposition}
\newtheorem{conjecture}{Conjecture}

\newcommand{\Z}{{\mathbb Z}}
\newcommand{\R}{{\mathbb R}}
\newcommand{\N}{{\mathbb N}}

\newcommand{\rad}{{\rm rad}}

\title[]{$2A$-Majorana Representations of $A_{12}$}
\author{Clara Franchi, Alexander A. Ivanov, Mario Mainardis}
\begin{document}

\maketitle

\begin{abstract}
Majorana representations have been introduced by Ivanov in~\cite{Iva} in order to provide an axiomatic framework for studying the actions on the Griess algebra of the Monster and of its subgroups generated by Fischer involutions. A crucial step in this programme is to obtain an explicit description of the Majorana representations of $A_{12}$ (by \cite{FIM2}, the largest alternating group admitting a Majorana representation)   for this might eventually lead to a new and independent construction of the Monster group (see ~\cite[Section 4, pag.115]{IvInd}). 

In this paper we prove that $A_{12}$ has a unique Majorana representation on the set of its involutions of type $2^2$ and $2^6$ (that is the involutions that fall into the class of Fischer involutions when $A_{12}$ is embedded in the Monster) and we determine the degree and the decomposition into irreducibles of such representation. As a consequence we get that Majorana algebras affording a $2A$-representation of $A_{12}$ and of the Harada-Norton sporadic simple group satisfy the Straight Flush Conjecture (see~\cite{IvCon} and~\cite{IvInd}). As a by-product we also determine the degree and the decomposition into irreducibles of the Majorana representation induced on the $A_8$ subgroup of $A_{12}$. We finally state a conjecture about Majorana representations of the alternating groups $A_n$, $8\leq n\leq 12$. \end{abstract}

 \section{Introduction}
 
Let  $k$ be a field and $R$ a commutative associative ring containing $k$ as a subring.  Given a finite subset $\mathcal S$ of $R$, containing $1$ and $0$,  a {\it fusion law} on $\mathcal S$ is a map 
$$\star\colon \mathcal S \times \mathcal S \to 2^{\mathcal S}.$$
The  {\it pythagorean table} of $\star$ is the matrix  whose rows and columns are indexed by the elements of $S$ and, for $a$ and $b$ in $S$, the entry of row $a$ and column $b$ is $a\star b$. 
Let $\mathcal S$ and $\star$ be as above and let $V$ be a commutative non-associative $R$-algebra. For an element $a\in V$ define the {\it adjoint action}
of $a$ on $V$ to be the map 
$$\begin{array} {rccc}
{\rm ad}(a)\colon& V &\to& V\\
&v&\mapsto&av
\end{array}$$
 An idempotent element $a\in V$ is called a $\star$-{\it axis} (or simply {\it axis}) if    
\begin{enumerate}
\item[Ax1] ${\rm ad}(a)$ is a semisimple endomorphism of $V$ with spectrum contained in $ \mathcal S$,
\item[Ax2] for every $\lambda, \mu \in \mathcal S$, $V_\lambda V_\mu \leq \oplus_{\delta\in \lambda\star \mu}V_\delta$, 
where, for every $\theta\in {\mathcal S}$, $V_\theta$ is the eigenspace $\{v\in V| av=\theta v\}$ (we allow the possibility that $V_\theta=\{0\}$).
\end{enumerate} 
For example, if $\star$ is the fusion law with pythagorean table as in  (\ref{isingf}) and  if $u$ and $v$ are, e.g., $1/32$-eigenvectors for the adjoint action of the $\star$-axis $a$, their product lies in the sum of the $1$, $0$, and $1/4$-eigenspaces for ${\rm ad}(a)$.
If the algebra $V$ is generated by $\star$-axes, we say that $V$ is an {\it axial algebra}  over $R$ with  spectrum $\mathcal S$ and fusion law $\star$.  An axial algebra $V$ is called {\it primitive} if 
\begin{enumerate}
\item[Ax3] for every axis $a$, the $1$-eigenspace $V_1$ of $ad(a)$ has dimension $1$ (or, equivalently, $V_1$ is the linear span of  $a$) 
\end{enumerate}
and {\it dihedral} if
\begin{enumerate}
\item[Ax4] $V$ is generated by two axes.
\end{enumerate}
A {\it Frobenius axial algebra} is a pair $(V, \sigma)$ where $V$ is an axial algebra and $$\sigma\colon V\times V \to k$$ is a bilinear form on $V$ such that
\begin{enumerate}
\item[Ax5] $\sigma$ {\it associates} with the algebra product, i.e.: $\sigma(uv,w)=\sigma(u,vw)$, for every $u,v,w\in V$\end{enumerate}
and 
\begin{enumerate}
\item [Ax6] for each axis $a$, $\sigma(a,a)\neq 0$.
\end{enumerate}
For an element $v\in V$, define, as usual, the (squared) {\it length} of $v$ to be the value $\sigma(v,v)$.
 
\noindent {\it Majorana} algebras are primitive real Frobenius axial algebras $(V, \sigma)$ such  that $\sigma$ is positive definite, axes have length $1$,  with spectrum $\{1,0,1/4, 1/32\}$, and {\it Ising} (or {\it Monster}) fusion law:
\begin{equation}
\label{isingf}
 \begin{array}{|c||c|c|c|c|}
\hline
   \star& 1    & 0    & 1/4      & 1/32\\
   \hline
   \hline
   1     & \{1\}    &\emptyset    & \{1/4\}      &\{1/32\}\\
   \hline
   0     & \emptyset    & \{0\}    & \{1/4\}      &\{1/32\}\\
   \hline
   1/4  & \{1/4\} & \{1/4\} & \{1,0\}  & \{1/32\}\\
   \hline
   1/32& \{1/32\}&\{1/32\}& \{1/32\}& \{1,0,1/4\}\\
    \hline
    %&&&&
    \end{array}
    \end{equation}

\medskip

\noindent Moreover it is also required that $\sigma$ satisfies the {\it Norton inequality}\footnote{Norton inequality is satisfied by the Griess algebra and, because of that, has become part of the axiomatics of Majorana algebra. We observe, however, that within this theory it has never been used so far.}: for every $u,v\in V$,
    $$ \sigma(uu,vv)\geq \sigma(uv,uv).$$
    As usual, when the  form $\sigma$ needs not to be specified, we write simply $V$ for $(V, \sigma)$. 
   An {\it automorphism} of a Majorana algebra $V$ is an isometry of $V$ that preserves the algebra product, in particular it sends axes to axes. The set of automorphisms of a Majorana algebra $V$ is a group and will be denoted by ${\rm aut}(V)$. 
The Ising fusion law implies that, for a Majorana algebra $V$,   the setting 
 $$V_+:=V_1\oplus V_0\oplus V_{\frac{1}{4}} \mbox{ and } V_-:=V_{\frac{1}{32}},$$
 is  a $\Z_2$-grading on $V$. Since, by [Ax5], the scalar product is associative, it follows that,  
  for each axis $a$, the map  $\tau_a$  that inverts every element of its $1/32$-eigenspace and fixes each element of the other eigenspaces is an involutory automorphism of $V$. The automorphism 
 $\tau_a$ is called the {\it Miyamoto involution} associated to the axis $a$ (the name comes after  Masahiko Miyamoto who introduced these involutions in~\cite{Miya} in the context of vertex operator algebras).
 By a result of John Conway (see~\cite{Co}),  the ($196884$-dimensional) Griess algebra constructed by Robert Griess in~\cite{Griess}  is a Majorana algebra (see also~\cite[Section 8.5]{Iva}). In particular, identifying the Monster with the automorphism group of the Griess algebra, the Miyamoto involutions of the Griess algebra are precisely the {\it Fischer involutions} of the Monster, i.e. the involutions in the Monster whose centraliser is the double cover of the Baby Monster. These involutions form a unique conjugacy class in the Monster  that is labelled $2A$ in the ATLAS~\cite{ATLAS} and their name comes after Bernd Fischer, who, together with Griess, first foreshadowed the existence of the Monster.
Dihedral Majorana algebras are called {\it Norton-Sakuma algebras}. The complete classification of Norton-Sakuma algebras goes under the name of {\it Norton-Sakuma Theorem} and it has been achieved by  Simon P. Norton~\cite{Norton} for the subalgebras of the Griess algebra, by  Shinya Sakuma~\cite{Sakuma} in the context of certain vertex operator algebras,  and, within the Majorana algebra axiomatics, by Ivanov, Dima Pasechnik, \'Akoss Seress, and Sergey Shpectorov~\cite{IPSS}. 

\medskip
 
\noindent {\bf  Norton-Sakuma Theorem.} {\it There are nine isomorphism classes of Norton-Sakuma algebras. The representatives of these classes, are described in the rows of Table~\ref{table1}. }

\medskip
 
In Table~\ref{table1} the basis given for each representative is called the {\it Norton basis},  $\rho$ denotes the product of the two Miyamoto involutions associated to the generating axes $a_0$ and $a_1$ and, for $\epsilon \in\{0,1\}$ and $i\in\Z$,  $a_{\epsilon+2i}$ is the image of $a_\epsilon$ under $\rho^i$. Note that, since  $\rho$ is an algebra automorphism, $a_{\epsilon+2i}$ is still an axis. The {\it type} of a Norton-Sakuma algebra is the name given in the ATLAS~\cite{ATLAS} for the conjugacy class in the Monster of the product of the Miyamoto involutions  associated to the generating axes (when this algebra is identified with a subalgebra of the Griess algebra). The vectors $a_\rho$, $a_{\rho^2}$, $a_{\rho^3}$ appearing in the algebras of type $2A$, $4B$, and $6A$ are axes whose associated Miyamoto involutions are $\rho$, $\rho^2$, and $\rho^3$, respectively. Further, the vectors $u_\rho$, $v_\rho$, $w_\rho$, and $u_{\rho^2}$,  appearing in algebras of type $NA$, for $N\in\{3,4,5,6\}$, are called {\it odd axes} or, more specifically, $N$-axes, for $N<6$, and $3$-axis for $N=6$.   Odd axes do not depend on the generating pair $(a_0, a_1)$ but (up to a sign in the case $5A$) only on the cyclic subgroup  $\langle \rho \rangle$ of index $2$ in  the dihedral subgroup of the Monster generated by the associated Miyamoto involutions. $3$-, $4$-, and $5$-axes will also be called {\it odd} axes.
 \begin{table}
{\tiny 
$$
\begin{array}{|l|l|l|l|}
\hline

\mbox{Type}& \mbox{Basis} &\mbox{Structure constants} &\mbox{Scalar products}\\
\hline
& &  &\\
1A & \begin{array}{l} a_0 \end{array}&
\begin{array}{l}
 a_0\cdot a_0=a_0 \end{array}
 &
\begin{array}{l}
(a_0,a_0)=1 \end{array}\\
 & &  &\\
 \hline
& & &\\
2A 
&\begin{array}{l} a_0,\\ a_1,\\ a_\rho \end{array}
&
\begin{array}{l}
 a_0\cdot a_1=\frac{1}{2^3}(a_0+a_1-a_\rho),\\
 \\
 a_0\cdot a_\rho=\frac{1}{2^3}(a_0+a_\rho-a_1)\\
 \\
 a_\rho\cdot a_\rho=a_\rho
 \end{array}
&
\begin{array}{l}
 (a_0, a_1)=\frac{1}{2^3}\\
 \\
 (a_0, a_\rho)=\frac{1}{2^3}\\
 \\
 (a_1, a_\rho)=\frac{1}{2^3}
 \end{array}
 \\
 & & & \\
 \hline
  & &  &\\
2B &\begin{array}{l} a_0,\\ a_1 \end{array}& \begin{array}{l} 
a_0\cdot a_1=0 \end{array}& \begin{array}{l} (a_0, a_1)=0 \end{array}\\
 & & & \\
 \hline
  & & & \\

 3A & \begin{array}{l} a_{-1},\\ a_0,\\ a_1,\\ u_\rho \end{array}&
 \begin{array}{l}
 a_0\cdot a_1=\frac{1}{2^5}(2a_0+2a_1+a_{-1}) -\frac{3^3\cdot 5}{2^{11}}u_\rho\\
 \\
  a_0\cdot u_\rho=\frac{1}{3^2}(2a_0-a_1-a_{-1})+\frac{5}{2^5}u_\rho\\
  \\
    u_\rho\cdot u_\rho= u_\rho\\
  \end{array}
  &
  \begin{array}{l}
   (a_0, a_1)=\frac{13}{2^8},\\ 
   \\
   (a_0, u_\rho)=\frac{1}{4},\\
   \\
   (u_\rho, u_\rho)=\frac{2^3}{5}
 \end{array}\\
 & & & \\
 \hline
 & & & \\
 3C &\begin{array}{l} a_{-1},\\ a_0,\\ a_1 \end{array}&  \begin{array}{l} a_0\cdot a_1=\frac{1}{2^6}(a_0+a_1-a_{-1})\end{array} & \begin{array}{l}(a_0,a_1)=\frac{1}{2^6}\end{array} \\
 & & & \\
 \hline
 & &  &\\
 4A & \begin{array}{l}a_{-1}, \\a_0, \\a_1,\\ a_2, \\v_\rho \end{array} 
 &  \begin{array}{l}         
 a_0\cdot a_1= \frac{1}{2^6}(3a_0+3a_1+a_{-1}+a_2-3v_\rho) \\
 \\
  a_0\cdot a_2=0\\
 \\
 a_0\cdot v_\rho= \frac{1}{2^4}(5a_0-2a_1-2a_{-1}-a_2+3v_\rho)\\
 \\
 v_\rho\cdot v_\rho=v_\rho
 \end{array}
 & \begin{array}{l}  
 (a_0,a_1)=\frac{1}{2^5}\\
 \\
 (a_0,a_2)=0 \\
 \\
 (a_0,v_\rho)=\frac{3}{2^3}\\
 \\
  (v_\rho,v_\rho)=2
           \end{array} \\
 & & & \\
  \hline
 & & & \\
4B & \begin{array}{l} a_{-1}, \\a_0,\\ a_1,\\a_2\\ a_{\rho^2} \end{array}
&  \begin{array}{l}   
a_0\cdot a_1= \frac{1}{2^6}(a_0+a_1-a_{-1}-a_2+a_{\rho^2}) \\
\\
 a_0\cdot a_2= \frac{1}{2^3}(a_0+a_2-a_{\rho^2})
 \end{array}
& \begin{array}{l}    
(a_0,a_1)=\frac{1}{2^6}\\
 \\
 (a_0,a_2)=\frac{1}{2^3} \\
 \\
 (a_0,a_{\rho^2})=\frac{1}{2^3}  
 \end{array} \\
 & &  &\\
 \hline
 & &  &\\
 5A & \begin{array}{l}a_{-2},\\ a_{-1},\\ a_0, \\a_1,\\ a_2, \\w_\rho \end{array}
&  \begin{array}{l}  
 a_0\cdot a_1= \frac{1}{2^7}(3a_0+3a_1-a_{-1}-a_2-a_{-2})+ w_\rho \\
 \\
 a_0\cdot a_2=\frac{1}{2^7}(3a_0+3a_2-a_1-a_{-1}-a_{-2})- w_\rho \\
 \\
 a_0\cdot w_\rho= \frac{7}{2^{12}}(a_1+a_{-1}-a_2-a_{-2})+ \frac{7}{2^{5}}w_\rho \\
 \\
 w_\rho\cdot w_\rho=\frac{5^2 \cdot 7}{2^{19}}(a_0+a_1+a_{-1}+a_2+a_{-2}) 

           \end{array}
& \begin{array}{l}     
 (a_0,a_1)=\frac{3}{2^7}\\
 \\
 (a_0,w_\rho)=0\\
 \\
  (w_\rho,w_\rho)=\frac{5^3 \cdot 7}{2^{19}}
        \end{array}  \\
 & & & \\
 \hline
 & & & \\
 6A & \begin{array}{l}
 a_{-2},\\ a_{-1},\\ a_0,\\
 a_1,\\ a_2,\\ a_3, \\
 a_{\rho^3},\\ u_{\rho^2} \end{array}
& \begin{array}{l}   
 a_0\cdot a_1= \frac{1}{2^6}(a_0+a_1-a_{-1}-a_2-a_{-2}-a_3+a_{\rho^3})+\frac{3^2\cdot5}{2^{11}} u_{\rho^2} \\
 \\
  a_0\cdot a_2= \frac{1}{2^5}(2a_0+2a_2 + a_{-2})-\frac{3^3\cdot5}{2^{11}} u_{\rho^2} \\
 \\
 a_0\cdot a_3=\frac{1}{2^3}(a_0+a_3-a_{\rho^3})\\
 \\
 a_0\cdot u_{\rho^2}= \frac{1}{3^{2}}(2a_0-a_2-a_{-2})+ \frac{5}{2^{5}}u_{\rho^2} \\
 \\
 a_{\rho^3}\cdot u_{\rho^2}=0

          \end{array}
& \begin{array}{l}    
(a_0,a_1)=\frac{5}{2^8}\\
 \\
 (a_0,a_2)=\frac{13}{2^8}\\
 \\
  (a_0,a_3)=\frac{1}{2^3}\\
  \\
  ( a_{\rho^3}, u_{\rho^2})=0

         \end{array}  \\
 & & & \\
\hline

\end{array} 
 $$}
 \caption{The nine types of Norton-Sakuma algebras}
  \label{table1}
 \end{table}

 The properties of the action of  Monster on the Griess algebra, together with the correspondence between axes and Miyamoto involutions are  axiomatised  in the definition of Majorana representation: 
given
\begin{itemize}  
\item [-] a group $G$,  
\item  [-] a $G$-invariant set of involutions ${\mathcal T}$ generating $G$ (the {\it Majorana set})  
\item  [-] a Majorana algebra $V$,  
\item  [-]  an injective map  $\psi$ from ${\mathcal T}$ to the set of axes of $V$,
\end{itemize} 
a {\it Majorana representation  of $G$ on $V$ with respect to} $\psi$, is a group homomorphism $$\phi\colon G\to {\rm aut}(V)$$ such that 
\begin{enumerate}
\item[M1]  $\mathcal T^\psi$ generates $V$ as an algebra,
\item[M2]  for every $t\in \mathcal T$ and every $g\in G$, $(t^g)^\psi=(t^\psi)^{g^\phi}$,
\item[M3]  for every $t\in {\mathcal T}$, $t^\phi$ is the Miyamoto involution associated to the axis $t^\psi$,
\item [M4]  for $t_1$ and $t_2$ in $ {\mathcal T}$, the Norton-Sakuma algebra generated by $t_1^{\psi}$ and $t_2^{\psi}$ has type $2A$ if and only if $t_1t_2\in  {\mathcal T}$ and in this case $(t_1t_2)^\psi = t_1^\psi+t_2^\psi-8t_1^\psi t_2^\psi$ (cfr. Table~\ref{table1} with $\rho= (t_1t_2)^\phi$),
\item[M5] if $t_1$, $t_2$, $t_3$, and $t_4$ are elements of ${\mathcal T}$ such that $t_1t_2=t_3t_4$ and the subalgebras generated by $t_1^{\psi}, t_2^{\psi}$ and $t_3^{\psi}, t_4^{\psi}$ both have type $3A$, $4A$, or $5A$, then $u_{(t_1t_2)^\phi}=u_{(t_3t_4)^\phi}$, $v_{(t_1t_2)^\phi}=v_{(t_3t_4)^\phi}$, or $w_{(t_1t_2)^\phi}=w_{(t_3t_4)^\phi}$, respectively.
\end{enumerate}

\medskip

\noindent Given a Majorana representation $\phi$ as above, we shall also refer to it  as the quintuple $$(G, \mathcal T, V, \phi,\psi).$$

\noindent The {\it shape} of a Majorana representation $\phi$ is the map $sh$ that assigns to each orbital of $G$ on $\mathcal T$ the isomorphism class of the dihedral algebra generated by the two axes associated to one (every) pair of involutions in that orbital.
\begin{enumerate}
\item[M6] the map $sh$ {\it respects} the following embeddings of dihedral algebras:
$$
 2A\hookrightarrow 4B, \:\:2A\hookrightarrow 6A, \:\:2B\hookrightarrow 4A, \:\:3A\hookrightarrow 6A
 $$
 \noindent in the sense that, for $t, r_1, r_2\in \mathcal{T}$, if $\langle \langle t^\psi\rangle \rangle < \langle \langle t^\psi, r_1^\psi\rangle \rangle < \langle \langle t^\psi, r_2^\psi \rangle \rangle$, then 
 $$(sh((t,r_1)^G), \,sh((t,r_2)^G))\in \{(2A,4B), (2A, 6A), (2B,4A), (3A,6A)\}.$$
\end{enumerate}
Axioms M4, M5, M6, and Norton-Sakuma Theorem  imply that,  
if $t_1^\psi$ and  $t_2^\psi$ generate a dihedral subalgebra of $V$ of type $2A$, $4B$, or $6A$, then $t_1t_2$, $(t_1t_2)^2$, or $(t_1t_3)^3$ belongs to ${\mathcal T}$, and $(t_1t_2)^\psi $, $((t_1t_2)^2)^\psi $, or $((t_1t_3)^3)^\psi $ coincide with $a_{(t_1t_2)^\phi}$, $a_{((t_1t_2)^2)^\phi}$, or  $a_{((t_1t_2)^3)^\phi}$, respectively. 
The next result follows immediately from the definition of Majorana representation.
\begin{lemma}
\label{oxf}
If ${\mathcal T_0}$ is a nonempty subset of ${\mathcal T}$ such that $\mathcal T_0$ is $\langle \mathcal T_0\rangle$-invariant, then  
$
\phi_{|_{\langle {\mathcal T_{0}}\rangle}}
$
is again a Majorana representation of $\langle {\mathcal T_{0}}\rangle$ on  the subalgebra of $V$ generated by $\mathcal T_0^\psi$  with respect to $\psi_{|_{{\mathcal T_{0}}}}$. 
\end{lemma}
The standard example of a Majorana representation is given by the usual action of the Monster on the Griess algebra, where the Majorana set  is the conjugacy class of Fischer involutions (see~\cite[Proposition 8.6.2]{Iva}). 
More generally, let  $H$ be any finite group, ${\mathcal T}_H$ an $H$-invariant generating set of involutions of $H$. 
A Majorana representation $(H, \mathcal T_H, V, \phi, \psi)$ is called  {\it $2A$-Majorana representation} if there is an embedding 
$\epsilon$ of $H$  in the Monster such that 
$${\mathcal T}_H^\epsilon=H^\epsilon\cap \mathcal T_{2A}.$$ 
In particular, if $\phi$ can also be factorised as the composition of $\epsilon$ with the Majorana representation induced on $H^\epsilon$ by the usual action of the Monster on the Griess algebra, we say that $\phi$ is {\it based on the embedding}  $\epsilon$ of $H$ in the Monster.

Note that, by the Norton-Sakuma Theorem, dihedral groups $D_n$ admit Majorana representations if and only if $n\leq 6$,  there are $9$ such representations  and they all arise from an embedding in the Monster (in particular they are $2A$-Majorana representations). Conversely, the ideal goal would be to prove that a $2A$-Majorana representation is almost always based on an embedding in the Monster, and find the exceptions.

Majorana representations of various groups have been studied and Markus Pfeiffer and  Madeleine Whybrow have produced the GAP package {\it MajoranaAlgebras}~\cite{MM} for this purpose.  For small groups this programme can compute the structure constants and the inner products of a Majorana representation of a group $G$ once the Majorana set and the shape are given. On the other hand the computational complexity increases rapidly with order of the group involved and still Majorana representations of many groups are not completely known. In particular,  the Majorana representations of the alternating groups $A_n$  have been determined only for $n\in \{5,6,7\}$ (\cite{IS12, A67, Iv11b}).  From~\cite{FIM2} one can deduce that $A_{12}$ is the largest alternating group admitting a Majorana representation, but there is only partial information about the representations of $A_n$, for $n\in \{8,9,10,11,12\}$ (see \cite{Alonso}, \cite{FIM2}, and~\cite{FIM4}). In the spirit of~\cite{Serre}, we point out that, in this paper, the package {\it MajoranaAlgebras} is needed only for for the proofs of Lemma~\ref{formulaccia} and  Theorem~\ref{A8}, and the GAP software~\cite{GAP} only for long, but routine, computations of elementary linear algebra with tabloids.
% or for determining the orbits and orbitals of some subgroups of $S_{12}$.    

The main result of the present paper is the following. 

\begin{theorem}\label{main}
There is, up to equivalence, a unique $2A$-Majorana representation $(A_{12}, \mathcal T, V, \phi,\psi)$ of $A_{12}$. In particular 
\begin{enumerate}
\item[{\it (1)}]  $\phi$ is based on the embedding of $A_{12}$ in the Monster as the centraliser of a $(2A, 3A, 5A)$-subgroup isomorphic to $A_5$.
\item  [{\it (2)}] $V$ has dimension $3,960$ and its decomposition into irreducible submodules is given in the last column of Table~\ref{dect},
\item  [{\it (3)}]  $V$ is $2$-closed and it 
is linearly spanned by the Majorana axes and the $3A$-axes associated to permutations of cycle type $3^2$.
\end{enumerate}
\end{theorem}
Recall that, given a partition $\lambda$ of $n$, there is a unique irreducible $\R[S_n]$ module  $S^\lambda$ associated to $\lambda$ (the {\it Specht} module, see e.g.~cite{J}). When restricted to the action of $A_n$, $S^\lambda$ remains irreducible if  $\lambda$ is not self conjugate, otherwise $S^\lambda$ splits into the direct sum of two non isomorphic irreducibles $\R[A_n]$-modules of the same dimension. In Table~\ref{dect} we write $1+1$ when this splitting occours, meaning that both modules appear. 
\begin{table}
$$
\begin{array}{|c|c|c|c|c|c|c|}
\hline
 &\multicolumn{5}{|c|}{W^\circ} & \multicolumn{1}{c|}{V} \\
\hline
\lambda & n=8 & n=9& n=10 & n=11& n=12& \\
\hline
(n) & 2 & 2 & 2 & 2& 1& 1\\
\hline
(n-1,1)& 2 & 2 & 2 & 2& 1&1\\
\hline
(n-2,2) & 3 &3 & 3 & 2& 1&1\\
\hline
(n-3,3)& 2 & 2 & 2 & 2& 1&1\\
\hline
(n-4,4)&1 & 1 & 1 & 1& 1&1\\
\hline
(n-3,2,1) &1 & 1 & 1 & 1& 1&1\\
\hline
(n-4,2,2) & 2& 2& 2& 1& 1&1\\
\hline
(n-5,2,2,1)& 1 & 1 & 1 & 1& 0&0\\
\hline
(n-6,2,2,2) &1 & 1 & 1 & 1& 1&1\\ 
\hline
(n-4,1,1,1,1)&1 & 1+1 & 1 & 1& 0&0\\ 
\hline
(n-5,1,1,1,1,1)&1 & 1 & 1+1 & 1& 1&1\\
\hline
(4,4,4) &  & &  & &  &1\\
\hline
\hline
\dim(V^\circ_n)& 462 &1008 & 2052 & 3498 & 3498 &3960\\
\hline
\end{array}
$$
\caption{Irreducible $\R[A_n]$-submodules and their multiplicities in $W^\circ$ and in $V$}\label{dect} 
\end{table}

As consequences of  Theorem~\ref{main} we get the following two results.%Theorem~\ref{embuno} and Theorem~\ref{HN}.
\begin{theorem}\label{embuno}
The Harada-Norton group has a unique (up to equivalence) Majorana representation, which is the one based on its embedding in the Monster. 
\end{theorem}

Recall the following conjecture made by Ivanov in~\cite{IvInd, IvCon}.
\medskip

\noindent {\bf Straight Flush Conjecture.}
{\it Suppose $A$ is an indecomposable Majorana algebra in which, for
every $i\in  \{2, 3, 4, 5, 6\}$, there exists a pair of Miyamoto involutions $t_1$ and $t_2$, such that the order of the product $t_1t_2$  is $i$. Then $A$ embeds into the Monster algebra.}
\medskip

Since, for every  $2A$-Majorana representation $(A_{12}, \mathcal T, V, \phi,\psi)$  of $A_{12}$, the set $\mathcal T$ contains involutions of cycle type $2^6$, it follows that $(A_{12}, \mathcal T, V, \phi,\psi)$ is a $2A$-Majorana representation if and only if $V$ contains a subalgebra $\langle \langle a_s, a_t\rangle \rangle$ of type $4A$ for some $s,t\in \mathcal T$ (see Lemma~\ref{shape2^6}). 

\begin{theorem}\label{HN}
The Majorana algebras affording a $2A$-Majorana representation of either $A_{12}$ or the Harada-Norton sporadic group satisfy the Straight Flush Conjecture. 
\end{theorem}

Further, for $n\in \{8,9,10,12\}$, we determine the $2$-closure of the subalgebra of the Majorana algebra $V$ affording a $2A$-Majorana representation of $A_{12}$ generated by the Majorana axes corresponding to involutions of cycle type $2^2$ of $A_n$ (in its natural embedding in $A_{12}$ as the stabiliser of $12-n$ points). We denote by $\mathcal X_b$ the set of bitranspositions of $A_n$.

\begin{theorem} \label{An}
For $n\in \{8, \ldots ,12\}$ let  $(A_{n}, \mathcal X_b, W, \phi_b,\psi_b)$ be the Majorana representation of $A_{n}$ induced by the restriction to $A_n=\langle \mathcal X_b\rangle$ of the $2A$-Majorana representation of $A_{12}$ and let  $W^\circ$  be the $2$-closure of $W$. Then, the irreducible $\R[A_n]$-submodules of $W^\circ$ and their multiplicities are the ones given in Table~\ref{dect}.
\end{theorem}

The results of this paper enable us to determine, for $n=8$, the subalgebra of $V$ that is generated by Majorana axes corresponding to the bitranspositions of $A_8$ and thus enlight the general setting of the Majorana representations of the groups $A_n$ for 
$n\in \{8, \ldots , 12\}$.

\begin{theorem} \label{A8}
Let $(A_{8}, \mathcal X_b, W, \phi_b,\psi_b)$ be the Majorana representation of $A_{8}$ induced by the restriction of a $2A$-Majorana representation of $A_{12}$ to the stabiliser of $4$ points.
Then $W$ has dimension $476$ and it is generated by $W^\circ$ together with the set 
$$\{v_\rho \:|\: \rho \mbox{ is a permutation of cycle type } 4^2 \mbox{ in } A_8\}.$$
In particular, $W$ is not $2$-closed.
\end{theorem}
We prove in Section~\ref{shape} that, for each $n\in \{8, \ldots , 11\}$, the Majorana set and the shape of a Majorana representation of the group $A_n$ are unique. This fact, together with Theorems~\ref{An} and~\ref{A8}, leads us to formulate the following conjecture.

\begin{conjecture}\label{con}
Let $n\in \{8, \ldots , 12\}$, let $(A_n, \mathcal X_b, V, \phi, \psi)$ be a Majorana representation of $A_n$ and let $V^\circ$ its $2$-closure.
Then, for every element $\tau\in A_n$ of cycle type $2^4$  there exists a vector $\delta_\tau\in V$,  depending only on $\tau$, such that
$$V=V^\circ +\langle \delta_\tau \:|\: \tau \mbox{ is a permutation of type $2^4$ in } A_n\rangle\cong V^\circ \oplus S^{(4^2, n-8)}.$$ 
\end{conjecture}

%%%%%%%%%%%%%%%%%%%%%%%%%%%%%%%%%%%%%%%%%%%%%%%%%%%%

\section{Strategy}\label{strategy}

Let $n\in \N$ with $8\leq n\leq 12$, let $G:=A_{n}$,  $\hat G:= S_{n}$, and let  $(G, \mathcal T, V, \phi,\psi)$ be a Majorana representation of $G$.
Our goal is to determine the structure of the algebra $V$ which, by definition, is generated (as an algebra) by the images via $\psi$ of the elements in $\mathcal T$.  The first step is to determine the shape of $\phi$, which will be accomplished in Section~\ref{shape}. Since the Norton-Sakuma algebras are not, in general, linearly spanned by axes (see  Table~\ref{table1}) one cannot expect $V$ to be just the linear span of the set $\mathcal T^\psi$.  The next step is, therefore, to get control of the linear span of  the products of two Majorana axes. For a subspace $U$ of $V$, we therefore  define $$U^\circ :=\langle u\cdot v| u,v\in U\rangle.$$ We shall call  $ U^\circ$ the $2$-{\it closure} of $U$ and, when $ U^\circ=U$, we say that  $U$ is $2$-{\it closed}. 
By the Norton-Sakuma Theorem, there is a subset $\mathcal O^{ax}$ of the set of the odd axes appearing in the Norton-Sakuma subalgebras  of $V$ generated by two elements of $\mathcal T^\psi$, such that 
$$V^\circ=\langle \mathcal T^\psi\cup  \mathcal O^{ax}\rangle.$$
Let 
\begin{itemize}
\item[-] $\mathcal X_b$ be the set of permutations of cycle type $2^2$,
\item[-] $\mathcal X_s$ be the set of permutations of cycle type $2^6$ (which is nonempty only if $n=12$), 
\item[-] $\mathcal X_r$ be the set of subgroups generated by permutations of cycle type  $3$,  
\item[-] $\mathcal X_t$ be the set of subgroups generated by permutations of cycle type  $3^2$, and 
\item [-] $\mathcal X=\mathcal X_b\cup\mathcal X_s\cup\mathcal X_r\cup\mathcal X_t$.
\end{itemize}
In Section~\ref{shape} we'll prove that 
$$\mathcal X_b\subseteq \mathcal T\subseteq \mathcal X_b\cup \mathcal X_s$$ 
and, in Section~\ref{4assi} we'll prove that we can choose the set $\mathcal O^{ax}$ to be the set of $3$-axes associated to the (generators of the) elements of $\mathcal X_r\cup\mathcal X_t$.
In this case we have an obvious bijection
$$ax: \mathcal X \to {\mathcal T}^\psi \cup \mathcal O^{ax},$$
sending the elements of  $\mathcal X$ to their associated axes. With respect to the action of  $G$ on $\mathcal X$ by conjugation (with orbits $\mathcal X_b$, $\mathcal X_s$, $\mathcal X_r$, and $\mathcal X_t$)  $ax$  is an isomorphism of  $G$-sets and induces a surjective homomorphism of $\R[G]$-modules 
$$ \pi \colon M_{\mathcal X} \to V^\circ,$$
where $M_{\mathcal X}$ is the real permutation module of $G$ on $\mathcal X$. 
For $x\in \{b,s,r, t\}$, denote by  $M_x$ the  permutation $\R[G]$-module on the set $\mathcal X_x$ on which $G$ acts by conjugation. Since  
$$M_{\mathcal X}= M_b\oplus M_s \oplus M_r \oplus M_t,$$ 
the structure of  $M_{\mathcal X}$ follows from the results in~\cite{FIM2}, \cite{FIM3}, \cite{FIM4}, and~\cite{Thrall} and we are left to compute the kernel of $\pi$.  The Majorana scalar product $(\:,)_V$ on $V$ induces, in an obvious way, a symmetric bilinear form on the permutation module 
$$ \begin{array}{rccc}
f\colon &M_{\mathcal X}\times M_{\mathcal X}&\to &\R\\
& (u,v)&\mapsto&(u^\pi, v^\pi)_V
\end{array} $$ and an elementary and well known argument shows that the kernel of $\pi$ coincides with the radical of the form $f$. 
Note that, since $f$ is positive semidefinite, 
\begin{equation}
\label{brunch}
\rad(M_{\mathcal X})\cap M = \rad(M)
\end{equation}
for every submodule $M$ of $M_{\mathcal X}$. So, in order to determine which irreducible submodules of $M_{\mathcal X}$ are contained in $\rad(M_{\mathcal X})$, we need to compute the restrictions of $f$ to these submodules. Due to the dimensions involved, a direct computation of these restrictions is out of question, even for machine computing, so we need to use the machinery developed in~\cite{FIM4} (see also~\cite[Section 6]{FIM5}). 
Note that since the set $\mathcal X_x$ is invariant under the action of $\hat G$ by conjugation, the (permutation) $\R[G]$-module $M_x$ lifts to a  permutation $\R[\hat G]$-module.  Since, as we shall see in Section~\ref{Scalar}, the form $f$ is invariant under the action of $\hat G$,  we may consider $M_x$ to be an $\R[\hat G]$-module and use the representation theory of the symmetric groups. We next focus our investigation to the module $M_x$. Let 
\begin{equation}\label{dec}
M_{x}=\bigoplus_{i=1}^{l}\left (\bigoplus_{h=1}^{n_i} M_{x,h}^{\lambda_i}\right ),
\end{equation}
be a decomposition of $M_x$ into irreducible submodules, where, for $i\in \{1, \ldots , l\}$, $\lambda_i$ is a partition of $n$  and $M_{x,h}^{\lambda_i}$  is an $\R[\hat G]$-submodule isomorphic to the Specht module $S^{\lambda_i}$ (see~\cite{J}), with $\lambda_i\neq \lambda_j$ for $i\neq j$, and  
$$\bigoplus_{h=1}^{n_i} M_{x,h}^{\lambda_i}$$
is the homogeneous component relative to the Specht module $S^{\lambda_i}$. Let 
\begin{equation}
\label{pipi}
\pi_{x, h}^{\lambda_i}\colon M_x \to M_{x,h}^{\lambda_i}
\end{equation} 
be  the projection of $M_x$ onto the submodule $M_{x,h}^{\lambda_i}$ with respect to the decomposition~(\ref{dec}). 
Let $$\Gamma_1, \ldots,\Gamma_{r_x}$$ be the orbitals of $\hat G$ on $\mathcal X_x$ (i.e. the orbits of $\hat G$ on the set ${\mathcal X_x}\times {\mathcal X_x}$  with respect to the action defined, for every $g\in \hat G$ and  $u,v\in \mathcal X_x$, by $(u,v)^g=(u^g, v^g)$) and, for $j\in\{1,\ldots, r_x\}$,  let  
$$
\Delta_j(u):=\{v\in \mathcal X_x\:|\:(u,v)\in \Gamma_j\}.
$$

Let  $P(x)$ be a generalised first eigenmatrix (see \cite[ \S 2]{FIM2}) associated to $M_{x}$ with respect to the decomposition~(\ref{dec}) and  denote by 
$$
P(x)^{\lambda_i}_j
$$ 
the $(i,j)$-entry of $P(x)$.  Note that  
$P(x)^{\lambda_i}_j$ is a matrix of size $n_i\times n_i$ whose $(k,l)$-entry will be denoted by $$(P(x)^{\lambda_i}_j)_{kl}.$$ 
\begin{lemma}
\label{BannIto}
\cite[Equation~(13) and Lemma~1(i)]{FIM2}
With the above notation, we have  
\begin{equation}
\label{progi}
u^{\pi_{x,h}^{\lambda_i}}=\sum_{j=1}^{r_x} \frac{\dim(M_{x,h}^{\lambda_i})}{|\Gamma_j|}(P(x)^{\lambda_i}_j)_{hh}\left ( \sum_{v\in \Delta_j(u)} v\right ).
\end{equation}
\end{lemma}

Since $(\:,)_V$ is $G$-invariant, the values of $f$ are constant on the orbitals of $G$ on  ${\mathcal X}$  Having the shape of $\phi$ at hand, these values are given by the Norton-Sakuma Theorem for pairs of Majorana involutions (see the fourth column of Table~\ref{table1}). The remaining cases are worked out in Section~\ref{Scalar}, in particular one sees that these values are constant also on the $\hat G$-orbitals.
For $j\in \{1, \ldots , r_x\}$, denote by
$$
\gamma_{x, j}
$$ 
the value of the form $f$ on any pair $(y,z)$ belonging to $\Gamma_j$ and set  
\begin{equation}\label{eqzapata}
f_{x,h}^\lambda:=\sum_{ j=1}^{r_x} \gamma_{x,j}\cdot (P(x)^{\lambda}_j)_{hh}.
\end{equation}
 Since the modules $M_{x,h}^{\lambda_i}$ are absolutely irreducible, they admit, up to a scalar multiplication, a unique nondegenerate scalar product (see~\cite[p.534]{La}). In particular if 
$$
\kappa\colon M_{\mathcal X}\times M_{\mathcal X}\to \R
$$
is the unique scalar product on $M_{\mathcal X}$ for which $\mathcal X$ is an orthonormal basis, the relation between the form $f$ and the form $\kappa$ on the irreducible submodules of $M_{\mathcal X}$ is given by the following lemma.
 
\begin{lemma}\label{Zapata} \cite[Lemma~4.14]{FIM4}
Let $x \in \{b,t,r,s\}$, $i\in \{1,\ldots,l\}$, $h\in \{1,\ldots, n_i\}$,  and $\lambda:=\lambda_i$. Then for every $v\in M_{x,h}^\lambda$ we have
$$
f(v,v)=f_{x,h}^{\lambda} \kappa(v,v).
$$
In particular 
$M_{x,h}^\lambda$ is contained in $\rad(M_{\mathcal X})$ if and only if $f_{x,h}^\lambda=0$.
\end{lemma}
\begin{cor}
\label{nonnegative}
With the above notation, the values $f_{x,h}^\lambda$ are nonnegative.
\end{cor}

\begin{proof}
This follows immediately from Lemma~\ref{Zapata}, since  the form $\kappa$ is positive definite and the form $f$ is positive semidefinite.
\end{proof}

Lemma~\ref{Zapata} allows us to determine $\rad(M_x)$ if the action of $\hat G$ on $\mathcal X_x$ is multiplicity-free, which is the case when $x=s$. For $x=b$ and $x=r$,  $\rad(M_x)$ has been determined in~\cite[Theorem 2]{FIM2} and~\cite[Theorem 1.2]{FIM4} respectively. 
Finally, for $x=t$,  we'll see in Section~\ref{3assi} 
 that  there exists an involutory  $\R[\hat G]$-automorphism $\beta$ of $M_t$, induced by a permutation of $\mathcal X_t$, such that 
$$M_t=C_{M_t}(\beta)\oplus [M_t, \beta]\mbox{ and }(C_{M_t}(\beta))^\pi\leq (M_b+M_s)^\pi,$$
which reduces us to computing $\rad([M_t, \beta])$. Since $\beta$ is an $\R[\hat G]$-automor\-phism,  for every element $u\in \mathcal X_t$, $\beta$ induces in the obvious way a permutation of the $\hat G_u$-orbits $\Delta_1(u), \ldots , \Delta_{r_t}(u)$, whence,  for each $j\in \{1, \ldots , r_t\}$, there is a unique  $j^\beta\in \{1, \ldots , r_t\}$, such that 
\begin{equation}
\label{deltaj}
\Delta_j(u)^\beta=\Delta_{j^\beta}(u).
\end{equation}
The computation of the rows of the generalised first eigenmatrix $P(t)$ needed to determine $\rad([M_t, \beta])$ can be simplified by the following result.
\begin{lemma}\label{dispari}
Let $\beta$ be as above, $u\in \mathcal X_t$, $i\in \{1, \ldots , l\}$, $h\in \{1, \ldots , n_i\}$, and suppose that $u^{\pi_{t,h}^{\lambda_i}}\neq 0$. Then
\begin{enumerate}
\item $M_{t,h}^{\lambda_i}\leq [M_t,\beta]$ $\Leftrightarrow$ for every $j\in \{1, \ldots , r_t\}$,  
$(P(t)_j^{\lambda_i})_{hh}=-(P(t)_{j^\beta}^{\lambda_i })_{hh}$.
\item $M_{t,h}^{\lambda_i}\leq C_{M_t}(\beta)$ $\Leftrightarrow$ for every $j\in \{1, \ldots , r_t\}$,  
$(P(t)_j^{\lambda_i})_{hh}=(P(t)_{j^\beta}^{\lambda_i })_{hh}$.
\end{enumerate}
\end{lemma}
\begin{proof}
Since $\beta$ acts as the identity on $C_{M_t}(\beta)$ and as the multiplication by $-1$ on $[M_t,\beta]$, by Equation~(\ref{progi}) it follows that $u^{\pi_{t,h}^{\lambda_i}}\in [M_t,\beta]$ if and only if,  for every $j\in \{1, \ldots , r_t\}$,  
$(P(t)_j^{\lambda_i})_{hh}=-(P(t)_{j^\beta}^{\lambda_i })_{hh}$,  while $u^{\pi_{t,h}^{\lambda_i}}\in C_{M_t}(\beta)$ if and only if,  for every $j\in \{1, \ldots , r_t\}$,  
$(P(t)_j^{\lambda_i})_{hh}=(P(t)_{j^\beta}^{\lambda_i })_{hh}$.
Since $u^{\pi_{t,h}^{\lambda_i}}\neq 0$, $C_{M_t}(\beta)$, $[M_t,\beta]$, and  $M_{t,h}^{\lambda_i}$ are $\R[\hat G]$-modules and $M_{t,h}^{\lambda_i}$ is irreducible, the claims follow.
\end{proof}

Once the quotients of $M_{b}$, $M_r$, $M_{s}$, and $[M_t,\beta]$ by their respective radicals have been determined,  we need to consider the case when isomorphic irreducible submodules are contained in two or more of those quotients (see Section~\ref{2assi} and Section~\ref{23assi}).
To deal with this situation we use the following result.
\begin{lemma}\cite[Corollary~2.4]{FIM4}\label{int}
Let $x,y\in \{b,r,s,t\}$, let $\lambda$ be a partition of $n$ and let $M_{x,h}^\lambda$ and $M_{y, k}^\lambda$  be irreducible submodules of $M_x$ and $M_y$ respectively, appearing in decomposition~(\ref{dec}), such that
\begin{enumerate}
\item  the multiplicities of $M_{x,h}^\lambda$ in $M_x$ and $M_{y, k}^\lambda$ in $M_y$ are not $0$,
\item $ \rad(M_x)\cap M_{x,h}^\lambda=\{0\}$, 
\item $ \rad(M_y)\cap M_{y,k}^\lambda=\{0\}$.
\end{enumerate}
Let $H$ be a subgroup of $\hat G$ such that $\dim(C_{S^\lambda}(H))=1$ and let $v_x\in M_{x,h}^\lambda$ and $v_y\in M_{y,k}^\lambda$ be two $H$-invariant non-trivial vectors. 
Then 
$$
\rad(M_{\mathcal X})\cap (M_{x,h}^\lambda\oplus M_{y,k}^\lambda))\cong S^{\lambda}
$$
if and only if 
$$
\det \left (\begin{array}{cc}
f(v_x, v_x) & f(v_x, v_y) \\
f(v_y, v_x) &  f(v_y, v_y)
\end{array}
\right )=0.
$$
\end{lemma} 
Finally, once $V^\circ$ has been cleared (as shown in Table~\ref{dect}), in Section~\ref{closure} we restrict to the case where $n=12$ and prove that, in this case, $V^\circ$ is closed under the algebra product.  
\medskip

For the rest of this paper we let $(G, W, \mathcal X_b, \phi_b, \psi_b)$ be the Majorana representation of $G$ induced by the restriction of a $2A$-Majorana representation of $A_{12}$ to $G=\langle \mathcal X_b \rangle$. Further, set 
$$V^{(2A)}:=\langle {\mathcal T}^{\psi} \rangle\:\: \mbox{ and }\:\: W^{(2A)}:=\langle {\mathcal X_b}^{\psi_b} \rangle
$$
and for $U\in \{V, W\}$ and $3\leq N\leq 5$, let $U^{(NA)}$ denote the linear span of the $NA$-axes contained in all the Norton-Sakuma algebras generated by pairs of Majorana axes in $U$. Finally, For $2\leq N\leq 5$, $m\in \N$, denote by $U^{(N^m)}$ the linear span of $NA$-axes corresponding to permutations of cycle type $N^m$. Note that $V^{(3^2)}=W^{(3^2)}$. $NA$-axes corresponding to permutations of cycle type $N^m$ will be called simply $N$-axes of type $N^m$.

%%%%%%%%%%%%%%%%%%%%%%%%%%%%

\section{The shape}\label{shape} 

In this section we keep the notation of Section~\ref{strategy}, in particular $n\in \{8,\ldots, 12\}$,  $G=A_n$,  and $(G, \mathcal T, V, \phi, \psi)$ is a Majorana representation of $G$. We  determine the shape of $\phi$ and prove that it depends only on $G$.

\begin{lemma}
\label{2^4}
$
\mathcal T \subseteq \mathcal X_b\cup \mathcal X_s.
$
\end{lemma}
\begin{proof} The cycle types of the involutions in $G$ are $2^2$, $2^4$, or $2^6$. 
Assume, by means of contradiction, that there are permutations of cycle type $2^4$ in $\mathcal T$. Since these are all conjugate in $G$ and $\mathcal T$ is invariant under conjugation, $\mathcal T$ contains all such involutions. Let  $H:=\langle t_1, t_2,t_3\rangle$, where  
$t_1:=(1,2)(3,4)(5,6)(7,8)$,  $t_2:=(1,3)(2,4)(5,8)(6,7)$,  and  $t_3:=(1,8)(2,7)(3,5)(4,6)$.
Then $H$ is elementary abelian of order $8$ and  every nontrivial element of $H$, being of cycle type $2^4$, is contained in $\mathcal T$, which is a contradiction to~\cite[Lemma 4.2]{Why} 
(note that, if $n\geq 9$, the result follows immediately, since the set of involutions of cycle type $2^4$ in $G$ is not a set of $6$-transpositions).  
\end{proof}

Set $r_1:=(1,2)(3,4)$, $s_1:=(1,2)(3,4)(5,6)(7,8)(9,10)(11,12)$, and denote the orbitals of G on $\mathcal X_b\cup \mathcal X_s$ as in the first column of Table~\ref{orbs}. Here the pairs of orbitals of $A_{12}$ that fuse under the natural action of $S_{12}$ are marked with $^\ast$ and $^{\ast\ast}$.

\begin{table}
$$
\begin{array}{|c|c|c|c|c|c|}
\hline
  \mbox{\sc orbitals}  & \mbox{\sc representatives}&\mbox{\sc cycle type of $r_1r_j$}&\mbox{\sc shape} \\
\hline
\Sigma_{1,b} &  (r_1,r_1)&   1 &1A\\
\Sigma_{2,b} &  (r_1,(1,3)(2,4))&   2^2 &2A\\
\Sigma_{3,b}  &  (r_1,(1,5)(3,4)) &3& 3A\\
\Sigma_{4,b}  &  (r_1, (1,5)(2,3)) &5 &5A\\
\Sigma_{5,b}  &  (r_1, (5,6)(3,4)) & 2^2&2A\\
\Sigma_{6,b}  &   (r_1,(5,6)(2,3)) &2\cdot 4& 4B\\
\Sigma_{7,b}  &  (r_1,(1,5)(2,6)) &2\cdot 4& 4B\\
\Sigma_{8,b}  &  (r_1, (1,6)(3,5)) &3^2 &3A\\
\Sigma_{9,b}  & (r_1, (1,5)(6,7)) &2^2\cdot 3 &6A\\
\Sigma_{10,b}  & (r_1, (5,6)(7,8)) &2^4&2B\\
\hline
  &  & \mbox{\sc cycle type of $s_1s_j$} &\\
\hline
\Sigma_{1,s}  &   (s_1,s_1)&   1& 1A \\
\Sigma_{2,s}  &   (s_1, (1,2)(3,4)(5,6)(7,8)(9,11)(10,12))&   2^2& 2A\\
\Sigma_{3,s}  &    (s_1, (1,2)(3,4)(5,7)(6,8)(9,11)(10,12))&   2^4 &2B\\
\Sigma_{4,s}  &    (s_1, (1,2)(3,4)(5,6)(7,9)(8,11)(10,12))&   3^2 &3A\\
\Sigma_{5,s}^\ast  &    (s_1, (1,3)(2,4)(5,7)(6,8)(9,11)(10,12))&   2^6 &2A\\
\Sigma_{5,s}^{\ast\ast}  &    (s_1, (1,3)(2,4)(5,7)(6,8)(9,12)(10,11))&   2^6& 2A\\
\Sigma_{6,s}  &   (s_1,  (1,2)(3,5)(4,6)(7,9)(8,11)(10,12))&   2^2\cdot 3^2 &6A\\
\Sigma_{7,s}  &  (s_1, (1,3)(2,5)(4,6)(7,9)(8,11)(10,12))&   3^4& 3A\\
\Sigma_{8,s}  &   (s_1,  (1,2)(3,4)(5,7)(6,9)(8,11)(10,12))&   4^2 &4A\\
\Sigma_{9,s}^\ast  &  (s_1, (1,3)(2,4)(5,7)(6,9)(8,11)(10,12))&   2^2\cdot 4^2 &4A\\
\Sigma_{9,s}^{\ast\ast}  &  (s_1, (1,3)(2,4)(5,7)(6,9)(8,12)(10,11))&   2^2\cdot 4^2 &4A\\
\Sigma_{10,s}  &    (s_1, (1,2)(3,5)(4,7)(6,9)(8,11)(10,12))&   5^2 &5A\\
\Sigma_{11,s}^\ast  &  (s_1, (1,3)(2,5)(4,7)(6,9)(8,11)(10,12))&   6^2&6A\\
\Sigma_{11,s}^{\ast\ast} &  (s_1, (1,3)(2,5)(4,7)(6,9)(8,12)(10,11))&   6^2&6A\\
\hline
  &  & \mbox{\sc cycle type of $s_1r_j$} &\\
\hline
\Sigma_{1,sb}  &   (s_1,r_1)&   2^4 & 2B\\
\Sigma_{2,sb}   &  (s_1,(1,3)(2,4))&   2^6 &2A\\
\Sigma_{3,sb}   &  (s_1, (2,3)(5,6)) &  2^3\cdot 4  &4B \\
\Sigma_{4,sb}   &  (s_1, (2,3)(4,5)) &   2^3 \cdot 6& 6A\\
\Sigma_{5,sb}   &  (s_1, (2,3)(6,7)) &  2^2\cdot4^2& 4A   \\
\hline
 &  & \mbox{\sc cycle type of $r_js_1$} &\\
\hline
\Sigma_{1,bs}  &   (r_1,s_1)&   2^4 & 2B\\
\Sigma_{2,bs}   &  ((1,3)(2,4), s_1)&   2^6 &2A\\
\Sigma_{3,bs}   &  ((2,3)(5,6), s_1) &  2^3\cdot 4  &4B \\
\Sigma_{4,bs}   &  ((2,3)(4,5), s_1) &   2^3 \cdot 6& 6A\\
\Sigma_{5,bs}   &  ((2,3)(6,7), s_1) &  2^2\cdot4^2& 4A   \\
\hline

\end{array}
 $$
 \caption{The $A_{12}$-orbits on $\mathcal T\times \mathcal T$ and their representatives for $r_1:=(1,2)(3,4)$, $s_1:=(1,2)(3,4)(5,6)(7,8)(9,10)(11,12)$, and  $\mathcal T= \mathcal X_b\cup \mathcal X_s$.}
 \label{orbs}
 \end{table}
 
\begin{lemma}~\cite[Theorem III]{Thrall} 
\label{la belle dame sans mercy} 
With the notation of  Section~\ref{strategy}, 
$$
\begin{array}{rcl}
M_s &= &  M_{s,1}^{(12)}\oplus   M_{s,1}^{(10,2)} \oplus M_{s,1}^{(8,4)}\oplus M_{s,1}^{(6,4,2)}\oplus M_{s,1}^{(4^2,2^2)}\oplus M_{s,1}^{(4,2^4)}\oplus \\
& & \\
& &M_{s,1}^{(6^2)}\oplus M_{s,1}^{(2^6)}\oplus  M_{s,1}^{(8,2^2)} \oplus M_{s,1}^{(4^3)}\oplus  M_{s,1}^{(6,2^3)}.
\end{array}
$$
\end{lemma}

\begin{lemma}
\label{hath thee in thrall}
For $i\in \{1, \ldots , 10\}\setminus \{5,9,11\}$, let $\Sigma_i:=\Sigma_{i,s}$  and, for $i \in\{5,9,11\}$, let  $\Sigma_i:=\Sigma_{i,s}^\ast\cup \Sigma_{i,s}^{\ast\ast}$. Then $\Sigma_1,\ldots,\Sigma_{11}$ are the orbitals of $S_{12}$ on its action on $\mathcal X_s$. The first eigenmatrix $P(s)$ relative to this action is given in the last eleven columns of Table~\ref{FE}. The entries in the second column of Table~\ref{FE} are the dimensions of the irreducible $\R[S_{12}]$-modules $S^\lambda$.  
\end{lemma}

\begin{proof}
This can be computed using standard methods of algebraic combinatorics (see~\cite[Chapter II, \S 2.1]{BI}) as in \cite{FIM1}, or using \cite[Lemma 3]{FIM2}.
\end{proof}

\begin{table}
{\Small
$$
\begin{array}{|c|c||c|c|c|c|c|c|c|c|c|c|c|}
 \hline
\lambda &  &\Sigma_1 & \Sigma_2  & \Sigma_3  & \Sigma_4  & \Sigma_5 & \Sigma_6  &\Sigma_7  & \Sigma_8  & \Sigma_9  &\Sigma_{10}  & \Sigma_{11} \\
\hline
\hline 
(12)&1 & 1 & 30 & 180 & 160 & 120 & 960 & 640 & 720 & 1440 & 2304 & 3840 \\
\hline
(10,2)&54& 1 & 19 & 48 & 72 & -12 & 80 & -64 & 192 & -144 & 192 & -384\\
\hline
(8,4)& 275 & 1 & 12 & 27 & 16 & 30 & 24 & -8 & -18 & 108 & -144 & -48 \\
\hline
(6,4,2)& 2673 & 1 & 4 & 3 & -8 & -2 & 0 & -24 & -18 & -4 & 32 & 16 \\
\hline
(4^2,2^2)& 2640 & 1 & -3 & 3 & -8 & -9 & 0 & 4 & 24 & 24 & -24 & -12 \\
\hline
(4,2^4)& 1485 & 1& -8 & 3 & 12 & 6 & 20 & -16 & -6 & -36 & -24 & 48 \\
\hline
(6^2)& 132 & 1 & -15 & 45 & 40 & -15 & -120 & 40 & -90 & 90 & 144 & -120\\
\hline
(2^6)& 132 & 1 & 9 &   33& -8 & -27 & 120 & 136 & -78 & -114 & -48 & -24\\
\hline
(8,2^2)& 616& 1 & 9 & -12 & 22 & -12 & -60 & 16 & 12 & -24 & -48 & 96 \\
\hline
(4^3)& 462 & 1 &  0&15 & -20 & 30 & -60 & 40 & 30 & -60 & 24 & 0 \\
\hline
(6,2^3)& 1925 & 1 &0&  -21 & 4 & 6 & 12 & 16 & -6 & 12 & 24 & -48 \\
\hline
\end{array}
$$
}
\caption{First eigenmatrix for the action of $S_{12}$ on its involutions of cycle type $2^6$} \label{FE} 
\end{table}

\begin{proposition}
\label{shape2^6}
Let $n=12$ and assume that $\mathcal X_s\subseteq \mathcal T$. Then 
\begin{enumerate}
\item $\phi$ is a $2A$-Majorana representation;
\item the shape of $\phi$ is the one given in the last column of Table~\ref{orbs};
\item if $\Sigma^\ast$ and $\Sigma^{\ast\ast}$ are orbitals of $G$ on $\mathcal T$ that fuse under $\hat G$ then $sh(\Sigma^\ast)=sh(\Sigma^{\ast\ast})$;
\item the restriction of the form $f$ on the module $M_b+M_s$ is $\hat G$-invariant.
\end{enumerate}
\end{proposition}

\begin{proof}
Let $\Sigma$ be an orbital of $G$ on $\mathcal T$, let $(u,y)\in \Sigma$, and let $[uy]$ be the cycle type of $uy$. If $|uy|\in \{1,5,6\}$,  then, by the Norton-Sakuma Theorem, $sh(\Sigma)=|uy|A$. By the hypothesis and Axiom M4, $sh(\Sigma)=2A$, if $[uy]=2^6$,  while $sh(\Sigma)=2B$,  if $[uy)]=2^4$.  Assume $|uy|=4$ and $[uy]\in \{4^2, 2^2\cdot 4^2\}$. Then, by Lemma~(\ref{2^4}), $(uy)^2\not \in \mathcal T$, whence $sh(\Sigma)=4A$ by Axiom M6.  This gives the shape of $\Sigma$ and, consequently, the scalar products of the axes associated to $u$ and $y$ for all cases such that $u,y\in \mathcal X_s$, except for $|uy|\in\{2,3\}$ and $[uy]\not\in\{2^4, 2^6\}$. In the latter case, we have  
$\Sigma\in \{ \Sigma_{2,s}, \Sigma_{4,s}, \Sigma_{7,s}\}$ and, by the Norton-Sakuma Theorem, the possibilities for $sh(\Sigma)$ are $2A$ or $2B$, if $|uy|=2$, and $3A$ or $3C$ if $|uy|=3$. 
We shall eliminate the possibilities $2B$ and $3C$, by showing that in these cases the form $f$ would not be positive semidefinite, since some of the values $f_{s,1}^\lambda$ would be negative.
Since the orbitals $\Sigma_{2,s}$,  $\Sigma_{4,s}$, and $\Sigma_{7,s}$ are invariant under the action of  $\hat G$, it follows that the restriction of the form $f$ on the module $M_s$ is $\hat G$-invariant so we can use the structure of  $M_s$ as an $\R[\hat G]$-module to compute the values $f_{s,1}^\lambda$.
Let 
$$v_{(6^2)}:=(1, -15, 45, 40, -15, -120, 40, -90, 90, 144, -120)$$
and
$$v_{(2^6)}:=(1, 9, 33, -8, -27, 120, 136, -78, -114, -48, 96)$$
be the rows of the first eigenmatrix associated to the action of $\hat G$ on the set $\mathcal X_s$ corresponding to the partitions $(6^2)$ and $(2^6)$ (see Table~\ref{FE}).
For $W\in \{2A,2B, 3A, 3C\}$, let $\gamma(W)$ be the Majorana inner product of the two generating axes of a Norton-Sakuma algebra of type $W$, so, by Table~\ref{table1}, 
$$\gamma(2A)=1/2^8, \,\,\,\gamma(2B)=0, \,\,\,\gamma(3A)=13/2^8, \mbox{ and } \gamma(3C)=1/2^6$$
and, for $X\in\{2A, 2B\}$ and $Y,Z\in\{3A, 3C\}$, let
$$\gamma(X,Y,Z):=(1, \gamma(X), 0, \gamma(Y), 1/2^3, 5/2^8, \gamma(Z), 1/2^5, 1/2^5, 3/2^7, 5/2^8).$$
Then, by Equation~(\ref{eqzapata}), for $\lambda\in \{(6^2), (2^6)\}$,  
$f_{s,1}^{\lambda}$ is given by the rows by columns multiplication of $\gamma(X,Y,Z)$ by the transpose of $v_{\lambda}$.
A straightforward computation shows that, in the case $(X,Y,Z)=(2A,3A,3A)$, $f_{s,1}^{(2^6)}=f_{s,1}^{(6^2)}=0$, while in all the other cases  either $f_{s,1}^{(6^2)}$ or  $f_{s,1}^{(6^2)}$ is negative, against Corollary~\ref{nonnegative}. Thus $(sh(\Sigma_{2,s}), sh(\Sigma_{4,s}), sh(\Sigma_{7,s}))=(2A, 3A, 3A)$.
In particular, $sh(\Sigma_{2,s})=2A$, so Axiom M4 implies $\mathcal X_b\subseteq \mathcal T$, that is $\phi$ is a $2A$-Majorana representation. By Axiom M4, we get $sh(\Sigma_{2,b})=2A$ and, by Axiom M6, $sh(\Sigma)=4B$, if $[uy]\in \{2\cdot 4, 2^3\cdot 4\}$. Since every element of cycle type $3^2$ is the product of two elements in $\mathcal X_s$, Axiom M5 implies that $sh(\Sigma_{4,b})=3A$. Finally, if  $u,y\in \mathcal X_b$ and $[uy]=3$, then the subgroup they generate fixes more than $4$ points, so  $sh(\Sigma)=3A$, by~\cite[Corollary 1.2]{A67}. This proves (2). Assertions (3) and (4) then follow immediately.
\end{proof}

\begin{cor}\label{equivalence}
Let $n=12$. The following assertions are equivalent.
\begin{enumerate}
\item $\phi$ is a $2A$-Majorana representation.
\item $\mathcal X_{s} \subseteq \mathcal T$;
\item  there are involutions $t_1, t_2\in \mathcal T$ such that $\langle \langle a_{t_1}, a_{t_2}\rangle \rangle$ has type $4A$;
\end{enumerate}
\end{cor}
\begin{proof}
Obviously $(1)$ implies $(2)$. Moreover,  by Lemma~(\ref{la belle dame sans mercy}) it follows that $(2)$ implies $(3)$. Assume $(3)$ holds. Looking at the possible cycle types of the product $t_1t_2$, one can see that the algebra $\langle \langle a_{t_1}, a_{t_2}\rangle \rangle$ has type $4A$ if and only if $\{t_1, t_2\}\cap \mathcal X_s\neq \emptyset$. This gives $\mathcal X_s\subseteq \mathcal T$ and, by Proposition~\ref{shape2^6}, $(1)$ holds.  
\end{proof}

%%%%%%%%%%%%%%%%%%%%%%%%%%%%%%%%%%%%%%%%%%%%%%

\section{The inner products}\label{Scalar}
We keep the notation of Section~\ref{strategy}. In this section we find the values of the inner product between axes and  $3$-axes needed to determine the radical of the restriction of the form $f$ to the $\R[G]$-submodule $M_b+M_s+M_t$ of $M_{\mathcal X}$.
Note that the values of the inner products between Majorana axes are given by the Norton-Sakuma Theorem and can be found in Table~\ref{table1}. 
The results involving $3$-axes of type $3^2$ have been considered by Chien  in~\cite{chien}. However, Chien's results partially depend on an earlier work of S. Norton, who computed the inner products between axes and $3$-axes within the Griess' algebra (see~\cite{Norton}). Therefore, to make Chien's results independent from the Monster, we recalculated the inner products involving $3$-axes of type $3^2$ relying only on the Majorana axiomatics, which we did either by hand or using the GAP package {\it MajoranaAlgebras}~\cite{MM}. 

We begin by considering the Majorana inner products $(a_r, u_{c})_V$ between Majorana axes $a_r$, with  $r\in \mathcal X_{b}\cup \mathcal X_{s}$, and $3$-axes $u_{c}$, with  $\langle c\rangle \in \mathcal X_{t}$. Assume first $r\in \mathcal X_{s}$ and  $\langle c\rangle \in \mathcal X_{t}$. Let, as in Section~\ref{shape},   
$$r_1:=(1,2)(3,4).$$
For $n=12$, let 
$$s_1:=(1,2)(3,4)(5,6)(7,8)(9,10)(11,12),$$
let 
 $$\{(s_1, \langle c_i \rangle)| i\in \{1,\ldots, 11\}\}$$ be a set of representatives for the orbitals of $G$ on $\mathcal X_s$,  and let 
 \begin{equation}\label{di}
 \{(r_1, \langle d_i\rangle )\:|\:i\in \{1, \ldots , 13\}\}
 \end{equation}
  be  a set of representatives for the orbitals of $G$ on $\mathcal X_b$.
{\tiny

\begin{table}
$$
\begin{array}{|c|c|c|c|c|}
\hline
i & c_i & [s_1c_i]  & \langle s_1, c_i\rangle & (a_{s_1}, u_{c_i})_V\\
\hline
&&&&\\
1 & (7,8,9)(10,11,12) & 2^3\cdot 4&S_4 & \frac{1}{36}\\
&&&&\\
\hline
&&&&\\
2 & (7,9,11)(8,10,12)& 2^3\cdot6 &C_6 &0\\
&&&&\\
\hline 
&&&&\\
3 & (7,9,12)(8,11,10)& 2^6& S_3&\frac{1}{4}\\
&&&&\\
\hline 
&&&&\\
4 & (6,7,8)(10,11,12)&2^2\cdot3^2 &2\times A_4& \frac{2}{45}\\
&&&&\\
\hline
&&&&\\
5 & (6,7,9)(10,11,12)& 2^2\cdot7&2\times L_2(7) &\frac{11}{360} \\
&&&&\\
\hline
&&&&\\
6 & (6,7,9)(8,10,11)& 2^2\cdot4^2& S_4 & \frac{13}{180}\\
&&&&\\
\hline
&&&&\\
7 & (6,9,7)(8,10,11)&2^3\cdot 6 &2\times A_4&\frac{2}{45}\\
&&&&\\
\hline
&&&&\\
8 & (3,6,7)(10,11,12)&2\cdot 3 \cdot 6 & S_3\times A_4&\frac{17}{360}\\
&&&&\\
\hline
&&&&\\
9 & (3,6,7)(8,11,9)&2\cdot 10 & 2\times A_5&\frac{1}{30}\\
&&&&\\
\hline
&&&&\\
10 & (2,3,6)(8,9,11)&6^2 &3 \times S_3& \frac{1}{20}\\
&&&&\\
\hline
&&&&\\
11 & (2,3,6)(8,11,10)&6^2& 3\times S_3&\frac{1}{20} \\
&&&&\\
\hline
\end{array}
$$

\caption{The scalar products between axes of type $2^6$ and $3$-axes of type $3^2$}\label{tau3}
\end{table}
}
\begin{lemma}\label{innerproduct}
Let $n=12$ and $\phi$ be a $2A$-Majorana representation of $G$. Then, the (relevant) Majorana  inner products $(a_z, u_c)_V$ such that $z\in \mathcal X_s$ and $\langle c \rangle \in \mathcal X_t$ are given in Table~\ref{tau3}. \end{lemma}
\begin{proof}
By Proposition~\ref{shape2^6}, the shape of $\phi$ is the one given in Table~\ref{orbs}.
For $i\in \{1,6\}$,  $\langle s_1, c_i\rangle\cong S_4$, so $$(s_1,c_i)_V=1/36$$ by~\cite[Table~9 and Table~7]{IPSS}. 
For $i=2$ we have $\langle s_1, c_2\rangle \cong C_6$ and the algebra $\langle \langle a_{s_1}, u_{c_2}\rangle \rangle$ is contained in a Norton-Sakuma algebra of type $6A$. So, by the Norton-Sakuma Theorem, $$(a_{s_1}, u_{c_2})_V=0.$$ Similarly, for $i=3$,  $\langle \langle a_{s_1}, u_{c_3}\rangle \rangle$ is contained into a Norton-Sakuma algebra of type $3A$, whence 
$$(a_{s_1}, u_{c_3})_V=1/4.$$
Now assume $i\in\{4,\ldots, 11\}$ and decompose  $c_i$ as the product of two permutations $g$ and $h$ of cycle type  $2^2$.
Then, by the Norton-Sakuma Theorem, 
$$
u_{c_i}=\frac{2^{11}}{27\cdot 5} \left \{ \frac{1}{32}\left [ 2a_g+2a_h+a_{ghg}\right ]- a_g\cdot a_h \right \}
$$
and, by the associativity of the Majorana inner product, 
\begin{equation}
\label{nonandràbene}
(a_{s_1}, u_{c_i})_V=\frac{2^{11}}{27\cdot 5} \left \{ \frac{1}{32}\left [ 2(a_{s_1},a_g)_V+2(a_{s_1}, a_h)_V+(a_{s_1},a_{ghg})_V\right ]- (a_{s_1}\cdot  a_h, a_g)_V \right \}.
\end{equation}
The first three scalar products in the second member of Equation~(\ref{nonandràbene}) are scalar product of axes, so they can be detected by Proposition~\ref{shape2^6} and the Norton-Sakuma Theorem. The last one can be computed by choosing $s$ so that $a_{s_1}\cdot a_h$ is a linear combination either of axes (which is the case when the Norton-Sakuma subalgebra generated by  $a_{s_1}$ and  $a_h$ has type $2A$, $2B$, or $4B$), or of axes together with  a $3$-axis whose inner product with $a_g$ is already computed. 
Thus, for $i=4$,  choose 
$$
g:=(6,7)(10,11) \mbox{ and } h:=(7,8)(11,12).
$$
Then,  
since $[s_1h]=2^4$, by Proposition~\ref{shape2^6}, $a_{s_1}$ and $a_h$ generate a  subalgebra of type $2B$, whence  $(a_{s_1},a_h)_V=0$. Similarly  $(a_{s_1},a_g)_V=(a_{s_1},a_{ghg})_V=1/32$. For the last scalar product, again by the Norton-Sakuma Theorem, we have $a_{s_1}\cdot a_h=0$ giving  $(a_{s_1}\cdot a_h, a_g)_V=0$. Substituting these values in Equation~\ref{nonandràbene} we get 
$$(a_{s_1}, u_{c_4})_V=\frac{2}{45}.$$ For $i=7$ the same argument with $g=(6,7)(10,11)$ and $h=(7,9)(8,10)$ gives the result.

 In order to compute the scalar product  $(a_{s_1}\cdot a_h, a_g)_V$ for $i=5$, choose 
$$
g:=(6,7)(10,11) \mbox{ and } h:=(7,9)(11,12).
$$
Then $a_{s_1}$ and  $a_h$ generate a subalgebra of type $4B$, so, by the Norton-Sakuma Theorem, 
$$
a_{s_1}\cdot a_h=\frac{1}{64}\left [a_{s_1}+a_h-a_{{s_1} h{s_1}}-a_{h{s_1} h}+a_{{s_1} h{s_1} s}\right ].
$$
Therefore 
\begin{eqnarray*}
(a_{s_1}\cdot a_h, a_g)_V&=&-\frac{1}{64} [ (a_{s_1},a_g)_V+(a_h, a_g)_V\\
&-&(a_{s_1 hs_1 }, a_g)_V-(a_{h{s_1} h}, a_g)_V+(a_{{s_1} h{s_1} h}, a_g)_V  
 ].
\end{eqnarray*}
If $i=8$, the inner product $(a_{s_1}, u_{c_8})_V$ can be computed in a similar way, with $g:=(3,6)(10,11)$ and $h:=(6,7)(11,12)$ since, with this choice, $\langle \langle a_{s_1}, a_h\rangle \rangle$ is a Norton-Sakuma algebra of type $4B$.

Assume $i=9$. Set 
$$
g:=(3,6)(9,11)   \mbox{ and }  h:=(6,7)(8,9).
$$
Then the Norton-Sakuma algebra generated by  $a_{s_1}$ and $a_h$ has type $6A$ and $({s_1} h)^2=(5,9,8)(6,7,10)$. Since $(a_{s_1}, u_{({s_1} h)^2})$ is $A_{12}$-conjugate to $(a_{s_1}, u_{c_6})$, by the previous case we have
$(a_{s_1}, u_{({s_1} h)^2})_V=13/180$. It follows that
\begin{eqnarray*}
(a_{s_1}, a_g\cdot a_h)_V&=&(a_g , a_{s_1} \cdot a_h)_V \\
&= &\frac{1}{64} \left [ (a_g,a_{s_1})_V+(a_g, a_h)_V-(a_g, a_{{s_1} h{s_1}})_V-(a_g, a_{h{s_1} h})_V\right.\\
&&\left . -(a_g, a_{{s_1}h{s_1} h{s_1} })_V-(a_g, a_{h{s_1} h{s_1} h})_V+(a_g, a_{{s_1} h{s_1} h{s_1} h} )_V \right ]\\
& &+\frac{45}{2^{11}}(a_g, u_{({s_1} h)^2})_V=\frac{13}{2^{13}}
 \end{eqnarray*}
and $(a_{s_1} , u_{c_9})_V= 1/30$.

Finally, the inner products $(a_{s_1}, u_{c_i})_V$ for $i\in \{10, 11\}$ have been computed, within the Majorana axiomatics, by Chien~\cite[p.86]{chien}.
\end{proof}
{\tiny
\begin{table}
$$
\begin{array}{|c|c|c|c|c|}
\hline
&&&&\\
i & d_i &  [r_1d_i] &\langle r_1, d_i \rangle & (a_{r_1}, u_{d_i})_V\\
&&&&\\
\hline
&&&&\\
1 &  (6,7,8)(9,10,11) &2^2\cdot3^2 & 3\times 3 & 0\\
&&&&\\
\hline
&&&&\\
2& (4,5,6)(9,10,11) &2\cdot3\cdot4  & 3\times S_4& \frac{1}{36}\\
&&&&\\
\hline
&&&&\\
3 & (3,4,5)(9,10,11) &2^2\cdot3 &3\times S_3& \frac{1}{20}\\
&&&&\\
\hline
&&&&\\
4 & (3,5,7)(4,6,8) & 2\cdot 6&2\times A_4&\frac{1}{45}\\
&&&&\\
\hline 
&&&&\\
5 & (2,3,4)(9,10,11)& 3^2&A_4& \frac{1}{9}\\
&&&&\\
\hline
&&&&\\
6 & (2,3,5)(9,10,11)&3\cdot5 & 3\times A_5& \frac{11}{360}\\
&&&&\\
\hline 
&&&&\\
7 & (2,3,5)(4,6,7) &7 &L_3(2)& \frac{1}{24}\\
&&&&\\
\hline 
&&&&\\
8& (2,5,7)(4,6,8)&4^2 &S_4& \frac{13}{180}\\
&&&&\\
\hline 
&&&&\\
9& (2,5,7)(3,4,6)&2\cdot 4& S_4& \frac{1}{36}\\
&&&&\\
\hline 
&&&&\\
10&  (1,2,3)(4,5,6)&5	 &A_5&\frac{1}{18}\\
&&&&\\
\hline
&&&&\\
11 &  (1,2,5)(3,4,6)&2^2 &S_3&\frac{1}{4}\\
&&&&\\
\hline 
&&&&\\
12&  (1,3,5)(2,4,6)& 3^2&A_4& \frac{1}{9}\\
&&&&\\
\hline
&&&&\\
13& (1,3,6)(2,5,4)&2\cdot4 &S_4& \frac{1}{36}\\
&&&&\\
\hline
\end{array}
$$
\caption{The scalar products between axes of type $2^2$ and $3$-axes of type $3^2$}\label{t13}
\end{table}
}

\begin{lemma}\label{innerproductbi}
Let $\phi$ be a $2A$-Majorana representation of $A_{n}$. Then, the (relevant) Majorana  inner products $(a_z, u_c)_V$ such that $z\in \mathcal X_b$ and $\langle c \rangle \in \mathcal X_t$ are given in Table~\ref{t13}. 
\end{lemma}
\begin{proof}
By Proposition~\ref{shape2^6}, the shape of $\phi$ is the one given in Table~\ref{orbs}. When $i=1$, $r_1$ and $d_1$ are disjoint and hence their inner product is $0$ (see case $i=2$ in Lemma~\ref{innerproduct}). The inner products $(a_{r_1}, u_{d_i})_V$ for $i\in \{5, 7, 9, 10, 12, 13\}$ have been computed in~\cite{A67}, since in those cases $r_1$ and $d_i$ are contained in a stabiliser of $n-5$ points. When $i=11$, $\langle r_1 , d_{11}\rangle\cong S_3$ and the result follows from Norton-Sakuma theorem. In the remaining five cases one proceeds as in the proof of Lemma~\ref{innerproduct} and get the required values.
\end{proof}

\begin{table}
{\Small
$$
\begin{array}{|c|c|c|c|c|c|c|}
\hline
i& e_i & e_1e_i & e_1e_i^{-1}& \langle e_1, e_i\rangle & (u_{e_1}, u_{e_i})_V \\
\hline
1&(1,2,3)(4,5,6)& 3A & 1 & C_3& \frac{8}{5}\\
\hline
2& (1,3,2)(4,5,6)& 3A & 3A & C_3\times C_3& 0\\
\hline
3&   (1,2,4)(3,5,6) & 3A & 2A & A_4 & \frac{136}{405}\\
\hline
4& (1,4,2)(3,5,6) & 5A & 5A & A_5 & \frac{16}{405} \\
\hline
5& (2,3,4)(5,6,7) & 6 & 5 & A_7 & \frac{8}{81} \\
\hline
6& (2,4,3)(5,6,7) & 4B & 4B & L_3 (2) & \frac{32}{405} \\
\hline
7& (2,3,7)(4,5,6) & 6 &3A&C_3 \times A_4 &\frac{64}{405}\\
\hline
8&(2,7,3)(4,5,6) & 3A & 2A & A_4 & \frac{136}{405}\\
\hline
9& (2,4,6)(3,5,7) & 4B & 7 & L_3 (2) & \frac{4}{81}\\
\hline
10&(2,4,5)(3,6,7) & 7 & 3 & F_{21} & \frac{4}{27}\\
\hline
\hline
11& (2,3,7)(5,6,8) & 2B & 3A & A_4 & \frac {56}{27\cdot 25} \\
\hline
12&(2,7,3)(5,6,8) & 6 & 6 & A_4\times A_4 & \frac {2^7}{3^4\cdot 5^2}\\
\hline
13& (3,4,5)(6,7,8) & 4A & 7A & L_3 (2) & \frac {4\cdot 19}{27 \cdot 25}\\
\hline
14& (3,4,7)(5,6,8) & 6C & 7A & (2^3:7):3 & \frac {4}{75}\\
\hline
15 &(3,7,8)(4,5,6) & 15A & 5A & 3\times A_5 & \frac {2^2 \cdot 47}{3^4 \cdot 5^2}\\
\hline
16& (2,3,4)(6,7,8) & 6C & 7A & (2^3:7):3 & \frac {4}{75}\\
\hline
17 & (2,4,7)(3,5,8) & 4A & 4A & L_3 (2) & \frac {2^5}{9\cdot 25}\\
\hline
18& (2,4,7)(3,8,5)& 15A & 15A & A_8 & \frac {88}{3^4 \cdot 5^2}\\
\hline
19 & (2,4,7)(3,6,8) & 15A & 15A & A_8 & \frac {88}{3^4 \cdot 5^2}\\
\hline
20 & (2,4,7)(3,8,6) & 4A & 6C & SL_2 (3) & \frac {56}{27 \cdot 25}                  \\
\hline
\hline
21 & (3,4,7)(6,8,9) & 9A & 9A& L_2 (8):3 & \frac {16}{9\cdot 25} \\
\hline
22 & (4,5,7)(6,8,9) & 12C & 15A & 3\times A_6 & \frac {4\cdot 31}{3^4 \cdot 5^2}\\
\hline
23 & (3,7,9)(5,6,8) & 10A & 15A & A _4 \times A_5 & \frac {4\cdot 31}{3^4\cdot 5^2}\\
\hline
24 & (3,4,5)(7,8,9) & 12C &15A & 3\times A_6& \frac {4\cdot 31}{3^4 \cdot 5^2}\\ 
\hline
25 & (4,5,6)(7,8,9) & 3B & 3A & 3\times 3 & \frac {4}{25}\\
\hline
\hline
26 & (3,7,9)(6,8,10) & 5A &5A & A_5  & \frac {2^4 \cdot 13}{3^4 \cdot 5^2}\\
\hline 
27 & (3,4,7)(8,9,10) & 21A & 21A & 3\times A_7 & \frac {8\cdot 11}{3^4 \cdot 5^2}\\
\hline
28 & (5,7,9)(6,8,10) & 21A & 21A & 3\times A_7 & \frac {8\cdot 11}{3^4 \cdot 5^2} \\
\hline 
29 & (5,6,7)(8,9,10) & 6A & 3B & 3 \times A_4 & \frac {2^4\cdot 17}{3^4\cdot 5^2} \\
\hline
\hline
30 & (6,7,8)(9,10,11) & 15A & 15A & 3^2 \times A_5 & \frac {2^4}{3^4\cdot 5} \\
\hline
\hline
31 & (7,8,9)(10,11,12) & 3A & 3A & 3 \times 3 & 0\\
\hline
\end{array}
$$
}
\caption{The representatives for the orbitals of $\hat G$ on $\mathcal X_t$ and the Majorana form, for $7\leq n\leq 12$}\label{Orbitals12}
\end{table}
The inner products between two $3$-axes have been computed by Chien in~\cite{chien}. Unfortunately in Chien's thesis it is not clear which ones have been computed relying only on the Majorana axiomatics and which ones depend on the embedding of $A_{12}$ in the Monster, so we have checked all the products needed for this paper.  The group $A_{12}$ has $32$ orbits on the set  $\mathcal X_t\times \mathcal X_t$,  which we denote by $\Omega_i $, $i\in \{1, \ldots , 32\}$. All but $\Omega_{30}$ and $\Omega_{32}$ are also $S_{12}$-orbits, while $\Omega_{30}$ and $\Omega_{32}$ merge into a unique $S_{12}$-orbit.
A set of representatives $(\langle e_1\rangle , \langle e_i\rangle )$ for the first $31$ orbits is given in the second column of Table~\ref{Orbitals12}.  Double horizontal lines detect the change of the parameter $n$: $\{(\langle e_1\rangle , \langle e_i\rangle )\:|\: i\in \{1, \ldots , 10\}\}$ is a set of representatives for the $S_7$-orbits,  $\{(\langle e_1\rangle , \langle e_i\rangle )\:|\: i\in \{1, \ldots , 20\}\}$ is a set of representatives for the $S_8$-orbits, and so on. A representative for $\Omega_{32}$ is $(\langle e_1\rangle , \langle (6,7,8)(9,11,10)\rangle )$.

\begin{lemma}\label{prod3assi}
Let $\phi$ be a $2A$-Majorana representation of $A_{12}$. Then the (relevant) inner products between two $3$-axes of type $3^2$ are those given in the last column of Table~\ref{Orbitals12}.  Moreover the value of the inner product on $\Omega_{32}$ is the same with the that on $\Omega_{30}$.
\end{lemma}
\begin{proof}
Inner products $(u_{ e_1}, u_{e_i})_V$ for $i\in \{1, \ldots , 10\}$ have been computed in~\cite{A67}. 
For $i\in \{11, 12, 17\}$, these products can be computed using~\cite[Lemma 4.6]{chien}.
For $i\in \{13, 19, 24\}$, these products can be computed using~\cite[Lemma 4.7]{chien}.
For $i\in \{14, 15, 16, 20,23\}$, these products can be computed using~\cite[Lemma 4.8]{chien}.
Assume  $i=21$. Let
$$f_1:=(1,2)(5,6),\:\:f_2:=(1,3)(4,5),\:\:g_1:=(3,4)(8,9),\mbox{ and } g_2:=(3,7)(6,8),$$
so that $$e_1=f_1 f_2 \mbox{ and } e_{21}=g_1 g_2.$$
By the Norton-Sakuma Theorem and Lemma~\ref{innerproductbi}, the inner product $(u_{ e_1}, u_{e_{21}})_V$ can be computed  once the 
value of 
$$(a_{f_1}\cdot a_{f_2}, a_{g_1}\cdot a_{g_2})_V$$
is known. By the associativity of the scalar product, 
$$(a_{f_1}\cdot a_{f_2}, a_{g_1}\cdot a_{g_2})_V=((a_{f_1}\cdot a_{f_2}) \cdot a_{g_1}, a_{g_2})_V$$
and, since $f_1$ and $g_1$ commute, by the Norton-Sakuma Theorem $a_{f_1}$ and  $a_{g_1}$ generate a subalgebra of type $2B$, whence $a_{g_1}$ is a $0$-eigenvector for the adjoint action of  $a_{f_1}$. By~\cite[Lemma 1.10]{IPSS} and the associativity of the scalar product we have 
$$((a_{f_1}\cdot a_{f_2}) \cdot a_{g_1}, a_{g_2})_V= (a_{f_1}\cdot (a_{f_2} \cdot a_{g_1}), a_{g_2})_V= (a_{f_2} \cdot a_{g_1}, a_{f_1}\cdot a_{g_2})_V.$$
Since the subalgebra generated by $a_{f_2}$ and $a_{g_1})$ is of type $4B$, by Norton-Sakuma, $a_{f_2}\cdot a_{g_1}$ can be written as a linear combination of Majorana axes. On the other hand, the subalgebra generated by $a_{f_1}$ and $a_{g_2}$ has type $6A$, with    $u_{(f_1g_2)^2}$  is of type $3$. By~\cite[Corollary 3.2]{Alonso} the product $a_{f_1}\cdot a_{g_2}$ can be written as  a linear combination of Majorana axes. Thus eventually we are reduced to computing scalar products of Majorana axes, which are given by the Norton-Sakuma Theorem.
Similar arguments give the scalar products for $i\in \{25, 26,27,29, 30,31\}$. 
Finally, by the symmetry of the scalar product, the products 
$(u_{ e_1}, u_{e_{18}})_V$, $(u_{ e_1}, u_{e_{22}})_V$, and $(u_{ e_1}, u_{e_{28}})_V$ coincide, respectively with $(u_{ e_1}, u_{e_{19}})_V$, $(u_{ e_1}, u_{e_{24}})_V$, and $(u_{ e_1}, u_{e_{27}})_V$.

\end{proof}

%%%%%%%%%%%%%%%%%%%%%%%%%%%%%%%%%%%%%%%%%%%%%%%

\section{A spanning set for the $2$-closures of $V$ and $W$}\label{4assi}

We stick to the notation of the Section~\ref{strategy} and, for the rest of this paper, we assume that $\phi$ is a $2A$-Majorana representation of $G$. 
Let $$g:=(a_1,b_1)(a_2,b_2)(a_3,b_3)(a_4,b_4)$$ 
be a permutation of cycle type $2^4$ and let $S_a$ be the subgroup of $G$ with support contained in $\{a_1,a_2,a_3,a_4\}$. Define $$H_{(g)}:=[S_a,g].$$
Clearly  $H_{(g)}$ is a diagonal subgroup of the direct product $S_a\times S_a^g$, in particular $H_{(g)}$ is isomorphic to $S_4$.  
For distinct $t_1, t_2\in \mathcal T\cap H_{(g)}$, set
\begin{equation}\label{sigma}
\sigma_{t_1t_2}:=a_{t_1}\cdot a_{t_2}-\frac{1}{32}(a_{t_1}+ a_{t_2}).
\end{equation}
By~\cite[Lemma 2.3]{IPSS},  $\sigma_{t_1t_2}$ is invariant under the actions of the subgroup of $H_{(g)}$ generated by $t_1$ and $t_2$ and isomorphic to $S_3$ and can be indexed by the product $t_1t_2$. Moreover, since the Norton Sakuma subalgebra generated by $a_{t_1}$ and $a_{t_2}$ has type $3A$, by Table~\ref{table1}, $\sigma_{t_1t_2}$ is expressible as a linear combination of $a_{t_1}, a_{t_2}, a_{t_1t_2t_1}$ and $u_{t_1t_2}$.

Let $$\xi:= (a_1,a_2)(a_3,a_4)(b_1,b_2)(b_3,b_4)$$ and define 
\begin{eqnarray}\label{delta}
\delta_{\xi} &:=&\sigma_{(a_2, a_3, a_4)(b_2, b_3, b_4)} \cdot  a_{(a_1,a_2)(b_1,b_2)}\\
& &-\frac{1}{2^5}\sigma_{(a_2, a_3, a_4)(b_2, b_3, b_4) }
+\frac{1}{2^{10}}a_{(a_1,a_2)(b_1,b_2)}. \nonumber
\end{eqnarray}

\begin{lemma}\label{4^2}
The vector $\delta_{\xi} $ depends only on the involution $\xi$.
\end{lemma}
\begin{proof}
First note that $\xi\in H_{(g)}$.  Since  $H_{(g)}$ is isomorphic to $S_4$ and it is generated by the set $\mathcal T\cap H_{(g)}$, i.e. the set of the permutations of type $2^2$ in $H_{(g)}$, by Proposition~\ref{shape2^6} it follows that every $2A$-Majorana representation of $G$ induces a Majorana representation of $H_{(g)}$ and the product of any two elements of $\mathcal T\cap H_{(g)}$ has either cycle type $2^4$ or $3^2$. By  Table~\ref{orbs}, with the notation of~\cite[Proposition 4.3]{IPSS}, the representation induced on $H_{(g)}$ has type  $(2B, 3A)$ and the claim then follows from~\cite[Lemma~4.6]{IPSS}.
\end{proof}
\begin{lemma}\label{delta4assi}
In the algebra $V$ the following relation holds
\begin{eqnarray*}
v_{(1,3,2,4)(5,7,6,8)}&=&\frac{1}{3^2}(a_{(1,2)(5,6)}+a_{(3,4)(7,8)})\\
& &+\frac{7}{2\cdot 3^2}(a_{(1,3)(5,7)}+a_{(1,4)(5,8)}+a_{(2,3)(6,7)}+a_{(2,4)(6,8)})\\
& &-\frac{2^5\cdot 5}{3^3}(\sigma_{(1,2,3)(5,6,7)}+\sigma_{(1,3,4)(5,7,8)}+\sigma_{(1,2,4)(5,6,8)}+\sigma_{(2,3,4)(6,7,8)})\\
& & -\frac{2^{11}}{3^3}\delta_{(1,2)(3,4)(5,6)(7,8)} +\frac{2^{12}}{3^3}(\delta_{(1,3)(2,4)(5,7)(6,8)}+\delta_{(1,4)(2,3)(5,8)(6,7)}).
\end{eqnarray*}
\end{lemma}
\begin{proof}
Set $g=(1,5)(2,6)(3,7)(4,8)$, $\xi=(1,2)(3,4)(5,6)(7,8)$ and 
$$H:=H_{(g)} \times \langle (1,5)(2,6)(3,7)(4,8)(9,10)(11,12)\rangle,
$$ 
so tha $H$ is a subgroup of $A_{12}$ isomorphic to $S_4\times 2$ and generated by $\mathcal T \cap H$. As in the proof of Lemma~\ref{4^2}, the $2A$-Majorana representation $\phi$ of $A_{12}$ induces a $2A$-Majorana representation $\phi_{|_{H}}$ of $H$. In particular, the shape of $\phi_{|_{H}}$  is determined by the shape of $\phi$, which is given in Table~\ref{orbs} by Proposition~\ref{shape2^6}.   
Thus the relation can be checked within the smaller subalgebra (of dimension $23$) generated by $(\mathcal T \cap H)^{\psi}$ either by hand (not reccommended) or using the GAP package {\it MajoranaAlgebras}~\cite{MM}. \end{proof}
The formula of Lemma~\ref{delta4assi} is the same as that given in~\cite[p.2462]{IPSS}, which has been obtained under the assumption that a certain parameter $m$ is equal to $-3$. The latter fact seems to have already been established by Norton computing inside the Griess algebra. Our result, which has been obtained within the Majorana axiomatics and is therefore independent of the Griess algebra and the Monster, confirms Norton's computations and the formula in~\cite{IPSS}.

\begin{lemma}\label{formulaccia}
Let $O$ be the subgroup of $\hat G$ generated by the permutations
$$ (1,2)(3,4)(5,6)(7,8), \:\:(1,6)(2,5)(3,7)(4,8), \mbox{ and } (1,3)(2,4)(5,8)(6,7),$$ let $N:=N_{A_8}(O)$, $S:=\langle (9,10)(11,12)\rangle \times N$, and let $V_S$ be the subalgebra of $V$ generated by the Majorana axes $a_z$ such that $z\in \mathcal T\cap S$. Then
\begin{enumerate}
\item  $V_S$ has dimension $204$,
\item $V_S$ is  $2$-closed,
\item  a linearly independent set of generators for  $V_S$ is given by 
$$
\mathcal B:=\{a_z\:|\: z\in \mathcal T \cap S \}\cup \{ u_\rho\:|\: \rho \in (2,8,6)(4,5,7)^S\}.
$$
\end{enumerate}
In particular, every  $4$-axis of type  $4^2$ in $V$ is a linear combination of 
Majorana axes and $3$-axes of type $3^2$. 
\end{lemma}
\begin{proof}
Note that $O$ is elementary abelian of order $8$ and $N$ is a maximal subgroup of $A_8$ isomorphic to $2^3:PSL(3,2)$ ($N$ is a parabolic subgroup in the group $PSL_4(2)\cong A_8$). 
Since $S$ is  a subgroup of $A_{12}$ generated by involutions of cycle type $2^2$ and $2^6$, the restriction of $\phi$ to $S$ is a Majorana representation $(S,\mathcal T_S,V_S,\phi_S, \psi_S)$ of $S$ where  $\mathcal T_S:=S\cap (\mathcal X_b\cup \mathcal X_s)$ is the union of the four $S$-orbits 
$$
(9,10)(11,12)^S,   (2,8)(4,5)^S, (1,2)(3,4)(5,6)(7,8)(9,10)(11,12)^S,  
$$ 
$$
(1,3)(2,4)(5,7)(6,8)(9,10)(11,12)^S
$$
and has size $1+42+7+42=92$. 
Note that, since $\phi_S$ is the restriction of a $2A$-Majorana representation, its shape is determined by Proposition~\ref{shape2^6}.
Using the GAP package {\it MajoranaAlgebras}~\cite{GAP, MM}, we get that $\phi_S$ is determined (up to equivalence) by its shape, it has dimension $204$, it is  $2$-closed and a linearly independent set of generators for  $V_S$ is given by 
$$
\mathcal B:=\{a_z\:|\: z\in \mathcal T_S \}\cup \{ u_\rho\:|\: \rho \in (2,8,6)(4,5,7)^S\}.
$$
Moreover, $V_S$ contains $4$-axes $v_\rho$ of type $4^2$. For $\rho=(1,3,2,4)(5,7,6,8)$, the explicit expression of $v_\rho$ as a linear combination of the elements of $\mathcal B$ is given in Section~\ref{Appendix}.
\end{proof}

\begin{proposition}\label{gen}
With the above notation , we have 
$$V^\circ= \langle V^{(2^2)}, V^{(2^6)}, V^{(3^2)}\rangle \: \mbox{ and } \:\:W^\circ=\langle W^{(2^2)}, W^{(3)}, W^{(3^2)}\rangle  .$$
\end{proposition}
\begin{proof}
As remarked in Section~\ref{strategy}, by the Norton-Sakuma Theorem,
$$
V^\circ=\langle V^{(2A)}, V^{(3A)}, V^{(4A)}, V^{(5A)}\rangle \:\:\mbox{ and }\:\:W^\circ=\langle  W^{(2A)}, W^{(3A)}, W^{(4A)}, W^{(5A)}\rangle .$$
By Proposition~\ref{shape2^6} and~\cite[Lemma~8]{FIM2}
$$
V^{(2A)}=\langle V^{(2^2)}, V^{(2^6)}\rangle \:\:\mbox{ and } \:\:
W^{(2A)}= W^{(2^2)} .
$$
By~\cite[Lemma~3.1]{Alonso}
$$
V^{(3)}\leq V^{(2A)}
$$
and 
by \cite[Lemmas~5.1 and~5.2]{Alonso}
$$
V^{(5A)} \leq  \langle V^{(2A)}, V^{(3^2)} \rangle\:\:\mbox{ and }\:\:
W^{(5A)} \leq  \langle W^{(2^2)}, W^{(3)} \rangle
$$
By~\cite[Lemmas~3.6,~4.14]{Alonso} we have
$$
 V^{(4A)}\leq  \langle V^{(2A)},V^{(4^2)}\rangle
$$
and, by Lemma~\ref{formulaccia}, we have $$V^{(4^2)}\subseteq \langle V^{(2A)}, V^{(3^2)} \rangle,$$ so the claim about  $V^\circ$ follows.
Finally, by~\cite[Lemma~8]{FIM2}, $W^{(4A)}=\{0\}$ and the result for $W^\circ$ also follows. 
\end{proof}

%%%%%%%%%%%%%%%%%%%%%%%%%%%%%%%%%%%%%%%%%%%%%%%%%%%%%

\section{The submodule $V^{(2A)}$}\label{2assi}
\begin{table}
{\Small
$$
\begin{array}{|c||c|c|c|c|c|c|c|c|c|c|}
 \hline
\lambda  &\Xi_1 & \Xi_2  & \Xi_3  & \Xi_4  & \Xi_5 & \Xi_6  &\Xi_7  & \Xi_8  & \Xi_9  &\Xi_{10} \\
\hline
\hline 
(12)& 1 & 2 & 32 & 64 & 56 & 112 & 112 & 224 & 672 & 210 \\
\hline
(11,1)& 1 & 2 & 20 & 40 & 14 & 28 & 28 & 56 & -84 & -105 \\
\hline
(9,3)& 1 & 2 & 2 & 4 & -4 & -8 & -8 & -16 & 42 & -15  \\
\hline
(8,4)& 1 & 2 & -4 & -8 & 2 & 4 & 4 & 8 & -12 & 3 \\
\hline
(9,2,1)& 1 & -1 & 5 & -5 & -7 & 7 & 7 & -7 & 0 & 0 \\
\hline
(8,2^2) & 1& -1 & -4 & 4 & 2 & -2 & -2 & 2 & 0 & 0 \\
\hline
(10,2)&
A_1
 &
 A_2
  & 
  A_3   &
    A_4
     &  A_5
     &  A_6
     & A_7      & A_8
      &  A_9
      &   A_{10}\\

\hline
\end{array}$$

\medskip 

\noindent $A_1= \left (\begin{array}{cc} 1&0\\0&1\end{array}\right )$, 
$A_2= \left (\begin{array}{cc} 2&0\\0&-1\end{array}\right )$, 
$A_3= \left (\begin{array}{cc} 10&0\\0&16\end{array}\right )$, 
$ A_4= \left (\begin{array}{cc} 20&0\\0&-16\end{array}\right )$, 
$ A_5= \left (\begin{array}{cc} -8/3&a\\\bar a&56/3\end{array}\right )$,
$A_6= \left (\begin{array}{cc} -16/3&b\\\bar c&-56/3\end{array}\right )$, 
$A_7= \left (\begin{array}{cc} -16/3&c\\\bar b&-56/3\end{array}\right ), $
$ A_8= \left (\begin{array}{cc} -32/3&d\\\bar d&56/3\end{array}\right )$,
$A_9= \left (\begin{array}{cc} -54&0\\ 0&0\end{array}\right )$,
 $A_{10}= \left (\begin{array}{cc} 45&0\\0&0\end{array}\right )$
  $$a\bar a=3080/9,\,\,\, b\bar c=-6160/9,\,\,\, d\bar d=12320/9$$
}
\caption{First eigenmatrix for the action of $S_{12}$ on its involutions of cycle type $2^2$} \label{brt} 
\end{table}

In this section we determine the decomposition into irreducible submodules of $V^{(2A)}$. Recall, from Section~\ref{strategy}, that $V^{(2A)}=(M_b+M_s)^\pi$. 
\begin{lemma}
\label{bitra} We have 
\begin{enumerate}
\item $
M_{b}= M_{b,1}^{(12)}\oplus  M_{b,1}^{(11,1)} \oplus  M_{b,1}^{(9,3)}\oplus  M_{b,1}^{(9,2,1)}\oplus  
M_{b,1}^{(8,2^2)}\oplus  M_{b,1}^{(10,2)} \oplus  M_{b,2}^{(10,2)}, 
$
\item a generalised first eigenmatrix $P(b)$, relative to the action of $\hat G$ on $\mathcal X_b$ with respect to the decomposition of $M_b$ in Equation~(\ref{dec}), is given in Table~\ref{brt}, and 
\item $\rad(M_b)=  M_{b,1}^{(9,2,1)}\oplus  M_{b,\ast}^{(10,2)},$
where $ M_{b,\ast}^{(10,2) }$ is a diagonal submodule of $ M_{b,1}^{(10,2)} \oplus  M_{b,2}^{(10,2)}$.
\end{enumerate}
\end{lemma}

\begin{proof}
The first assertion follows from~\cite[Lemma~6]{FIM2}, the generalised first eigenmatrix $P(b)$  has been computed in~\cite[Tables~8 and~9]{FIM2}, and the last assertion follows from~\cite[Theorem 2.]{FIM2}. 
\end{proof}

\begin{lemma}\label{radical2}
We have
\begin{enumerate}
\item $\rad(M_{s})=M_{s,1}^{(6,4,2)}\oplus M_{s,1}^{(4^2,2^2)}\oplus M_{s,1}^{(4,2^4)}\oplus M_{s,1}^{(6^2)}\oplus M_{s,1}^{(2^6)},
$
and
\item $
V^{(2^6)}\cong S^{(12)}\oplus S^{(10,2)}\oplus S^{(8,4)}\oplus S^{(8,2^2)}\oplus S^{(4^3)}\oplus S^{(6,2^3)}.
$
\end{enumerate}
\end{lemma}
\begin{proof}
The decomposition of $M_s$ as a direct sum of irreducible submodules is given in Lemma~\ref{la belle dame sans mercy}. By Lemma~\ref{Zapata}, a submodule $M_{s,1}^{\lambda}$ of $M_s$ is contained in  $\rad(M_{s})$ if and only if $f_{s,1}^\lambda=0$. The first assertion then follows from the formula for $f_{s,1}^ \lambda$ given in Equation~\ref{eqzapata}, using the values of the Majorana inner products given by the Norton-Sakuma Theorem (Table~\ref{table1}), and  Table~\ref{FE}. Since $V^{(2^6)}\cong M_s/\rad(M_s)$ (see Section~\ref{strategy}), the second assertion also follows.
\end{proof}

In the remainder of this paper we fix the following notation for  vectors in $M_{x}$, $x\in \{ b,s,t\}$: suppose $H$ is a subgroup of $\hat G$ and $\mathcal H_1^x, \ldots , \mathcal H_{n_x}^x$ are the $H$-orbits on $\mathcal X_x$, so that  
\begin{equation}\label{basis}
\mathcal H^x:=( \sum_{v\in \mathcal H_1^x} v, \ldots , \sum_{v\in \mathcal{H}_{n_x}^x} v)
\end{equation}
is a basis for $C_{M_{x}}(H)$. For every  $w\in C_{M_{x}}(H)$  let 
$$
\bar w:=(w_1^x, \ldots , w_{n_x}^x)^{\mathcal H^x}
$$
be the $n_x$-tuple of  the coefficients of $w$ written as a linear combination of the elements of the  basis ${\mathcal H}^x$.

\begin{lemma}\label{scalarproduct2}
Let $x, y \in \{b,s,t\}$, let $H$ be a subgroup of $\hat G$, $\mathcal H_1^x, \ldots , \mathcal H_{n_x}^x$ be the $H$-orbits on $\mathcal X_x$, $\mathcal H_1^y, \ldots , \mathcal H_{n_y}^y$ be the $H$-orbits on $\mathcal X_y$, and $z_i\in \mathcal H_i^x$ for every $i\in \{1, \ldots , n_x\}$.
Let $u_x\in C_{M_x}(H)$ and $u_y \in C_{M_y}(H)$. Then, with the notation above, 
$$
f(u_x, u_y)=\bar u_x D_x F(x, y) \bar u_y ^t, 
$$
where $D$ is the diagonal $n_x\times n_x$ matrix with $D_{ii}=|\mathcal H_i^x|$ and $F(x, y) $ is the $n_x\times n_y$ matrix whose $ij$-entry is
$$
F(x, y)_{ij}= \sum_{w \in \mathcal H_j^y} f(z_i, w). 
$$ 
\end{lemma}
\begin{proof}
By the  bilinearity and the $H$-invariance of $f$ and the definition of $F$, we have 
\begin{eqnarray*}
f(u_x,u_y)&=& f\left (\sum_{i=1}^{n_x}(u_i^x\sum_{v\in \mathcal H_i^x}v ),
\sum_{j=1}^{n_y}(u_j^y\sum_{w \in \mathcal H_j^y}w)\right ) \\
&=& \sum_{i=1}^{n_x} \sum_{j=1}^{n_y} u_i^x u_j^y\sum_{v\in \mathcal H_i^x}
\sum_{w \in \mathcal H_j^y}f(v, w)\\
&=& \sum_{i=1}^{n_x} \sum_{j=1}^{n_y} u_i^x u_j^y |\mathcal H_i^x|
\sum_{w \in \mathcal H_j^y}f(z_i, w)\\
&=&  \sum_{i=1}^{n_x} \sum_{j=1}^{n_y} u_i^x u_j^y |\mathcal H_i^x|F(x,y)\\
&=&\bar u_x D_x F(x, y) \bar u_y ^t.
\end{eqnarray*}
\end{proof}

\begin{lemma}\label{intersection2}
With respect to the decomposition in Equation~(\ref{dec}), for $\lambda\in \{(12), (8,4)$, $(8,2^2), (10,2)\}$, we have 
$$
\rad(M_{b}\oplus M_s)\cap (M_{b,1}^\lambda\oplus M_{s,1}^\lambda)\cong S^\lambda.
$$
\end{lemma}
\begin{proof}
Keeping the notation of Section~\ref{Scalar}, let 
$$s_1=(1,2)(3,4)(5,6)(7,8)(9,10)(11,12) \mbox{  and let  }H:=C_{S_{12}}(s_1).$$
For every $\lambda \in \{ (8,4), (8,2^2), (10,2)\}$, the module $S^\lambda$ has multiplicity $1$ in $M_{s}$, hence, by the  Frobenius Reciprocity Theorem, 
 $$\dim (C_{S^\lambda}(H))=1.$$
Denote the orbits of $H$ on the set $\mathcal  X_b$ as follows:
$$
\begin{array}{lll}
\mathcal R_1:=( (9,10)(11,12))^H, & \mathcal R_2:=( (9,11)(10,12))^H, & \mathcal R_3:=( (8,9)(11,12))^H,\\
\mathcal R_4:=( (8,9)(10,11))^H, & \mathcal R_5:=( (6,7)(10,11))^H & 
\end{array}
$$
and let $\mathcal R$ be the corresponding basis for $C_{M_{b}}(H)$ as in Equation~(\ref{basis}).
Set
$$
u := \sum_{v\in \mathcal R_1} v
$$
Since $u$ is $H$-invariant $u^{\pi_{b,1}^\lambda}\in C_{M_{b,1}^\lambda}(H)$. By Equation~(\ref{progi}) and Table~\ref{brt}, we get 
$$
\overline{ u^{\pi_{b,1}^{(8,2^2)}}}= \left (\frac{8}{15}, -\frac{4}{15}, -\frac{1}{15},\frac{1}{30},0\right )^{\mathcal R}
$$
$$
\overline{u^{\pi_{b,1}^{(8,4)}}}= \left (\frac{16}{63}, \frac{16}{63}, -\frac{2}{63},- \frac{2}{63}, \frac{1}{63}\right)^{\mathcal R}
$$
$$
\overline{u^{\pi_{b,1}^{(10,2)}}}= \left (\frac{11}{3}, -\frac{11}{6}, \frac{11}{6}, -\frac{11}{12} , 0\right)^{\mathcal R}. 
$$
Since $s_1$ is also $H$-invariant  $s_1^{\pi_{s,1}^{\lambda}}\in C_{M_{s,1}^\lambda}(H)$ for every $\lambda$. Let 
\begin{equation}\label{si}
\mathcal S_i:=\{z\in \mathcal X_{s}\:|\:(s_1,z)\in \Sigma_{i,s}\},
\end{equation}
so that $\mathcal S_1, \ldots ,  \mathcal S_{11}$ are the orbits of $H$ on the set $\mathcal X_s$ (see Table~\ref{orbs} and Lemma~\ref{hath thee in thrall} for the definition of $\Sigma_{i,s}$). Let $\mathcal S$ be the corresponding basis of $C_{M_s}(H)$ defined as in Equation~(\ref{basis}). By Equation~(\ref{progi}) and Table~\ref{FE} we get
\medskip 

{\Small
$$
\overline{s_1 ^{\pi_{s,1}^{(8,2^2)}}}=\left (\frac{8}{135}, \frac{4}{225},  -\frac{8}{2025}, \frac{11}{1350},  -\frac{4}{675},  -\frac{1}{270} , \frac{1}{675}, \frac{2}{2025}, -\frac{2}{2025},
 - \frac{1}{810},\frac{1}{675} \right )^{\mathcal S},
$$
$$
\overline{s_1 ^{\pi_{s,1}^{(8,4)}}}=\left (\frac{5}{189} ,\frac{2}{189}, \frac{1}{252}, \frac{1}{378}, \frac{5}{756}, \frac{1}{1512}, -\frac{1}{3024},  -\frac{1}{1512}, \frac{1}{504} - \frac{5}{3024},  -\frac{1}{3024} \right )^{\mathcal S},
$$
$$
\overline{s_1 ^{\pi_{s,1}^{(10,2)}}}=\left (\frac{2}{385}, \frac{19}{5775}, \frac{8}{5775}, \frac{9}{3850},  -\frac{1}{1925},  \frac{1}{2310},  -\frac{1}{1925}, \frac{8}{5775}, -\frac{1}{1925},  \frac{1}{2310}, -\frac{1}{1925} \right)^{\mathcal S}. 
$$
}

\medskip 
By Lemma~\ref{Zapata} and Lemma~\ref{scalarproduct2} we get 
$$
\det \left ( \begin{array}{cc}
f(u^{\pi_{b,1}^{(8,2^2)}}, u^{\pi_{b,1}^{(8,2^2)}}) & 
f(u^{\pi_{b,1}^{(8,2^2)}}, s_{1}^{\pi_{s,1}^{(8,2^2)}})\\
f(u^{\pi_{b,1}^{(8,2^2)}}, s_{1}^{\pi_{s,1}^{(8,2^2)}}) & f(s_1^{\pi_{s,1}^{(8,2^2)}}, s_1^{\pi_{s,1}^{(8,2^2)}})
\end{array}\right )
=
$$
$$=\det  \left (
\begin{array}{cc}
135/16 &      -15/16\\
       -15/16 &  5/48 
\end{array}
\right )=0
$$
$$
\det \left ( \begin{array}{cc}
f(u^{\pi_{b,1}^{(8,4)}}, u^{\pi_{b,1}^{(8,4)}}) & 
f(u^{\pi_{b,1}^{(8,4)}}, s_1^{\pi_{s,1}^{(8,4)}})\\
f(u^{\pi_{b,1}^{(8,4)}}, s_1^{\pi_{s,1}^{(8,4)}}) & 
f(s_1^{\pi_{s,1}^{(8,,4)}},s_1^{\pi_{s,1}^{(8,4)}})
\end{array}\right )
=
$$
$$=\det  \left (
\begin{array}{cc} 
75/14 &      25/28\\
       25/28 &  25/168 
\end{array}
\right )=0
$$
$$
\det \left ( \begin{array}{cc}
f(u^{\pi_{b,1}^{(10,2)}}, u^{\pi_{b,1}^{(10,2)}}) & 
f(u^{\pi_{b,1}^{(10,2)}}, s_1^{\pi_{s,1}^{(10,2)}})\\
f(u^{\pi_{b,1}^{(10,2)}}, s_1^{\pi_{s,1}^{(10,2)}}) & f(s_1^{\pi_{s,1}^{(10,2)}}, s_1^{\pi_{s,1}^{(10,2)}})
\end{array}\right )
=
$$
$$=\det  \left (
\begin{array}{cc}
1164625/192 &  -275/32\\
    -275/32&   15/1232\\
\end{array}
\right )=0
$$
Then, Lemma~\ref{int} yields the claim when $\lambda \in \{(8,4), (8,2,2), (10,2)\}$. 
Finally, when $\lambda=(12)$, the modules $M_{b,1}^{(12)}$ and $M_{s,1}^{(12)}$ have dimension $1$ and they are generated by the vectors $v_1$ and $v_2$, where 
$$
\overline{v_1}:=(1,1,1,1,1)^{\mathcal R} \mbox{ and } \overline{v_2}:=(1,1,1,1,1,1,1,1,1,1,1)^{\mathcal S}
$$
respectively. By Lemma~\ref{scalarproduct2} we get 
$$
\det  \left ( \begin{array}{cc}
f(v_1, v_1) & 
f(v_1, v_2)\\
f(v_2, v_1) & f(v_2,v_2)
\end{array}\right )
= \det \left ( \begin{array}{cc}
467775/8 & 3274425/8
\\
 3274425/8& 22920975/8
\end{array}\right )=0,
$$
and again Lemma~\ref{int} yields the claim.
\end{proof}

\begin{proposition}
\label{resume}
We have 
$$V^{(2A)}\cong S^{(12)}\oplus S^{(10,2)}\oplus S^{(8,4)}\oplus S^{(8,2^2)}\oplus S^{(4^3)}\oplus S^{(6,2^3)}\oplus S^{(11,1)}\oplus S^{(9,3)}.
$$
\end{proposition}
\begin{proof}
This follows from Lemma~\ref{bitra}, Lemma~\ref{radical2}, and Lemma~\ref{intersection2}. 
\end{proof}

%%%%%%%%%%%%%%%%%%%%%%%%%%%%%%%%%%%%%%%%%%%%%%%%%%%%%

\section{The module $W^{(3^2)}$}
\label{3assi}

In this section $a,b$ will always denote a pair of disjoint   $3$-cycles in $\hat G$.
Let  $\bar \beta$  be the map 
$$\begin{array}{rcccc}
\bar \beta & : &\mathcal X_{t} &\to  &\mathcal X_{t}\\
& & \langle ab \rangle & \mapsto & \langle ab^{-1}\rangle 
\end{array}.
$$
Since $\bar \beta$ is an isomorphism of $\hat G$-sets,  $\bar \beta$ induces an involutory $\R[\hat G]$-automorphism $\beta$ on the module $ M_{t}$, which decomposes as 
\begin{equation}
\label{decio}
M_{t}=C_{M_{t}}(\beta)\oplus [M_{t}, \beta]
\end{equation}  
and this decomposition projects into the decomposition 
\begin{equation}
\label{deciao}
W^{(3^2)}=W^{(3^2)+}+W^{(3^2)-},
\end{equation}
where 
$$W^{(3^2)+}:=\langle u_{ab}+u_{ab^{-1}}\:|\: a,b \mbox{ disjoint $3$-cycles in } \hat G\rangle$$ 
and 
$$W^{(3^2)-}:=\langle u_{ab}-u_{ab^{-1}}\:|\: a,b \mbox{ disjoint $3$-cycles in } \hat G\rangle.$$

By Pasechnik's relation~\cite[Lemma~3.4]{A67} and~\cite[Lemma 3.1]{Alonso}, we have 
\begin{equation}\label{pasechnik}
W^{(3^2)+}\leq \langle W^{(2^2)}, W^{(3)}\rangle \leq  V^{(2A)}. 
\end{equation}
Since our ultimate goal is to determine $V^\circ$ and $W^\circ$ and the modules $\langle W^{(2^2)}, W^{(3)}\rangle$ and  $V^{(2A)}$ have been determined in~\cite[Theorem 1.2]{FIM4}, and in Proposition~\ref{resume}, respectively, 
by Proposition~\ref{gen} it follows that we only need to consider  the contribution of the module $W^{(3^2)-}$, which in turn is isomorphic to 
$$
[M_{t}, \beta]/\rad([M_{t}, \beta]).
$$

Let, for $8\leq n\leq 12$,  
\begin{itemize}
\item[-] $\Omega^n_1, \ldots ,\Omega^n_{r_n}$ be the orbitals of $\hat G$ on $\mathcal X_{t}$, where $r_n=20,25,29,30,31$ for $n=8,9,10,11,12$ respectively,
\item[-] $k_i(n)$ be the valencies of the action, that is $k_i(n):=|\Omega^n_i|/|{\mathcal X}_{t}|$. For $n=8,10,11,12$, they 
 are displayed in the last three columns of Table~\ref{T1}.
\item[-] $e_i, i\in \{1, \ldots , 31\}$, as in Table~\ref{Orbitals12} so that for each $i$, $(\langle e_1\rangle , \langle e_i\rangle )\in \Omega^n_i$, and
\item[-] $\sigma_i$, $i\in \{1, \ldots , 31\}$, as in Table~\ref{T1} so that for each $i$, $e_1^{\sigma_i}=e_i$.\end{itemize}
\medskip

The module $C_{M_{t}}(\beta)$ is isomorphic to  the permutation $\R[\hat G]$-module associated to the action of $\hat G$ on the set of conjugates of  $\langle (1,2,3), (4,5,6)\rangle $, and its decomposition into irreducible submodules  can be easily computed with GAP~\cite{GAP}. Comparing this decomposition with that of $M_{t}$, given in~\cite[Theorem~1.1]{FIM3}, we get that 
we may choose $M^{(n-4,2^2)}_{t,1}$ and $M^{(n-4,2^2)}_{t,2}$ such that $M^{(n-4,2^2)}_{t,2}\leq C_{M_{t}}(\beta)$ and 
\begin{equation} \label{decU-}
\begin{array}{ccl}
[M_{t}, \beta]&=& M^{(n-4,2^2)}_{t,1}\oplus M^{(n-5,2^2,1)}_{t,1}\oplus M^{(n-6,2^3)}_{t,1}\oplus M^{(n-4,1^4)}_{t,1}\\
& & \\
 & &\oplus M^{(n-5,2,1^3)}_{t,1}\oplus M^{(n-5,1^5)}_{t,1}\oplus M^{(n-6, 2,1^4)}_{t,1}.
\end{array}
\end{equation}

\begin{table}
{\Small
$$
\begin{array}{|c|c|c|c|c|c|}
\hline
i& \sigma_i&k_i(8)&k_i(10)&k_i(11)& k_i(12)\\
\hline
1& () & 1 & 1& 1& 1 \\
\hline
2&  (2,3) & 1 & 1&1 &1\\
\hline
3&   (3,4) & 9 & 9&9& 9\\
\hline
4&  (2,4,3) & 9 & 9&9 & 9\\
\hline
5& (1,4,7) & 36 &72 &90 & 108\\
\hline
6&  (1,4,7)(2,3) & 36 &  72&90&108\\
\hline
7&  (1,2,3,7) &12& 24 &30&36\\
\hline
8&(1,2,7) & 12 & 24 &30& 36\\
\hline
9&  (3,4)(1,6,7) & 72 & 144&180&216\\
\hline
\hline
10& (1,5,3,4,7) & 72 & 144&180& 216\\
\hline
11&  (1,7)(4,8) & 18 & 108 &180 &270\\
\hline
12&(1,7)(2,3)(4,8) & 18 & 108&180 &270\\
\hline
13&  (1,4,7)(2,5,8) & 36 &216 &360 & 540\\
\hline
14&  (1,4,8)(2,7) & 72 & 432 &720& 1080\\
\hline
15 &(1,7)(2,8)& 12 & 72 &120& 180\\
\hline
16&  (1,4,7)(5,8) & 72 & 432 &720& 1080\\
\hline
17 &  (1,7)(3,4)(6,8) & 18 & 108 &180& 270\\
\hline
18&  (1,7)(3,4,8,6)  & 18 & 108  &180& 270 \\
\hline
19 &  (1,7)(3,4,8,5) & 18 & 108 &180& 270\\
\hline
20 &  (1,7)(3,4)(5,8) &  18 &  108 &180& 270      \\
\hline
\hline
21&   (1,4,8)(2,7)(5,9) &  &  864&2160 &4320\\
\hline
22& (2,8)(3,9)(1,6,7)&  &  144&360 &720\\
\hline
23& (1,7)(2,9)(4,8) &  & 432&1080 & 2160\\
\hline
24& (1,7)(2,8)(3,9,6)&  & 144 &360& 720\\
\hline
25 &(1,7)(2,8)(3,9)&  & 16 &40& 80\\
\hline
\hline
26&  (1,7)(2,9)(4,8)(5,10) &  &  &540& 1620\\
\hline
27 & (1,4,8)(2,7)(5,9)(6,10) &  & &360& 1080\\
\hline
28& (1,5,10)(2,7)(3,9)(4,8)  &  &  &360&1080 \\
\hline
29 & (1,8)(2,9)(3,10)(4,7) &  &  &240& 720\\
\hline
30 &  (1,9)(2,10)(3,11)(4,7)(5,8)&   &  & 120& 720      \\
\hline
\hline
31 & (1,7)(2,8)(3,9)(4,10)(5,11)(6,12) &   &  &  &20  \\
\hline
\end{array}
$$
\caption{Set $\Sigma$ and values of the valencies $k_i(n)$.}\label{T1}
}
\end{table}

We now have to determine which irreducible submodules of the decomposition in Equation~(\ref{decU-}) are contained into $\rad([M_{t}, \beta])$.  As in Section~\ref{2assi}, we do this by computing the restrictions of the form $f$ to the irreducible submodules of $[M_{t}, \beta]$ using Lemma~\ref{Zapata}. 

For a partition $\lambda$ of $n$, denote by $M^\lambda$ the permutation $\R[\hat G]$-module associated to the action of $\hat G$ on the set of $\lambda$-tableaux, let $\mathcal B_\lambda$ be the basis of $M^\lambda$ consisting of all $\lambda$-tableaux and let $\kappa_\lambda:M^\lambda\times M^\lambda \to \R$ be the $\hat G$-invariant non-degenerate symmetric bilinear form on $M^\lambda$ such that $\mathcal B_\lambda$ is an orthonormal basis with respect to $\kappa_\lambda$. We shall use in the sequel the following well known fact from the representation theory of the symmetric groups (see e.g.~\cite[Theorem 8.15]{J}).
\begin{lemma}\label{iso}
For every partition $\lambda$ of $n$, let  $\lambda^\prime$ denote the partition conjugate to $\lambda$ and let $A$ be a real vector space of dimension $1$ generated by a vector $a$ on which even permutations act trivially and odd permutations act as the multiplication by $-1$. Then the $\R[S_n]$-modules $S^{\lambda}$ and   $S^{\lambda^\prime}\otimes A$ are isomorphic.
\end{lemma}

\begin{lemma}\label{firsteigenmatrix}
The (relevant) values $(P(t)_i^\lambda)_{11}$ are listed in Table~\ref{T11}.
\end{lemma}
\begin{table} %%%%%%%TABELLA 1
{\tiny
$$\begin{array}{|c||r|r|r|r|r|r|r|r|r|}
\hline
 & \multicolumn{4}{|c|}{n=8} &\multicolumn{5}{c|}{n=11}\\
\hline
\begin{array}{crc}
&&\\
&&\lambda\\
&\,\,\,\,i&
\end{array}&(2^2,1^4) &(3,2,1^3) &  (2^4)& (3,1^5)& (6, 2^2,1) & (5,2^3)    &(6, 2,1^3)&(6, 1^5)  &(5,2,1^4)  \\
\hline
\hline
1& 1            & 1            & 1       & 1  &1              &1              &1             &1            &  1    \\
\hline
2& -1           & -1          &-1       &  -1 & -1            &-1             &-1            &-1           & -1    \\
\hline
3&  -3         & -3          &3         & -3 &  3               &3               & -3          & -3          & -3    \\
\hline
4&   3         &  3          &-3        & 3 &  -3              &-3             &3             &  3           & 3     \\
\hline
5&   6        & -8           &-6        & 12  &1                &-6             &-14          & 21           &6     \\
\hline
6&  -6        & 8            & 6        & -12   & -1            & 6             & 14          &-21           &-6     \\
\hline
7&  -6       & 1            & -6        & -4  &   1           &  -6            & 4            & -1            &-6     \\
\hline
8& 6         & -1           & 6         & 4  & -1            & 6             &-4             & 1             &6      \\
\hline
9&  12      & 5           &- 12       &  0  & 2            &-12            & 2           & -18            &12    \\
\hline
10& -12    &  -5         & 12        &  0  &   -2          &12              & -2         & 18              &-12  \\
\hline
11& 6        &-1          & 12        & -6   &  -16        &12            & -4          & -24             &6    \\
\hline
12& -6       & 1          & -12      &  6  &   16         &-12             & 4          & 24               &-6    \\
\hline
13-16        &  0          &  0 & 0& 0 &0& 0 & 0&0&0\\
\hline
17 & -6     &  1          & 6        &  6&       -8        &6              & 4             &24               &-6     \\
\hline
18&  6      &  -1         & -6       &  -6 &       8       &-6            & -4             &-24              &6      \\
\hline
19 &        6& -1         & -6      & -6   &        8      &-6             & -4            &-24               &6     \\
\hline
20 & -6       &  1        & 6        & 6   &       -8    &6              & 4              & 24                &-6    \\
\hline
21-30& & &  & &0 &0 &0 & 0&0\\
\hline
\end{array}
$$
$$\begin{array}{|c||r|r|r|r|r|r|r|r|}
\hline
&n=11 & \multicolumn{6}{|c|}{n=12} &8\leq n\leq 12\\
\hline
\begin{array}{crc}
&&\\
&&\lambda\\
&\,\,\,\, i&
\end{array}
      &    (7,2^2)  &(8,2^2) &(7,2^2,1) &(6,2^3) & (7,2,1^3) & (7,1^5)  & (6, 2,1^4) &(n-4, 1^4)   \\
\hline
\hline
1         & 1        &1                   & 1           & 1         & 1              &1              &1              &k_i(n)           \\
\hline  
2         &-1	        &-1           & -1          &-1        &-1              &-1             &-1            &-k_i(n)            \\
\hline
3         & 3          &3           & 3          &3          & -3               &-3            & -3          & -k_i(n)/3           \\
\hline
4       &  -3          &-3             &  -3          &-3        &  3             &3             &3            &  k_i(n)/3     \\
\hline
5        &10         &12              & 2         &-6        &-16                &36           &6          & k_i(n)/9\\
\hline
6       &-10         &-12          & -2            & 6        & 16              &-36          &- 6       &-k_i(n)/9\\
\hline
7        &10         &12          & 2              & -6        &   5            &-12           & -6         & k_i(n)/3\\
\hline
8       &-10         &-12          & -2             & 6         &   -5           &12            &6          & -k_i(n)/3\\
\hline
9        &20         &24           & 4               &- 12       &    1           &0             &12           & -2k_i(n)/9\\
\hline
10     &-20         &-24           &  -4             & 12        &     -1         &0             & -12         & 2k_i(n)/9\\
\hline
11     &20         &30             &-20            & 12        &     -5       &-90           & 6      & k_i(n)/9\\
\hline
12     &-20         &-30            & 20            & -12      &        5      &90             &-6       & -k_i(n)/9\\
\hline
13-16      &0&0         &  0 & 0&0&0 & 0&0\\
\hline
17   &10         &15  &  -10             & 6         &       5        &90           &-6       &-k_i(n)/9\\
\hline
18    &-10         &-15  &  10          & -6       &         -5       &-90            & 6      &k_i(n)/9\\
\hline
19 & -10         &-15    &10               & -6      &           -5      &-90           & 6        &k_i(n)/9\\
\hline
20    &10         &15     &  -10       & 6        &           5    &90             & -6        & -k_i(n)/9\\
\hline
21-30&0&0 & 0& 0 & 0 &0 &0 & 0\\
\hline
31& &0 & 0& 0 & 0 &0 &0 & 0\\
\hline
\end{array}
$$
\caption{The values $(P(t)_i^\lambda)_{11}$ for the relevant partitions (columns correspond to the rows of a first eigenmatrix)}\label{T11}
}
\end{table}

\begin{proof} Let $n\in\{8,\ldots,12\}$. For an orbital $\Omega_i^n$, let  
$$(\Omega_i^n)^\beta:=\{(X^\beta,Y^\beta)|(X,Y)\in \Omega_i^n\}.$$
Then, for $i\in\{1,3,5,7,9,11,17,19\}$, $(\Omega_i^n)^\beta=\Omega_{i+1}^n$ and, for $i\in\{13,\ldots,16,\}\cup \{21,\ldots,31\}$,  $(\Omega_i^n)^\beta=\Omega_{i}^n$.  By Lemma~\ref{dispari}, it follows that every row 
$$((P(t)^{\lambda}_1)_{11}, \ldots , (P(t)^{\lambda}_{ r_n})_{11})$$ of $P(t)$ corresponding to the irreducible submodules of $[M_t, \beta]$ is of the form 
{\Small
\begin{equation}
\label{rowsfirst}
(1,-1, x_1, -x_1,   x_2, -x_2,   x_3, -x_3,  x_4, -x_4,  x_5, -x_5, 0, 0, 0, 0,  x_6, -x_6,  x_7, -x_7, 0, \ldots ,0),
\end{equation}}
whence,  for each row, only the seven parameters $x_1,\ldots x_7$ need to be computed.

 Assume first that $\lambda\neq(n-4,2^2)$. By~\cite[Theorem~1.1]{FIM3}, 
the module $S^\lambda$ has multiplicity $1$ in $M_t$, so,  by~\cite[Lemma~3]{FIM2}, \begin{equation}\label{formulaP}
(P(t)^\lambda _ i)_{11}=\frac{\kappa_{\lambda}(w, w^{\sigma_i})}{\kappa_{\lambda}(w, w)}k_i(n),
\end{equation}
where $w$ is any non zero vector in $C_{S^\lambda}(N_{\hat G}(\langle e_1\rangle))$ and $e_1=(1,2,3)(4,5,6)$,  as in Table~\ref{Orbitals12}.  Clearly the vector $w$ can be chosen as a non zero sum of a $N_{\hat G}(\langle e_1\rangle)$-orbit of a polytabloid in the Specht module $S^\lambda$ (resp.  $S^{\lambda^\prime}\otimes A$). On turn this polytabloid is obtained by a suitable $\lambda$-tableau (resp. a  $\lambda^\prime$-tableau) $T$. 
So e.g. assume $\lambda=(n-4,1^4)$. 
Set 
$$T:= \begin{array}{l} 
1 6 \:7 \:8\ldots n \\
2  \\
 3\\
4\\
5 
  \end{array}, 
  $$
let $v$ be the polytabloid associated to the tableau $T$, and 
$$
w:=\sum_{\sigma \in N_{\hat G}(\langle e_1\rangle)} v^\sigma.
$$
Note that, since the numbers $6, \ldots , n$ do not appear in the lower rows of the summands of $v$ and $N_{\hat G}(\langle e_1\rangle)$ is the direct product of a subgroup fixing the set $\{1, \ldots , 6\}$ with a subgroup fixing the set $\{7, \ldots , n\}$, the numbers $7, \ldots , n$ do not appear in $w$ as well. Thus $w$ and the quotients 
$$
\frac{\kappa_{(n-4,1^4)}(w, w^{\sigma_i})}{\kappa_{(n-4,1^4)}(w, w)}
$$
are constant for $7\leq n\leq 12$. We can, therefore, reduce computations assuming $n=8$.

Similarly, assume $\lambda=(2^2,1^4)$.  
By Lemma~\ref{iso}, instead of working inside the module $S^{(2^2,1^4)}$, we can work inside the isomorphic module $S^{(6,2)}\otimes A$, which simplifies the computations. Consider the $\lambda^\prime$-tableau
$$
\begin{array}{l}
123457\\
86
\end{array}.
$$ 
As in~\cite{FIM2}, to simplify the notation, here and in the sequel  we shall write tabloids by substituting their first row by a bar. The polytabloid obtained from it is 
$$
v:=\overline{86}+\overline{12}-\overline{16}-\overline{82}
$$
and we choose  
$$
w:=\sum_{\sigma \in N_{S_8}(\langle e_1\rangle)} (v\otimes a)^\sigma=(\overline{86} + \overline{85}+\overline{84}+\overline{73}+\overline{72}+\overline{71} -\overline{76}-\overline{75}-\overline{74}-\overline{83}-\overline{82}-\overline{81})\otimes a .
$$
For all partitions but  $(7,2^2)$ and $(8,2^2)$ we proceed in the same way, choosing $T$ as in the following table

{\tiny
$$\begin{array}{|c||c|c|c|c|c|c|c|c|}
\hline
\lambda &(n-5,2^2,1) &(n-6,2^3) & (3,2,1^3) & (n-5,2,1^3) & (n-4,1^4) & (n-5,1^5) &(n-6,2,1^4)\\
\hline
T             
& 
%(n-5,2^2,1)
\begin{array}{l} 
1\: 4\:8\ldots n \\
2\:5  \\
 3\:6 \\
7
  \end{array}            
&            
%(n-6,2^3)
\begin{array}{l} 
\\
7\: 1\:2\:3 \\
8 \:4\:5\:6 \\
9\\
\vdots\\
n \\
\end{array}   

&
%(3,2,1^3)
\begin{array}{l} 
1\: 2\:3\:4\:5 \\
6\:7  \\
 8
  \end{array}

&
%(n-5,2,1^3)
\begin{array}{l} 
1\: 6\:8\ldots n \\
2\:7  \\
 3\\
4\\
5 
  \end{array}   
                
& 
%(n-4,1^4) 
\begin{array}{l} 
1\: 6\ldots n \\
2  \\
 3\\
4\\
5 
  \end{array}              
 &
 %(n-5,1^5)  
\begin{array}{l} 
1\: 6\:8\ldots n \\
2  \\
 3\\
4\\
5 \\
7
  \end{array}     
  &
 %(n-6,2,1^4)  
\begin{array}{l} 
7\: 1\:2\:3\:4\:5\\
8\: 6\\
\vdots \\
n
  \end{array}     
 \\
\hline
\end{array}
$$}
Once the tableau $T$ (hence the vector $w$) has been chosen, Equation~(\ref{formulaP}) and a straightforward computation give  the corresponding columns of Table~\ref{T11}. 

Assume now $\lambda\in \{(7,2^2), (8,2^2)\}$. In this case, we can't use the above method since $S^\lambda$ has multiplicity $2$ in $M_t$, so we need to proceed in a different way. The rows orthogonality relations (see~\cite[(3.8)]{Hi} and~\cite[Lemma 1]{FIM2}between the row of the first eigenmatrix $P(t)$ corresponding to $M_{t,1}^\lambda$  and the rows corresponding to the submodules $M_{t,1}^\mu$ for the  partitions $\mu$ in the following set
$$\{(n-5,2^2,1), \:(n-6,2^3), \:(n-4,1^4),\: (n-5,2,1^3),\: (n-5,1^5), (n-6,2,1^4), (n-4,2^2) \}$$
give, respectively  for $n=11, 12$,  the following quadratic systems of equations.
{\Small
$$
\left ( \begin{array}{cccccccccc}
\frac{2}{3}&\frac{2}{45}&\frac{2}{15}&\frac{2}{45}&-\frac{2}{9}&-\frac{1}{9}&\frac{1}{9}\\
&\\
\frac{2}{3}&-\frac{2}{15}&-\frac{2}{5}&-\frac{2}{15}&\frac{2}{25}&\frac{1}{15}&-\frac{1}{15}\\
&\\
-\frac{2}{3}&\frac{2}{9}&\frac{2}{3}&-\frac{4}{9}&\frac{2}{9}&-\frac{2}{9}&\frac{2}{9}\\
&\\
-\frac{2}{3}&-\frac{14}{45}&\frac{4}{15}&\frac{1}{45}&-\frac{2}{45}&\frac{2}{45}&-\frac{2}{45}\\
&\\
-\frac{2}{3}&\frac{7}{15}&-\frac{1}{15}&\frac{1}{5}&-\frac{4}{15}&\frac{4}{15}&-\frac{4}{15}\\
&\\
-\frac{2}{3}&\frac{2}{15}&-\frac{2}{5}&\frac{2}{15}&\frac{1}{15}&-\frac{1}{15}&\frac{1}{15}\\
&\\
\frac{2}{9}x_1& \frac{1}{45}x_2& \frac{1}{15}x_3& \frac{1}{90}x_4& \frac{1}{90}x_5& \frac{1}{90}x_6& \frac{1}{90}x_7
\end{array} \right)
\left (\begin{array}{c}
x_1\\ \\x_2\\ \\x_3\\ \\x_4\\ \\x_5\\ \\x_6\\ \\x_7
\end{array}
\right)
=
\left (\begin{array}{c}
-2\\ \\-2\\ \\-2\\ \\-2\\ \\-2\\ \\-2 \\ \\22
\end{array}
\right)
$$
}

{\Small
$$
\left ( \begin{array}{cccccccccc}
\frac{2}{3}&\frac{1}{27}&\frac{1}{9}&\frac{1}{27}&-\frac{4}{27}&-\frac{2}{27}&\frac{2}{27}\\
&\\
\frac{2}{3}&-\frac{1}{9}&-\frac{1}{3}&-\frac{1}{9}&\frac{4}{45}&\frac{2}{45}&-\frac{2}{45}\\
&\\
-\frac{2}{3}&\frac{2}{9}&\frac{2}{3}&-\frac{4}{9}&\frac{2}{9}&-\frac{2}{9}&\frac{2}{9}\\
&\\
-\frac{2}{3}&-\frac{8}{27}&\frac{5}{18}&\frac{1}{108}&-\frac{1}{27}&\frac{1}{27}&-\frac{1}{27}\\
&\\
-\frac{2}{3}&\frac{2}{3}&-\frac{2}{3}&0 &-\frac{2}{3}&\frac{2}{3}&-\frac{2}{7}\\
&\\
-\frac{2}{3}&\frac{1}{9}&-\frac{1}{3}&\frac{1}{9}&\frac{2}{45}&-\frac{2}{45}&\frac{2}{45}\\
&\\
\frac{2}{9}x_1& \frac{1}{54}x_2& \frac{1}{18}x_3& \frac{1}{108}x_4& \frac{1}{135}x_5& \frac{1}{135}x_6& \frac{1}{135}x_7
\end{array} \right)
\left (\begin{array}{c}
x_1\\ \\x_2\\ \\x_3\\ \\x_4\\ \\x_5\\ \\x_6\\ \\x_7
\end{array}
\right)
=
\left (\begin{array}{c}
-2\\ \\-2\\ \\-2\\ \\-2\\ \\-2\\ \\-2 \\ \\28
\end{array}
\right)
$$
}
The two systems have solutions 
$$
(3,10,10,20,20,10,-10) \:\mbox{ and } \:(3,12,12,24,30,15,-15),
$$
respectively. Substituting these values for the $x_i$'s in Equation~(\ref{rowsfirst}) we obtain, for $i\in \{1,\ldots, 30\}$ 
(resp. $i\in \{1,\ldots, 31\}$)  
the values $(P(t)_i^ {(7,2^2)})_{11}$  
(resp. $(P(t)_i^{(8,2^2)})_{11}$) 
of Table~\ref{T11}. 
\end{proof}

{\bf Remark.} Assume $M$ is an irreducible $\hat {G}$-submodule of $M_t$. Then, by Equation~(\ref{brunch}),  either $M\leq \rad(M_t)$ or $M\cap\rad(M_t)=\{0\}$. We'll use this fact together with the Branching Theorem (see~\cite[Theorem~9.2]{J}) to reduce computations to $S_m$-subgroups of $\hat G$ for $8<m<n$. So, e.g.,  by the Branching Theorem, the unique $\R[\hat G]$-submodule of $[M_t, \beta]$ that contains an $\R[S_8]$-submodule isomorphic to $S^{(2^2,1^4)}$  is $M_{t,1}^{(n-6,2,1^4)}$. Thus, $M_{t,1}^{(n-6,2,1^4)}\leq \rad(M_t)$ if and only if $M_{t,1}^{(2^2,1^4)}\leq \rad(M_t)$.

\begin{lemma}\label{n-6,2,1,1,1,1}
For $n\in \{8, \ldots , 11\}$, we have 
$$\rad([M_t, \beta])=M_{t,1}^{(n-5,2,1^3)}\oplus M_{t,1}^{(n-6,2,1^4)}
$$
and, for $n=12$, 
$$\rad([M_t, \beta])=M_{t,1}^{(7,2,1^3)}\oplus M_{t,1}^{(6,2,1^4)}\oplus M^{(8,1^4)}_{t,1}\oplus  M^{(7,2^2,1)}_{t,1}.
$$
\end{lemma}
\begin{proof}
By Lemma~\ref{Zapata}, $M_{t,1}^\lambda\leq \rad([M_t,\beta])$ if and only if $f^\lambda_{t,1}=0$. Thus, Equation~(\ref{eqzapata}), Table~\ref{T11}, and Table~\ref{Orbitals12} imply the following facts 
\begin{itemize}
\item [(i)] for $n\in \{11,12\}$, the decomposition of $\rad([M_t,\beta])$ is the one of the claim, 
\item [(ii)] $M^{(n-4,1^4)}_{t,1}\leq \rad([M_t,\beta])$ if and only if $n=12$,
\end{itemize} 
and, for $n=8$,  
\begin{itemize}
\item [(iii)] $M^{(2^2,1^4)}_{t,1}\leq \rad([M_t,\beta])$, 
\item [(iv)] $M^{(2^4)}_{t,1}\not \leq \rad([M_t,\beta])$,  
\item [(v)] $M^{(3,1^5)}_{t,1}\not \leq \rad([M_t,\beta])$, and
\item [(vi)] $M^{(3,2,1^3)}_{t,1}\not \leq \rad([M_t,\beta])$.

\end{itemize}
Let  $n\in\{9,10\}$. By the Branching Theorem (see~\cite[Theorem~9.2]{J}), the unique $\R[\hat G]$-submodule of $[M_t, \beta]$ that contains an $\R[S_8]$-submodule isomorphic to $S^{(2^2,1^4)}$ (resp. $S^{(2^4)}$)  is $M_{t,1}^{(n-6,2,1^4)}$ (resp. $M_{t,1}^{(n-6,2^3)}$). By (iii), (iv), and the above remark we have
\begin{itemize}
\item [(vii)] $M^{(n-6,2,1^4)}_{t,1} \leq \rad([M_t,\beta])$ for every $n$, and 
\item [(viii)] for $n\in\{9,10\}$,  $M^{(n-6,2^3)}_{t,1} \not \leq \rad([M_t,\beta])$.  
\end{itemize}
Similarly, the $\R[S_9]$-submodules of $[M_t, \beta]$ that contain an $\R[S_8]$-submodule isomorphic to $S^{(3,2,1^3)}$ are  $M_{t,1}^{(n-5,2,1^3)}$ and $M_{t,1}^{(n-6,2,1^4)}$.  Then, (vi) and (vii) imply that $M_{t,1}^{(3,2,1^3)}\leq M_{t,1}^{(4,2,1^3)}$ and so $M_{t,1}^{(4,2,1^3)}\not \leq \rad([M_t, \beta])$. Repeating the same argument also for $n=10$, we get 
\begin{itemize}
\item [(ix)] for  $n\in\{9,10\}$,  $M^{(n-5,2,1^3)}_{t,1} \not \leq \rad([M_t,\beta])$.   
\end{itemize}

Let  $n\in\{8,9,10\}$. By the Branching Theorem the submodule $M_{t,1}^{(5,2^2,1)}$ is contained either in $M_{t,1}^{(6,2^2,1)}$ or in $M_{t,1}^{(5,2^3)}$. Since, by (i) and (vii), the latter submodules are not contained in $ \rad([M_t,\beta])$, by the above remark, $M_{t,1}^{(5,2^2,1)} \not \leq \rad([M_t,\beta])$ either. 
The same argument for $n=9$ and $n=8$ gives
\begin{itemize}
\item [(x)] for  $n\in\{8,\ldots, 11\}$,  $M^{(n-5,2^2,1)}_{t,1} \not \leq \rad([M_t,\beta])$.   
\end{itemize}

Finally, let $n=10$. By (i), for $n=11$,  
$\rad([M_{t}, \beta])=M_{t,1}^{(6,2,1^3)}\oplus M_{t,1}^{(5,2,1^4)}$. By the Branching Theorem, none of its irreducible submodules contains an $\R[S_{10}]$-module isomorphic $S^{(6,2,2)}$, whence $M_{t,1}^{(6,2^2)} \not \leq \rad([M_{t}, \beta])$. By repeating this argument when $n\in \{8,9\}$ we finally get
\begin{itemize}
\item [(xi)] for  $n\in\{8,\ldots, 10\}$,  $M^{(n-4,2^2)}_{t,1} \not \leq \rad([M_t,\beta])$.   
\end{itemize}
\end{proof}

\begin{proposition}\label{fine3assi}
For $8\leq n\leq 11$,  
$$
W^{(3^2)-} \cong S^{(n-4,2^2)}\oplus S^{(n-5,2^2,1)}\oplus S^{(n-6,2^3)}\oplus S^{(n-4,1^4)}\oplus S^{(n-5,1^5)}, 
$$
while, for $n=12$,
$$
W^{(3^2)-} \cong S^{(8,2^2)}\oplus  S^{(6,2^3)}\oplus S^{(7,1^5)}.
$$
\end{proposition}
\begin{proof}
The claim follows from Equation~(\ref{decU-}) and Lemma~\ref{n-6,2,1,1,1,1}.
\end{proof}
In the sequel we'll denote by $V_1$, $V_2$, and $V_3$ the $\R[\hat{G}]$-submodules  $(M_{t,1}^{(n-4,2^2)})^\pi$,  $(M_{t,1}^{(n-6,2^3)})^\pi$, and  $(M_{t,1}^{(n-5,1^5)})^\pi$ of $W^{(3^2)-}$, which are isomorphic to $S^{(n-4,2^2)}$,   $S^{(6,2^3)}$, and $S^{(7,1^5)}$, respectively.

%%%%%%%%%%%%%%%%%%%%%%%%%%%%%%%%%%

\section{Intersections $W^{(2A)}\cap W^{(3^2)}$  and $V^{(2A)}\cap W^{(3^2)}$} \label{23assi}
In this section we determine the intersections $W^{(2A)}\cap W^{(3^2)}$  and $V^{(2A)}\cap W^{(3^2)}$. We keep the notation of Sections~\ref{2assi} and~\ref{3assi}. By the remark at the beginning of Section~\ref{3assi} we need only to determine the intersections 
$$
W^{(2A)}\cap  W^{(3^2)-} \mbox{ and } V^{(2A)}\cap W^{(3^2)-}.$$
Comparing the decompositions in Proposition~\ref{fine3assi}, Proposition~\ref{resume},  and~\cite[Theorem~2]{FIM2} gives, for $n=12$, 
$$
 V^{(2A)}\cap W^{(3^2)-}\leq V_1\oplus V_2 
$$
 and, for $8\leq n\leq 12$,
$$
W^{(2A)}\cap W^{(3^2)-}\leq U,  
$$
where $U$ is the irreducible submodule of $W^{(3^2)-}$ isomorphic to  $S^{(n-4,2^2)}$.
\bigskip

Consider first the intersection $
 V^{(2A)}\cap W^{(3^2)-}$.
\begin{lemma}\label{n-4,2,2}
For $n=12$, we have 
$V_1\oplus V_2\leq V^{(2A)}$.
\end{lemma}
\begin{proof}
Keeping the notation of Section~\ref{Scalar}, let 
$$s_1=(1,2)(3,4)(5,6)(7,8)(9,10)(11,12) \mbox{  and let  }H:=C_{S_{12}}(s_1).$$
For every $\lambda \in \{ (8,2^2), (6,2^3)\}$, the module $S^\lambda$ has multiplicity $1$ in $M_{s}$, hence, by the  Frobenius Reciprocity Theorem, 
 $$\dim (C_{S^\lambda}(H))=1.$$
For $i\in\{1,\ldots, 11\}$, let $c_i$ be as in Table~\ref{tau3}, denote by $\mathcal P_i$ the $H$-orbit $\{\langle c_i\rangle^h|h\in H\}$, and let $\mathcal P$ be the corresponding basis of $C_{M_t}(H)$ as in Equation~(\ref{basis}).
Set 
$$
u:=\sum_{v\in \mathcal P_2} v. 
$$
Since $u$ is $H$-invariant,  we have $u^{\pi_{t,1}^{\lambda}}\in C_{M_{t,1}^{\lambda}}(H)$. By Equation~(\ref{progi}) and Table~\ref{T11}, we get 
$$
\overline{u^{\pi_{t,1}^{(8,2^2)}}}=\left ( 0,\frac{1}{6}, -\frac{1}{6},0, 0, \frac{ 1}{18}, -\frac{1}{18}, 0, 0, 0\right )^{\mathcal P}
$$
$$
\overline{u^{\pi_{t,1}^{(6,2^3)}}}=\left (0, \frac{16}{5}, -\frac{16}{5}, 0, 0, -\frac{8}{15}, \frac{8}{15}, 0, 0, 0 \right )^{\mathcal P}.
$$
Since $s_1$ is also $H$-invariant  $s_1^{\pi_{s,1}^{\lambda}}\in C_{M_{s,1}^\lambda}(H)$ for every $\lambda$. Let  $\mathcal S_1, \ldots , \mathcal S_{11}$ be the orbits $H$ on the set $\mathcal X_t$ as in Equation~(\ref{si}) and let $\mathcal S$ be the corresponding basis of $C_{M_s}(H)$ defined as in Equation~(\ref{basis}). By Equation~(\ref{progi}) and Table~\ref{FE} we get
{\Small
$$
\overline{s_1^{\pi_{s,1}^{(8,2^2)}}}=\left (\frac{8}{135}, \frac{4}{225}, -\frac{8}{2025}, \frac{11}{1350}, -\frac{4}{675}, -\frac{1}{270}, \frac{1}{675}, \frac{2}{2025}, -\frac{2}{2025}, -\frac{1}{810}, \frac{1}{675}\right )^{\mathcal S}
$$ 
$$
\overline{s_1^{\pi_{s,1}^{(6,2^3)}}}=\left (\frac{5}{27}, 0, -\frac{7}{324}, \frac{1}{216}, \frac{1}{108},  \frac{1}{432}, \frac{1}{216}, -\frac{1}{648}, \frac{1}{648}, \frac{5}{2592}, -\frac{1}{432}\right )^{\mathcal S}.
$$
}

\medskip 
By Lemma~\ref{Zapata}, Lemma~\ref{scalarproduct2}, and Lemma~\ref{innerproduct} we get 
$$
\det 
\left (
\begin{array}{cc}
f(s_1^{\pi_{s,1}^{(8,2^2)}}, s_1^{\pi_{s,1}^{(8,2^2)}}) & f(s_1^{\pi_{s,1}^{(8,2^2)}}, u^{\pi_{t,1}^{(8,2^2)}})\\
\\
 f(s_1^{\pi_{s,1}^{(8,2^2)}}, u^{\pi_{t,1}^{(8,2^2)}}) &  f(u^{\pi_{t,1}^{(8,2^2)}}, u^{\pi_{t,1}^{(8,2^2)}})
\end{array}
\right )=
\left (
\begin{array}{cc}
\frac{320}{27} & -\frac{10}{9}\\
\\
-\frac{10}{9} & \frac{5}{48}
\end{array}
\right )=0
$$
and 

$$
\det 
\left (
\begin{array}{cc}
f(s_1^{\pi_{s,1}^{(6,2^3)}}, s_1^{\pi_{s,1}^{(6,2^3)}}) & f(s_1^{\pi_{s,1}^{(6,2^3)}}, u^{\pi_{t,1}^{(6,2^3)}})\\
\\
 f(s_1^{\pi_{s,1}^{(6,2^3)}}, u^{\pi_{t,1}^{(6,2^3)}}) &  f(u^{\pi_{t,1}^{(6,3^3)}}, u^{\pi_{t,1}^{(6,2^3)}})
\end{array}
\right )=
\left (
\begin{array}{cc} 
\frac{25}{54} & -\frac{256}{3}\\
\\
-\frac{256}{3}& \frac{1048576}{75}
\end{array}
\right )=0
$$
whence the claim, by Lemma~\ref{int}.
\end{proof}

\begin{proposition}\label{fineintersezione}
$V^{(2A)}+W^{(3^2)}=V^{(2A)}\oplus V_3$, where $V_3\cong S^{(7,1^5)}.$
\end{proposition}
\begin{proof}
Since by~\cite[Lemma~4.13]{Alonso}, $W^{(3^2)}$ is not contained in $ V^{(2A)}$, the claim follows from Equation~(\ref{pasechnik}), Proposition~\ref{fine3assi}, and Lemma~\ref{n-4,2,2}.
\end{proof}

We now turn to the intersection  $
W^{(2A)}\cap W^{(3^2)-} 
$, for $8\leq n\leq 12$.
As in Section~\ref{Scalar}, let $r_1=(1,2)(3,4)$ and set $K:=C_{S_{11}}(r_1).$ The module $S^{(n-4,2^2)}$ has multiplicity $1$ in $M_b$, hence, by the Frobenius Reciprocity Theorem, $\dim (C_{S^{(n-4,2^2)}}(K))=1$. For $i\in \{1, \ldots, 10\}$, let 
$$
\mathcal Q_i:=\{z\in \mathcal X_{b}\:|\:(r_1,z)\in \Sigma_{i,b}\},
$$
so that $\mathcal Q_1, \ldots ,  \mathcal Q_{10}$ are the orbits of $K$ on the set $\mathcal X_b$.
For $i\in \{1, \ldots , 13\}$, let $d_i$ be as in Table~\ref{t13},  $\mathcal N_i:=\{\langle d_i \rangle^k\:|\: k\in K\}$,
and let $\mathcal Q$ and $\mathcal N$ be the corresponding bases for $C_{M_{b}}(K)$ and $C_{M_{t}}(K)$ as in Equation~(\ref{basis}), respectively.

\begin{proposition}\label{bitranspositions}
For $n\in \{8, \ldots,10\}$, we have
$$
W^{(2A)}\cap W^{(3^2)-}=0
$$
and, for $n\in \{11, 12\}$, 
$$
W^{(2A)}\cap W^{(3^2)-}\cong
S^{(n-4,2^2)}. 
$$
\end{proposition}
\begin{proof}
As already observed, $W^{(2A)}\cap W^{(3^2)-}\leq U$, where $U$ is the submodule of  $W^{(3^2)-}$ isomorphic to $S^{(n-4,2^2)}$, so we only need to check for which $n$ the module $U$ is contained in $W^{(2^2)}$, or, equivalently, for which $n$  
$$\rad(M_{b,1}^{(n-4,2^2)}+M_{t,1}^{(n-4,2^2)})\neq 0.$$ 
If $n=12$, then by Proposition~\ref{resume} and Proposition~\ref{fineintersezione},  $S^{(8,2^2)}$ has multiplicity $1$ in $V^{(2A)}+W^{(3^2)}$. Since $W^{(2A)}+W^{(3^2)}\leq V^{(2A)}+W^{(3^2)}$ we get the claim.

Assume $n=11$.
Set 
$$
u_{11}:=\sum_{v\in \mathcal N_{11}} v.
$$
Then we have 
$$
\overline{u^{\pi^{(7,2^2)}_{t,1}}_{11}}=( 0, 0, 0, 0, 0, 0, 0, 0, 0, 0, 0, 12, -12 )^{\mathcal N}\in C_{M^{(7,2^2)}_{t,1}}(K).
$$ 
From~\cite[Table~8]{FIM2}, we get
$$
\overline{r_1^{\pi^{(7,2^2)}_{b,1}}}=\left ( 1, -\frac{1}{2}, -\frac{1}{7}, \frac{1}{14}, \frac{1}{21}, -\frac{1}{42}, -\frac{1}{42}, \frac{1}{84}, 0, 0 \right )^{\mathcal Q}\in C_{M^{(7,2^2)}_{b,1}}(K).
$$
By Lemma~\ref{Zapata}, Lemma~8 in~\cite{FIM2}, Lemma~\ref{innerproductbi} and Lemma~\ref{scalarproduct2} we get
$$
\det 
\left (
\begin{array}{cc}
f(r_1^{\pi_{b,1}^{(7,2^2)}}, r_1^{\pi_{b,1}^{(7,2^2)}}) & f(r_1^{\pi_{b,1}^{(7,2^2)}}, u^{\pi_{t,1}^{(7,2^2)}}_{11})\\
\\
 f(r_1^{\pi_{b,1}^{(7,2^2)}}, u^{\pi_{t,1}^{(7,2^2)}}_{11}) &  f(u^{\pi_{t,1}^{(7,2^2)}}_{11}, u^{\pi_{t,1}^{(7,2^2)}}_{11})
\end{array}
\right )=\det 
\left (
\begin{array}{cc} 
\frac{1215}{448}  & 216\\
\\
  216 & \frac{86016}{5}
\end{array}
\right )=0.
$$
Hence, by Lemma~\ref{int}, $\rad(M_{b,1}^{(7,2^2)}+M_{t,1}^{(7,2^2)})$ contains a submodule isomorphic to $S^{(7,2^2)}$. 
When $n=10$ we proceed as above setting  
$$
u_{10}:=\sum_{v\in \mathcal N_{10}} v.
$$
We have 
$$
\overline{u^{\pi^{(6,2^2)}_{t,1}}_{10}}=\left ( 0, 0, 0, 0, 0, 0, 0, 0, 0, 0, 0, \frac{28}{3}, -\frac{28}{3} \right )^{\mathcal N}\in C_{M^{(6,2^2)}_{t,1}}(K).
$$ 
From~\cite[Table~8]{FIM2} we get
$$
\overline{r_1^{\pi^{(6,2^2)}_{b,1}}}=\left ( 1, -\frac{1}{2}, -\frac{1}{6}, \frac{1}{12}, \frac{1}{15}, -\frac{1}{30}, -\frac{1}{30}, \frac{1}{60}, 0, 0\right )^{\mathcal Q}\in C_{M^{(6,2^2)}_{b,1}}(K).
$$
By Lemma~\ref{Zapata}, Lemma~8 in~\cite{FIM2}, Lemma~\ref{innerproductbi}, and  Lemma~\ref{scalarproduct2} we get
$$
\det 
\left (
\begin{array}{cc}
f(r_1^{\pi_{b,1}^{(6,2^2)}}, r_1^{\pi_{b,1}^{(6,2^2)}}) & f(r_1^{\pi_{b,1}^{(6,2^2)}}, u^{\pi_{t,1}^{(6,2^2)}}_{10})\\
\\
 f(r_1^{\pi_{b,1}^{(6,2^2)}}, u^{\pi_{t,1}^{(6,2^2)}}_{10}) &  f(u^{\pi_{t,1}^{(6,2^2)}}_{10}, u^{\pi_{t,1}^{(6,2^2)}}_{10})
\end{array}
\right )=\det 
\left (
\begin{array}{cc} 
\frac{ 189}{64}  & \frac{392}{3}\\
\\
 \frac{ 392}{3} & \frac{200704}{27}
\end{array}
\right )=\frac{43904}{9}.
$$
Therefore, $\rad(M_{b,1}^{(6,2^2)}+M_{t,1}^{(6,2^2)})=0$. The remaining cases follow using the Branching Theorem as in the proof of Lemma~\ref{n-6,2,1,1,1,1}.
\end{proof}

{\em Proof of Theorem~\ref{An}.} The claim follows from Proposition~\ref{gen}, Equation~(\ref{pasechnik}), Proposition~\ref{bitranspositions}, and~\cite[Theorem 1.2]{FIM4}.  

%%%%%%%%%%%%%%%%%%%%%%%%%%%%%%%%

\section{Projection on the irreducible submodule $V_3$}\label{projection}

Let $V_3$ be  the irreducible $\R[\hat{G}]$-submodule of $V$ isomorphic to $S^{(7,1^5)}$, as  
defined after Proposition~\ref{fine3assi}.  In this section we study the behaviour of the  $3$-axes of type $3^2$ under the projection $V^{(2A)}\oplus V_3 \to V_3$. Our aim is to take advantage of the smaller dimension of $V_3$ and find certain dependence relations between these axes that  will be needed  in Section~\ref{closure}. To do that, we need to express the projections of  the $3$-axes in terms of the basis of $V_3$ given by the images of polytabloids via an isomorphism $S^{(n-5, 1^5)} \to V_3$. We use the following elementary result.
\begin{lemma}\label{image}
Assume $M$ and $N$ are isomorphic irreducible $\R[\hat G]$-modules. Let $R$ be a subgroup of $\hat G$ such that  $\dim C_M(R)=1$. Then, for every $m\in C_M(R)\setminus\{0\}$ and every $n\in C_N(R)\setminus\{0\}$ there is a unique isomorphism of  $\R[\hat G]$-modules between $M$ and $N$ sending $m$ to $n$.
\end{lemma}

 We keep the notation of the previous sections.
For $n\in \{8,\ldots , 12\}$ and $x=t$, we choose the decomposition in Equation~(\ref{dec}) in such a way that Equation~(\ref{decU-}) is satisfied.  By Equation~(\ref{decio}) it follows that 
\begin{equation}
\label{se}
C_{M_t}(\beta)\leq \ker (\pi^{(n-5,1^5)}_{t,1}),
\end{equation}
where $\pi^{(n-5,1^5)}_{t,1}:M_t\to M^{(n-5,1^5)}_{t,1}$ is the canonical projection defined in Equation~(\ref{pipi}).  

Let $e_1=(1,2,3)(4,5,6)$, as defined in Section~\ref{Scalar}, and let $R:=N_{\hat G}(\langle e_1\rangle)$. By the Frobenius Reciprocity Theorem, $\dim(C_{S^{(n-5,1^{5})}}(R))=1$. In order to find a non-zero element $w_n$ in $C_{S^{(n-5,1^{5})}}(R)$, we proceed in a similar way as in the proof of Lemma~\ref{firsteigenmatrix}. Here it is convenient to use the isomorphism $S^{(n-5,1^5)}\cong S^{(6,1^{n-6})}\otimes A$, defined in Lemma~\ref{iso}, and work inside the latter module.
Set 
{\Small
$$
{T_n} := \begin{array}{l} 
1 \:2 \:3\: 4 \: 5\:7  \\
6  \\
8\\
\vdots\\
n 
  \end{array}
  $$
  }
  and let $e_{T_n}$ be the polytabloid associated to $T_n$. Set 
  $$
  w_7:=\sum_{\sigma\in N_{S_6}(\langle e_1\rangle)} (e_{T_n}\otimes a)^\sigma
  $$
 and, for $n\in \{8, \ldots , 12\}$,
  $$ 
  w_n:=\sum_{i=7}^{n}w_7^{(7,i)}.
  $$
 Then $w_n\in C_{S^{(6,1^{n-6})}\otimes A}(R)$. 
Since $M^{(n-5,1^5)}_{t,1}$ and  $S^{(6,1^{n-6})}\otimes A$ are isomorphic irreducible $\R[\hat G]$-modules and 
$(\langle e_1\rangle)^{\pi_{t,1}^{(n-5,1^5)}}\in C_{M^{(n-5,1^5)}_{t,1}}(R)$, by Lemma~\ref{image} there is a unique isomorphism of $\R[\hat G]$-modules
\begin{equation}
\label{zeta}
\zeta\colon M^{(n-5,1^5)}_{t,1}\to S^{(6,1^{n-6})}\otimes A
\end{equation}
sending 
$\langle e_1\rangle ^{\pi_{t,1}^{(n-5,1^5)}}$ to $w_n$.
  
\begin{lemma}\label{dipendenti} The following assertions hold:
\begin{enumerate}
\item for $n=8$, $S^{(6,1^2)}\otimes A$ is linearly generated by the set 
$$\mathcal G_8:=\{w_8^g \:|\: g\in S_8 \mbox{ and } (\langle e_1\rangle, \langle e_1^g\rangle )\in 
\Omega^{8}_i, i\in \{1, 3, 5,7, 10\}\},
$$
in particular the elements $w_8^{(1,7)(4,8)}$ and $w_8^{(1,4,7)(2,5,8)}$ of $S^{(6,1^2)}\otimes A$ can be written as linear combinations of the elements of $\mathcal G_8$ as follows
$$
\begin{array}{cclcl}
&\:\:\:\:\:&w_8^{(1,7)(4,8)}&=&-w_8+ w_8^{(1,4,7)}+w_8^{(1,4,8)}- w_8^{ (1,2,3,7)}+ w_8^{(4,5,6,7)}\\
\\
&&\mbox{and }\\
\\
&&w_8^{(1,4,7)(2,5,8)}&=&w_8^{ (3,4)}- w_8^{ (3,6)}-w_8^{  (1,4)}+w_8^{ (1,4,7)}+w_8^{ (1,5,7)}-w_8^{ (1,6,7)}\\
\\
&&& &-w_8^{ (2,4,7)}- w_8^{ (2,5,7)}+ w_8^{ (1,2,3,7)}- w_8^{ (1,7,2,3)}+w_8^{ (1,5,3,4,8) };
\end{array}
$$
\item for $n=9$, $S^{(6,1^3)}\otimes A$ is linearly generated by the set 
$$\mathcal G_9:=\{w_9^g \:|\: g\in S_9 \mbox{ and } (\langle e_1\rangle, \langle e_1^g\rangle )\in 
\Omega^{9}_i, i\in \{1,3, 5, 7, 11, 19\}\};
$$
\item for $n=10$, $S^{(6,1^4)}\otimes A$ is linearly generated by the set 
$$\mathcal G_{10}:=\{w_{10}^g \:|\: g\in S_{10} \mbox{ and  } (\langle e_1\rangle, \langle e_1^g\rangle )\in 
\Omega^{10}_i, i\in \{1, 5,7,11 ,22,23,24 , 25, 26, 27, 28\}\}.
$$
\end{enumerate}
\end{lemma}  
\begin{proof}
$(1)$ $S^{(6,1^2)}\otimes A$ has dimension $21$ and 
{\Small
$$
w_8=[
\begin{array}{l}
6\\
8
\end{array}
+
\begin{array}{l}
8\\
1
\end{array}
+       
\begin{array}{l}
6\\
7
\end{array}+
\begin{array}{l}
7\\
1
\end{array}-
\begin{array}{l}
8\\
6
\end{array}
-
\begin{array}{l}
1\\
8
\end{array}
-
\begin{array}{l}
7\\
6
\end{array}
-
\begin{array}{l}
1\\
7
\end{array}
+2(
\begin{array}{l}
1\\
6
\end{array}
-
\begin{array}{l}
6\\
1
\end{array})]\otimes a.
$$
}
If we set
{\Small
$$ \mathcal S_8:=\{(),
(3,4), (3,5), (3,6), (1,4), 
(1,4,7), (1,5,7), (1,6,7), (1,4,8), (1,5,8), (1,6,8), 
  (2,4,7), 
$$
$$  
  (2,5,7), (2,6,7), (2,4,8), (2,6,8),  (3,4,7), (3,5,7), (3,6,7), (3,4,8), 
  (3,5,8), (3,6,8),
  (1,2,3,7),
$$
$$ (1,2,3,8),  (1,7,2,3), (1,8,2,3), 
  (1,2,7,3), (1,2,8,3),  (4,5,6,7), (4,5,6,8),  (4,7,5,6), 
  $$
  $$(4,8,5,6),  (4,5,8,6), (1,5,3,4,8),
 (5,3,4,2,7)\},
 $$
 }
then, for every $g\in \mathcal S_8$,  $(\langle e_1\rangle, \langle e_1^g\rangle )\in 
\Omega^{8}_i$ with $ i\in \{1, 3, 5,7, 10\}$ and the tuple $(w_8^g)_{g\in \mathcal S_8}$ is a basis for the module $S^{(6,1^2)}\otimes A$. The expressions for $w_8^{(1,7)(4,8)}$ and $w_8^{(1,4,7)(2,5,8)}$ are easily found.
 
 (2) $S^{(6,1^3)}\otimes A$ has dimension $56$. As in the previous case, one can easily write explicitly the vector $w_9$ and check that  we get a basis for the module $S^{(6,1^3)}\otimes A$ by taking the tuple $(w_9^g)_{g\in \mathcal S_9}$, where
{\Small
$$ \mathcal S_9:=\{(), (3,4), (3,5), (3,6), (1,4), %3
  (1,4,7), (1,5,7), (1,6,7), (1,4,8), (1,5,8), (1,6,8), (1,4,9),  
  $$
  $$
  (2,4,7), (2,5,7), 
  (2,6,7), (2,4,8), (2,5,8), (2,6,8), (2,4,9), (2,5,9), (2,6,9), (3,4,7), (3,5,7), 
  $$
  $$
  (3,6,7),
  (3,4,8), (3,5,8), 
  (3,6,8), (3,4,9), (3,5,9), (3,6,9), %5
   (1,2,3,7), (1,2,3,8), (1,7,2,3), 
   $$
   $$
   (1,8,2,3),
   (1,9,2,3), (1,2,7,3), (1,2,8,3), 
  (1,2,9,3), (4,5,6,7), (4,5,6,8), (4,5,6,9), (4,7,5,6), 
  $$
  $$(4,8,5,6), (4,9,5,6), (4,5,7,6), %7
  (1,7)(5,9), (1,8)(4,7), 
  (1,9)(5,8), (2,7)(4,8), (2,7)(5,9), 
  $$
  $$
  (2,8)(4,7), (2,8)(5,9), (3,7)(5,9),%11
  (1,8,2,6)(4,9), %19
  (2,9)(1,4,8,5), % 19
  (3,9)(2,5,8,6) \}.
 $$
 }
Moreover,  for every $g\in \mathcal S_9$, $(\langle e_1\rangle, \langle e_1^g\rangle )\in 
\Omega^{9}_i$ for some  $i\in \{1,3, 5, 7, 11, 19\}$.  

(3) $S^{(6,1^4)}\otimes A$ has dimension $126$.  Using GAP~\cite{GAP}, one can check that  the tuple $(w_{10}^g)_{g\in \mathcal S_{10}}$ is a basis for the module $S^{(6,1^4)}\otimes A$,  where
{\Small
$$ \mathcal S_{10}:=\{ (),
 (4,5,6,7), 
  (4,5,6,8,7), (4,5,6,9,8,7), (4,5,6,10,9,8,7), (4,5,7,6), (4,5,8,7,6), 
  $$
  $$
  (4,5,9,8,7,6), 
  (4,5,10,9,8,7,6), (4,7,5,6), (4,7)(5,8)(6,9), (4,7)(5,8)(6,10,9), 
  $$
  $$
  (4,7)(5,9,6,10,8), (4,8,7,5,6), (4,8,5,9,6,10,7), 
  (4,9,8,7,5,6), (4,10,9,8,7,5,6), (3,4), 
  $$
  $$
  (3,4,7), (3,4,8,7), (3,4,9,8,7), (3,5,7), (3,5,8,7), (3,5,9,8,7), (3,6,7), 
  (3,6,8,7), (3,6,9,8,7), 
  $$
  $$
  (3,6,9)(4,7)(5,8), (3,6,10,9)(4,7)(5,8), (3,6,10,8,5,9)(4,7), (3,6,10,7,4,8,5,9), 
  (3,7,4)(5,8)(6,9), 
  $$
  $$
  (3,7,4)(5,8)(6,10,9), (3,7,4)(5,9,6,10,8), (3,7)(5,8,9), (3,7)(5,9), (3,8,5,7,4)(6,9), 
    $$
  $$
  (3,8,5,9,6,10,7,4), (3,8,9,5,7), (3,8,7)(5,9), (3,9,6,8,5,7,4), (3,9,6,10,8,5,7,4), 
  $$
 $$
  (3,9,6,10,7,4)(5,8), (3,9,5,7), (3,9,5,8,7), (3,10,9,6,8,5,7,4), (2,4,7), (2,4,8,7), 
  $$
  $$(2,4,9,8,7), (2,4,10,9,8,7), 
  (2,5,7), (2,5,8,7), (2,5,9,8,7), (2,6,7), (2,6,8,7), (2,6,9,8,7), 
  $$
  $$
  (2,6,9)(4,7)(5,8), (2,6,10,9)(4,7)(5,8), 
  (2,6,10,8,5,9)(4,7), (2,6,10,7,4,8,5,9),  
  $$
  $$
  (2,7,4)(5,8)(6,9), (2,7,4)(5,8)(6,10,9),
  (2,7,4)(5,9,6,10,8), 
  (2,7)(5,8,9), (2,7)(5,9), 
  $$
  $$
  (2,7)(4,8), (2,7)(4,9,8), (2,7)(3,8,6,9), (2,7)(3,8,6,10,9), (2,7)(3,8,5,9), 
  (2,7)(3,8,5,10,9), 
  $$
  $$(2,7)(3,8,4,9), (2,7)(3,8,4,10,9), 
  (2,7)(3,9)(6,8), (2,7)(3,9)(6,10,8), (2,7)(3,9)(5,8),  
  $$
  $$(2,7)(3,9)(5,10,8), 
  (2,7)(3,9)(4,8), (2,7)(3,9)(4,10,8), (2,7)(3,10,9)(6,8), (2,7)(3,10,8,6,9), 
    $$
  $$
  (2,7)(3,10,9)(5,8), 
  (2,7)(3,10,9)(4,8), 
(2,8,5,9,6,10,7,4), (2,8,9,5,7), (2,8,7)(5,9), 
$$
$$(2,8,4,9,7), (2,8,6,7)(3,9), (2,8,6,10,7)(3,9), 
  (2,8,5,7)(3,9), (2,8,5,10,7)(3,9), 
  $$
  $$
  (2,8,4,7)(3,9), (2,8,4,10,7)(3,9), (2,8,6,7)(3,10,9), (2,9,5,7), (2,9,5,8,7), 
  (2,9,3,10,8,6,7), 
  $$
  $$ (1,4,7), (1,4,8,7), (1,4,9,8,7), (1,4,8)(2,7)(5,9), (1,4,8)(2,7)(5,10,9), (1,4,9,5,10,8)(2,7), 
 $$
 $$
  (1,4,9,3,8,2,7), (1,4,10,9,3,8,2,7), (1,4,10,8,2,7)(3,9), (1,4,7,2,8)(5,9), 
  $$
  $$
  (1,4,7,2,8)(5,10,9), (1,4,9,5,10,7,2,8),
  (1,4,10,7)(2,8)(3,9), (1,4,7,2,9,5,8), 
  $$
  $$
  (1,4,7,2,9,5,10,8), (1,4,8)(2,9,5,10,7), (1,4,7,2,10,9,5,8), (1,5,9,3,8,2,7),
  $$
  $$
  (1,6,9)(4,7)(5,8), (1,6,10,9)(4,7)(5,8), (1,6,10,8,5,9)(4,7), (1,7)(5,8,9), (1,7)(5,9), 
  $$
  $$(1,7,5,9,2,8,4), 
  (1,7)(2,8,6,9), (1,7)(2,8,4,9), (1,9)(2,7)(5,8)(6,10), (2,5,8)(3,7)(4,10)(6,9)\}.
  $$
  }
  Moreover,  for every $g\in \mathcal S_{10}$, $(\langle e_1\rangle, \langle e_1^g\rangle )\in 
\Omega^{10}_i$ for some  $i\in \{1, 5,7,11 ,22,23,24 ,$
$ 25, 26, 27, 28\}$.  
\end{proof}

For a finite set $\Phi$ of natural numbers we denote by $A_\Phi$ the alternating group on $\Phi$.

\begin{proposition} \label{udependent}
In the algebra $V$ the following assertions hold:
\begin{enumerate}
\item For every $i\in \{11, \ldots , 20\}$,  the vector $u_{e_i}$ is a linear combination of Majorana axes and $3$-axes $u_\rho$ of type $3^2$ such that  $\rho$ is contained in $A_{\{1,\ldots , 6,7\}}\cup A_{\{1,\ldots , 6,7\}}$. 
In particular
$$
\begin{array}{cllll}
&\:\: &u_{e_{11}}\in& -u_{\rho_1}+u_{\rho_2}+u_{\rho_5}-u_{\rho_4}+u_{\rho_6} + V^{(2A)},\\
&&&\\
& &\mbox{and} & &\\
&&&\\
&\:\: &u_{e_{13}}\in & u_{ \rho_2}+ u_{\rho_3} + u_{\rho_4} -u_{\rho_7} -u_{\rho_{8}}+u_{\rho_{9}} -u_{\rho_{10}}-u_{\rho_{11}} - u_{\rho_{12}}\\
& & &-u_{\rho_{13}}+u_{ \rho_{14}} +V^{(2A)}.
\end{array}
$$
where
$$
\begin{array}{l}
 \rho_1=e_{1},\:\: \rho_2=e_{5},\:\:\rho_3=e_{3},\:\:\rho_4=e_7,\:\:\rho_5=(2,3,4)(5,6,8),\\
\rho_6=(1,2,3,)(5,6,7),\:\: \rho_7=(1,2,6)(3,4,5),\:\:\rho_{8}=(1,5,6)(2,3,4),\\
\rho_{9}=(2,3,5)(4,7,6),\:\:\rho_{10}=(2,3,6)(4,5,7),\:\:\rho_{11}=(1,4,3)(5,6,7),\\
\rho_{12}=(1,5,3)(4,7,6),\:\:\rho_{13}=(1,7,3)(4,5,6),\:\:\rho_{14}=(2,4,5)(3,6,8).\\
\end{array}
$$
\item For every $i\in \{21, \ldots , 25\}$, the vector $u_{e_i}$ is a linear combination of Majorana axes and $3$-axes $u_\rho$ of type $3^2$  such that $\rho$ is contained in 
$$
\bigcup_{k,l\in \{7,8,9\}}A_{\{1,\ldots , 6,k,l\}}.
$$
\item  For every $i\in \{26, \ldots , 29\}$, the vector $u_{e_i}$ is a linear combination of Majorana axes and $3$-axes $u_\rho$ of type $3^2$  such that $\rho$ is contained in 
$$
\bigcup_{k,l, m\in \{7,8,9, 10\}}A_{\{1,\ldots , 6,k,l,m\}}.
$$
\end{enumerate}
\end{proposition}

\begin{proof}
By Equation~(\ref{decio}), Equation~(\ref{se}), and   Lemma~\ref{n-6,2,1,1,1,1}, for every $n\in \{8, \ldots,12\}$, $$\rad(M_t)\leq \ker \pi_{t,1}^{(n-5,1^5)}.$$ Thus, the map $\pi_{t,1}^{(n-5,1^5)}$ induces a well defined homomorphism $$M_{t}/\rad(M_{t})\to M_{t,1}^{(n-5,1^5)}.$$
Composing this map with the isomorphism $W^{(3^2)}\cong M_{t}/\rad(M_{t})$,  we get  a well defined homomorphism 
$$\begin{array}{cccc}
\nu_n:\: &W^{(3^2)} &\to& M_{t,1}^{(n-5,1^5)}.\\
  \end{array}
  $$
  mapping $u_{\rho}$ to $\langle \rho \rangle ^{\pi_{t,1}^{(n-5,1^5)}}$, for every permutation $\rho$ of cycle type $3^2$. By Proposition~\ref{fineintersezione}, $W^{(3^2)}\cap V^{(2A)}$ contains no submodule isomorphic to $M_{t,1}^{(n-5,1^5)}$, so $$W^{(3^2)}\cap V^{(2A)}\leq \ker \nu_n.$$ The claims then follow from Lemma~\ref{dipendenti} when we consider the map $\nu_n\circ \zeta$.
\end{proof}

%%%%%%%%%%%%%%%%%%%%%%%%%%%%%%%

\section{Closure of $V^\circ$}\label{closure}
In this section we prove that $V^\circ$ is closed under the algebra product.

\begin{lemma}\label{red}
Let $z$ be an element in $\mathcal X_{b}\cup \mathcal X_{s}$, let $\sigma$ be a permutation of type $3^2$, and  $a,b$ disjoint $3$-cycles  in $A_{12}$. Then, the following assertions hold: 
\begin{enumerate}
\item  $
 a_z\cdot u_{ab} \in V^\circ \mbox{ if and only if } a_z\cdot u_{ab^{-1}} \in V^\circ $,
\item 
 $
  u_\sigma\cdot u_{ab} \in V^\circ \mbox{ if and only if } u_\sigma \cdot u_{ab^{-1}} \in V^\circ.
 $
 \end{enumerate}
 \end{lemma}
 \begin{proof}
By Equation~(\ref{pasechnik}), 
$$
W^{(3^2)+}\subseteq V^{(2A)}.
$$ 
Since
$$
a_z\cdot u_{ab}=a_z\cdot (u_{ab}+u_{ab^{-1}})-a_z\cdot u_{ab^{-1}}\in -a_z \cdot u_{ab^{-1}}+V^\circ 
$$ 
and similarly
$$
u_\sigma\cdot u_{ab}=u_\sigma\cdot (u_{ab}+u_{ab^{-1}})-u_\sigma \cdot u_{ab^{-1}}\in -u_\sigma \cdot u_{ab^{-1}}+V^\circ $$ 
the claims follow.
 \end{proof}

\begin{lemma}\label{nuovissimo}
$V^{(2^2)}\cdot V^\circ\subseteq V^\circ$ and, for every $z\in \mathcal X_{b}$ and every  every pair of disjoint $3$-cycles $a,b\in A_{12}$, the product $a_z \cdot u_{ab}$ is uniquely determined by the factors.
\end{lemma}
\begin{proof}
By Proposition~\ref{gen} we need to prove that, 
for every $z\in \mathcal X_{b}$ and every  every pair of disjoint $3$-cycles $a,b\in A_{12}$, the product 
$a_z \cdot u_{ab}$ is uniquely determined by the factors and is contained in $V^\circ$. 
Possibly substituting $t$ with one of its conjugates in $A_{12}$, we may assume that $z=(1,2)(3,4)$. 
Let $l$ be the order of the intersection of the supports of $z$ and $ab$. Suppose first $l>2$; possibly  substituting $ab$ with one of its conjugates under the stabiliser of $z$ in $A_{12}$, we may assume that the supports of $z$ and $ab$ are contained in $\{1,\ldots,7\}$ and the result follows by~\cite{A67}. 
Suppose $l=2$, then, as above, we may assume that  the supports of $z$ and $ab$ are contained in $\{1,\ldots,8\}$, but in none of its proper subsets. Then, with the notation of Table~\ref{Orbitals12},  $(\langle e_1\rangle , \langle ab \rangle)\in \Omega^{12}_i$ for some $i\in \{11, \ldots , 20\}$. By Proposition~\ref{udependent}(1), $u_{ab}$ is a linear combination of Majorana axes and $3$-axes $u_\rho$ of type $3^2$ such that $\rho$ belongs to the point stabiliser $A$ in $A_8$ of either $7$ or $8$. Since $A\cong A_7$, the result follows by the previous case. 
Suppose $l=1$, again we may assume  that  the supports of $z$ and $ab$ are contained in $\{1,\ldots,9\}$, but in none of its proper subsets, whence  $(\langle e_1\rangle , \langle ab \rangle)\in \Omega^{12}_i$ for some $i\in \{21, \ldots , 25\}$. By Proposition~\ref{udependent}(2), 
$u_{ab}$ is a linear combination of Majorana axes and $3$-axes $u_{\rho_j}$ of type $3^2$ such that the $\rho_j$'s belong to some subgroup $A_{\{1,\ldots ,6,k_j,l_j\}}$, $\{k_j, l_j\}\subseteq \{7,8,9\}$. Then the intersection between the supports of $z$ and $\rho_j$ is at least $2$, so, by the previous case, the product $a_z\cdot u_{\rho_j}$ is uniquely determined by the factors  and is contained in $V^\circ$ for every $j$, whence the claim for $a_z\cdot u_{ab}\in V^\circ$.
Finally, assume that the supports of $z$ and $ab$ are disjoint. Then there are $g, h\in \mathcal X_{b}$  disjoint  from $z$ such that  $ab=gh$. It follows that  the Norton-Sakuma subalgebras generated by $a_z$ and $a_g$ and by $a_z$ and $a_h$ are of type $2B$. In particular, by the Norton-Sakuma Theorem,  $a_g$ and $a_h$ lie in the  $0$-eigenspace for the adjoint action of $a_z$ (see Table~\ref{table1}). Thus the fusion law implies that $a_z\cdot u_{ab}=0$.
\end{proof} 

\begin{lemma}\label{sixtranspositions}
$V^{(2^6)}\cdot V^\circ\subseteq V^\circ$ and, for every $z\in \mathcal X_{s}$ and every pair of disjoint $3$-cycles $a,b\in A_{12}$, the product $a_z\cdot u_{ab}$ is uniquely determined by its factors.
\end{lemma}
\begin{proof}
As in Lemma~\ref{nuovissimo}, by Proposition~\ref{gen}, we only have to prove that, 
for every $z\in \mathcal X_{s}$ and every  every pair of disjoint $3$-cycles $a,b\in A_{12}$, 
the product $a_z\cdot u_{ab}$ is determined by its factors and is contained in $V^\circ$.
The orbits of $A_{12}$ on the set $\mathcal X_s\times \mathcal X_t$ correspond to the $11$ orbits $\mathcal P_1, \ldots , \mathcal P_{11}$ of the centraliser $H$ of the involution $s_1=(1,2)(3,4)(5,6)(7,8)(9,10)(11,12)$ on the set $\mathcal X_t$, considered in the proof of Lemma~\ref{n-4,2,2} and Table~\ref{tau3}.  As $A_{12}$ acts on $V$ as a group of algebra automorphisms, by Lemma~\ref{red}, it is enough to consider only the products $a_{s_1} \cdot u_{ab}$, where either $\langle ab\rangle$ or $\langle ab^{-1}\rangle$ are chosen in a set of  representatives of the orbits $\mathcal P_1, \ldots , \mathcal P_{11}$. 
Let $i\in\{1,\ldots, 11\}$. For the convenience of the reader, denote by $a_i$ (resp. $b_i$) the first (resp. second) factor appearing in the decomposition of $c_i$ in the second column of Table~\ref{tau3}: so, e.g. $a_1=(7,8,9)$ and $b_1=(10,11,12)$. 

The representatives of the orbits $\mathcal P_2$ and  $\mathcal P_3$ are $\langle a_2b_2\rangle$ and $\langle a_3b_3\rangle$, respectively. We have 
$$a_2b_2^{-1}=s_1r_2\mbox{ with }r_2:=(1,2)(3,4)(5,6)(8,9)(10,11)(7,12)$$ and 
$$a_3b_3=s_1r_3\mbox{  with }r_3:=(1,2)(3,4)(5,6)(7,11)(8,9)(10,12),$$ 
so that, by Axiom (M5) and Table~\ref{orbs}, $u_{a_2b_2^{-1}}=u_{s_1r_2}$ and $u_{a_3b_3}=u_{s_1r_3}$.
Thus, for $i\in \{2,3\}$,  $\langle \langle a_{s_1}, a_{r_i}\rangle \rangle$ is a Norton-Sakuma algebra of type $3A$ containing the $3$-axis $u_{s_1r_i}$ and 
the product $a_{s_1}\cdot u_{s_1r_i}$.  In this case, then, the claim follows by the Norton-Sakuma Theorem. 

The representatives of the orbits $\mathcal P_1$ and  $\mathcal P_4$ are $\langle a_1b_1\rangle$ and $\langle a_4b_1\rangle$, respectively, and, for $i\in \{1,4\}$, we may write $a_ib_1=gs_i$, where $g:=(7,8)(9,10)$  and $s_i\in \mathcal X_b$. Since $gs_1=s_1g$, by Axiom (M4) and Table~\ref{orbs}, it follows that 
$a_g$ is a $0$-eigenvector for the adjoint action of $a_{s_1}$ and, by~\cite[Lemma~1.10]{IPSS}, we have
$$
a_{s_1}\cdot \left (a_{s_i}\cdot a_g\right )= \left (a_{s_1}\cdot a_{s_i}\right )\cdot a_g.
$$
Since, by definition, $a_{s_1}\cdot a_{s_i}\in V^\circ$, by Lemma~\ref{nuovissimo}, $(a_{s_1}\cdot a_{s_i} )\cdot a_g \in V^\circ$. On the other hand, by the Norton-Sakuma Theorem,  we have
$$
a_{s_i}\cdot a_g=\frac{27\cdot 5}{2^{11}}u_{a_ib_1}+w_i \:\mbox{ with } \:w_i\in V^{(2A)},
$$
hence 
$$
a_{s_1}\cdot u_{a_ib_1}=\frac{2^{11}}{27\cdot 5}\left [a_{s_1} \cdot (a_{s_i}\cdot a_g)-a_{s_1}\cdot w_i \right ]\in  V^\circ.
$$

Let $i\in \{5, 6\}$. Let $g_5:=(2,3,6)(5,7)(9,10)$, $\bar c_5:= (5,6,7)(2,8,4)$ and $\bar c_6:= (4,7,5)(2,8,6)$,  and let $S$ be the group defined in Lemma~\ref{formulaccia}. 
Then $\{s_1,\bar c_i\} \subseteq S^{g}$ and $\bar c_i\in\mathcal P_i$. By Lemma~\ref{formulaccia}(2), we have
$$
a_{s_1}\cdot u_{\bar c_i}\in V^\circ.
$$

Let $i=7$. Then  $\langle a_5b_2\rangle \in \mathcal P_6$ and $\langle a_5b_2^{-1} \rangle \in \mathcal P_7$, so the claim follows by the previous case and Lemma~\ref{red}.

Let $i=8$ and $g_8:=(1,4,5,11,2,3,6,12)(8,10,9)$. Then $\langle e_{11}^{g_8}\rangle \in \mathcal P_8$. By~Proposition~\ref{udependent}(1) we have
\begin{equation}\label{11}
\begin{array}{rcll}
u_{e_{11}^{g_8}}& \in & -u_{(3,6,4)(5,11,12)} +u_{ (3,6,5)(7,11,12)}+u_{  (3,6,5)(10,11,12)}&\\
& &- u_{(3,6,7)(5,11,12)}+u_{(3,6,4)(7,11,12)}+V^{(2A)}.
\end{array}
\end{equation}
For every $u_{e}$ appearing on the right hand side of Equation(\ref{11}),  
$ \langle e \rangle
$ belongs to some of the orbits $\mathcal P_1, \ldots , \mathcal P_7$. Hence the claim follows by the previous cases.

Let $i=9$ and $g_9:=(1,4,6,8,9,12)(2,5,7,11) $. Then $ \langle e_{13}^{g_9}\rangle \in \mathcal P_{9}$ and, by~Proposition~\ref{udependent}(1), we have
\begin{equation}\label{13}
\begin{array}{rcll}
u_{e_{13}^{g_9}} &\in  
& u_{(3,7,8)(4,5,6)} 
-u_{(3,6,7)(4,5,8)} 
-u_{(3,6,5)(4,7,8)}
+u_{ (3,6,5)(7,8,11)}&\\
& &+u_{  (3,7,5)(6,11,8)}
-u_{ (3,8,5)(6,7,11)}
-u_{(3,4,6)(7,8,11)} 
- u_{(3,4,7)(6,11,8)}&\\
& &+ u_{(3,11,5)(6,7,8)}
-u_{(3,4,11)(6,7,8)}
+u_{ (3,8,9)(5,6,7)}+V^{(2A)}.
\end{array}
\end{equation}
For every $u_{e}$ appearing on the right hand side of Equation~(\ref{13}), $ \langle e \rangle$ belongs to some of the  of the orbits $\mathcal P_1, \ldots , \mathcal P_8$, so the claim follows as before.

For $i\in \{10, 11\}$ choose $g_{10}:=(5,8,11)(6,9,7)$ and $g_{11}:=(5,8,10)(6,11,7)$, then $ \langle e_{11}^{g_i}\rangle \in \mathcal P_i$ and again the result follows as in the previous cases.\end{proof}

We consider now the products between two $3$-axes of type $3^2$. 
\begin{proposition} \label{fine2closure}
$V=V^\circ$ and for every pair of permutations $e$ and $c$ of cycle type $3^3$, 
 the product $u_{e}\cdot u_{c}$ is uniquely determined by the factors. 
\end{proposition}
\begin{proof}
Clearly $V=V^\circ$ if and only if 
$
V^\circ$ is a subalgebra of $V$.
By Proposition~\ref{gen}, Lemma~\ref{nuovissimo}, and Lemma~\ref{sixtranspositions}, $
V^\circ$ is a subalgebra of $V$ if and only if,  for every pair of permutations $e$ and $c$ of cycle type $3^2$, 
\begin{equation}
\label{vucirc1}
u_{e}\cdot u_{c}\in V^\circ. 
\end{equation} 
Since $V^\circ$ is $\hat G$-invariant and $\hat G$ acts on $V$ as a group of algebra automorphisms, we may assume that $e=e_1$ and $c=e_i$ for $i\in\{1,\ldots,31\}$. 
Let $I_0:=\{1,\ldots,10\}$, $I_1:=\{1,\ldots,20\}$, $I_2:=\{1, \ldots, 25\}$, and $I_3:=\{1,\ldots, 29\}$. 
We first prove recursively that Equation~(\ref{vucirc1}) holds for every $j\in \{0,1,2,3\}$, every $i\in I_j$, and every pair of permutations $e$ and $c$ of type $3^2$ such that $(\langle e\rangle, \langle c \rangle)\in \Omega_i^{12}$. 
As above, we may assume  that $e=e_1$ and $c=e_i$, for every $i\in I_j$. 
Assume first $i\in I_0$, then $e_i\in A_7$. Then Equation~(\ref{vucirc1}) follows  by~\cite[Proposition~2.7(iii)]{A67}. Let $j\in \{1,2,3\}$. By Proposition~\ref{udependent}(j), for every  $i\in I_j$,  $u_{e_i}$ is a linear combination of Majorana axes and $3$-axes $u_{\rho}$ of type $3^2$ such that each $\rho$ is contained in a subgroup $A_{Y\rho}$ of $A_{12}$, for a suitable set $Y_\rho$ of cardinality $7+(j-1)$ with  $\{1,\ldots,6\} \subseteq Y_\rho\subseteq \{1,\ldots, 7+j\}$. The claim then follows by  Lemma~\ref{nuovissimo},  Lemma~\ref{sixtranspositions}, and a recursive argument. 
Finally, if $i\in \{30,31\}$ set 
$$
\begin{array}{lll}
t_1:=(1,2)(4,5), &t_2:=(1,3)(4,6),\\
s_1:=(7,8)(10,11), &s_2:=(6,7)(9,10),\mbox{ and } &s_3:=(7,9)(10,12), 
\end{array}
$$
 so that $e_1=t_1t_2$,  $e_{30}=s_1s_2$, and $e_{31}=s_1s_3$. By the Norton-Sakuma Theorem (Table~\ref{table1}), there exist  vectors $w_1, w_2 \in V^{(2A)}$ such that 
 $$
u_{e_1}=\frac{2^{11}}{27\cdot 5}(a_{t_1}\cdot a_{t_2}+w_1 ) \mbox{ and } u_{e_{30}}=\frac{2^{11}}{27\cdot 5}(a_{s_1}\cdot a_{s_2}+w_2 ).
$$
 Since $t_1$ commutes with all the $s_i$'s, by Axiom M4 and Table~\ref{orbs},  the Norton-Sakuma subalgebra generated by  $a_{t_1}$ and $a_{s_i}$ is of type $2B$. Thus  $a_{s_i}$ is a $0$-eigenvector for the adjoint action of  $a_{t_1}$ and the fusion law for eigenvectors implies that $a_{s_1}\cdot a_{s_2}$ and $a_{s_1}\cdot a_{s_3}$ are $0$-eigenvectors for the adjoint action of  $a_{t_1}$. Hence, by~\cite[Lemma~1.10]{IPSS}, we have
\begin{eqnarray*}
u_{e_1}\cdot u_{e_{30}}&=&
\frac{2^{11}}{27\cdot 5} \left [(a_{t_1}\cdot a_{t_2}+w_1)\cdot (a_{s_1}\cdot a_{s_2}+w_2)\right ] \\
&=&\frac{2^{11}}{27\cdot 5}\left [(a_{t_1}\cdot a_{t_2})\cdot (a_{s_1}\cdot a_{s_2})+ (a_{t_1}\cdot a_{t_2})\cdot w_2 \right .\\
& &\:\:\:\:\:\:\:\:\:\:\left .+ w_1\cdot (a_{s_1}\cdot a_{s_2})+w_1\cdot w_2\right ]\\
&=&\frac{2^{11}}{27\cdot 5}\left [(a_{t_1}\cdot (a_{t_2}\cdot (a_{s_1}\cdot a_{s_2}))+ (a_{t_1}\cdot a_{t_2})\cdot w_2 \right .\\
& &\:\:\:\:\:\:\:\:\:\:\left .+ w_1\cdot (a_{s_1}\cdot a_{s_2})+w_1\cdot w_2\right ] 
\end{eqnarray*}
By Lemma~\ref{nuovissimo} and Lemma~\ref{sixtranspositions} the last expression lies in $V^\circ$, proving that $u_{e_1}\cdot u_{e_{30}}\in V^\circ$. 
Similarly we get  $\:u_{e_1}\cdot u_{e_{31}}\in  V^\circ$. Clearly, in all cases, the product $u_{e}\cdot u_{c}$ is uniquely determined by the factors. 
\end{proof}

\begin{cor}\label{444}
$$
V/\langle V^{(2^2)}, V^{(3A)}\rangle \cong S^{(4,4,4)}.
$$
\end{cor}
\begin{proof}
This follows by comparing the decompositions into irreducible submodules of $V=V^\circ$ and $\langle V^{(2^2)}, V^{(3A)}\rangle$.
\end{proof}

{\it Proof of Theorem~\ref{main}}.
Let $(A_{12}, \mathcal T, V, \phi, \psi)$ be a $2A$-Majorana representation of $A_12$. By Proposition~\ref{shape2^6} the shape of $\phi$ is the one given in Table~\ref{orbs}. It follows that the scalar products between axes and $3$-axes are the ones given by the Norton-Sakuma Theorem (Table~\ref{table1}), Table~\ref{tau3}, Table~\ref{t13}, and Table~\ref{Orbitals12}. By Proposition~\ref{gen} and Proposition~\ref{fine2closure}, $V$ is $2$-closed and generated by the Majorana axes and the $3$-axes of type $3^2$, proving  the third assertion of Theorem~\ref{main}. By Lemma~\ref{nuovissimo}, Lemma~\ref{sixtranspositions},  and Proposition~\ref{fine2closure} it follows that the algebra is uniquely determined and this implies the first assertion. Finally the second assertion follows by Proposition~\ref{resume} and Proposition~\ref{fineintersezione}.
\hfill $\square$

\medskip

{\it Proof of Theorem~\ref{embuno}.} 
Let $F_5$ be the Harada-Norton group and let $$(F_5, \mathcal T_{F_5}, V_{F_5}, \phi_{F_5}, \psi_{F_5})$$ be a  $2A$-Majorana representation of  
$F_5$.   Then $F_5$ contains a subgroup $H$ isomorphic to $A_{12}$ and $\phi_{F_5}$ induces, by restriction, a $2A$-Majorana 
representation $$(H, \mathcal T_{H}, V_{H}, \phi_{H}, \psi_{H})$$ of $H$.  By~\cite[Lemma 3.1]{FIM1}, every $3$-axis (resp. $4$- and $5$-axis) of 
$V_{F_5}$ is conjugate by an element of $F_5$ to an element of $V_H$.  
Since, again by~\cite[Lemma 3.1]{FIM1}, the permutations of cycle type $3^2$ and those of cycle type $3$ in $H$ are fused in $F_5$,  by Theorem~\ref{main}(3) and~\cite[Corollary 3.2]{Alonso},  $V_{F_5}$ is generated by Majorana axes and so the algebra is uniquely determined by~~\cite[Theorem 1.1]{FIM1} and the Norton-Sakuma Theorem.
\hfill $\square$

\medskip

{\it Proof of Theorem~\ref{HN}.} This is immediate \hfill  $\square$

\medskip

{\it Proof of Theorem~\ref{A8}}. 
We identify $A_8$ with the subgroup of $A_{12}$ that fixes points $9, \ldots , 12$.
Using the GAP program "MajoranaAlgebras"~\cite{MM} we computed the Majorana representation of the group $L:=A_8\times \langle (9,10)(11,12) \rangle$ with generating set of involutions $\mathcal T \cap L$ and shape induced by the unique shape of a $2A$-Majorana representation of  $A_{12}$. This representation has dimension $554$ and it is $2$-closed. The subalgebra $W$ generated by Majorana involutions $a_t$, with $t\in \mathcal X_b\cap A_8$ has dimension $476$. On the other hand, Theorem~\ref{An} implies that $W^\circ$ has dimension $462$, whence it follows that $W$ is not $2$-closed.  
We checked that the subset $Z$ of $4$-axes of $V$ defined by
{\small
$$Z:=\left \{ v_{\rho}\:|\: \rho=(a,b,c,d)(e,f,g,h)\in A_8 \mbox{ and }\begin{array}{ll} a\: b \:c \:d\\
                                                                   e \:f \:g \:h
                                                                   \end{array} \mbox{ is a semistandard $(4^2)$-tableau }\right \}
                                                                   $$
                                                                   }
is linearly independent in $W$ and $W=W^\circ\oplus \langle  Z \rangle.$  
\hfill $\square$
\bigskip

%%%%%%%%%%%%%%%%%%%%%%%%%%%%%%%%%%%%%%%%%%%%%%

\section{Appendix}
\label{Appendix}
A formula for expressing a $4$-axis $u_\rho$ of type  $4^2$ as a linear combination of 
Majorana axes and $3$-axes  of type $3^2$ has been obtained working within the subgroup $S$ of $A_{12}$ defined in Lemma~\ref{formulaccia}. We give it here explicitly for  
  $\rho=(1,3,2,4)(5,7,6,8)$:
 {
\tiny
\begin{equation*}
\begin{array}{l}
v_{(1,3,2,4)(5,7,6,8)} = \frac{4}{9}  a_{(1,2)(3,4)(5,6)(7,8)(9,10)(11,12)}\\
\\
+\frac{2}{9}  [ a_{(2,6)(3,8)} +a_{(2,6)(4,7)}+a_{ (2,5)(3,7)}+a_{ (2,5)(4,8)}+a_{ (1,6)(3,7)}+a_{ (1,6)(4,8)}+\\
\\
\:\:\:\:\:\:\:\:\:a_{ (1,5)(3,8)}+a_{ (1,5)(4,7)}] \\
\\
-\frac{2}{3}[ a_{(3,7)(4,8)}+a_{ (3,8)(4,7)}+a_{ (1,6)(2,5)}+a_{ (1,5)(2,6)} ]\\
\\
+\frac{1}{3}[ a_{(3,4)(7,8)}+a_{(1,2)(5,6)}+a_{ (1,8)(2,7)(3,5)(4,6)(9,10)(11,12)}+\\
\\
\:\:\:\:\:\:\:\:a_{ (1,4)(2,3)(5,7)(6,8)(9,10)(11,12)}+a_{ (1,7)(2,8)(3,6)(4,5)(9,10)(11,12)}+\\
\\
\:\:\:\:\:\:\:\:a_{ 
  (1,3)(2,4)(5,8)(6,7)(9,10)(11,12) }]\\
   \\
+\frac{1}{9}[a_{ (1,8)(2,7)(3,6)(4,5)(9,10)(11,12)}+a_{ (1,4)(2,3)(5,8)(6,7)(9,10)(11,12)}+\\
\\
\:\:\:\:\:\:\:\:a_{ (1,7)(2,8)(3,5)(4,6)(9,10)(11,12)}+ a_{  (1,3)(2,4)(5,7)(6,8)(9,10)(11,12)}+\\
\\
\:\:\:\:\:\:\:\:a_{ (1,6)(2,5)(3,8)(4,7)(9,10)(11,12)} 
+a_{  (1,6)(2,5)(3,7)(4,8)(9,10)(11,12)}+ \\
\\
\:\:\:\:\:\:\:\:a_{  (1,5)(2,6)(3,7)(4,8)(9,10)(11,12)}+a_{ (1,5)(2,6)(3,8)(4,7)(9,10)(11,12)} ]\\
\\
  -\frac{1}{9}[ a_{(5,6)(7,8)}+a_{ (1,2)(3,4)} ]\\
  \\
   -\frac{7}{9}[ a_{(3,4)(5,6)}+a_{ (1,2)(7,8)} ]\\
 \\
 -\frac{2}{9} [ a_{(1,8)(2,3)(4,6)(5,7)(9,10)(11,12)}+a_{ (1,8)(2,4)(3,5)(6,7)(9,10)(11,12)}+\\
 \\
 \:\:\:\:\:\:\:\:a_{ (1,2)(3,7)(4,8)(5,6)(9,10)(11,12)}+a_{ 
  (1,2)(3,8)(4,7)(5,6)(9,10)(11,12)}+\\
  \\
  \:\:\:\:\:\:\:\:a_{ (1,4)(2,7)(3,5)(6,8)(9,10)(11,12)} +a_{(1,4)(2,8)(3,6)(5,7)(9,10)(11,12)}+\\
  \\
   \:\:\:\:\:\:\:\:a_{ 
  (1,7)(2,4)(3,6)(5,8)(9,10)(11,12)}+a_{ (1,7)(2,3)(4,5)(6,8)(9,10)(11,12)}+\\
  \\
 \:\:\:\:\:\:\:\: a_{ (1,3)(2,8)(4,5)(6,7)(9,10)(11,12)}+a_{ 
  (1,3)(2,7)(4,6)(5,8)(9,10)(11,12)} +\\
  \\ 
  \:\:\:\:\:\:\:\:a_{(1,6)(2,5)(3,4)(7,8)(9,10)(11,12)}+a_{ (1,5)(2,6)(3,4)(7,8)(9,10)(11,12)} ]\\
  \\
 
 + \frac{5}{64}[ u_{(1,3,4)(5,6,8)}+u_{ (1,2,3)(5,8,7)}+u_{ (1,7,8)(3,6,5)}+u_{ (1,8,2)(3,5,4)}+ u_{(1,2,4)(5,7,8)}+ \\
 \\
 \:\:\:\:\:\:\:\:u_{(1,4,2)(6,7,8)}+u_{ (1,3,2)(6,8,7)}+u_{ (1,4,3)(5,6,7)} +
  u_{(2,3,4)(5,8,6)}+u_{ (2,4,3)(5,7,6)}+\\
  \\
  \:\:\:\:\:\:\:\:u_{ (1,7,8)(4,5,6)} +
  u_{(1,2,7)(3,5,4)}+u_{ (2,8,7)(3,6,5)}+u_{ (1,2,8)(3,6,4)}+u_{ (1,7,2)(3,6,4)}+\\
  \\
  \:\:\:\:\:\:\:\:u_{ (2,7,8)(4,6,5)} ]\\
  \\
- \frac{5}{64}[ u_{(1,3,2)(5,8,6)}+u_{ (1,2,4)(5,6,7)}+u_{ (1,2,8)(3,5,6)}+u_{(1,7,2)(4,6,5)}+u_{ (1,4,2)(5,6,8)}+ \\
 \\
 \:\:\:\:\:\:\:\:u_{(1,3,2)(5,6,7)}+u_{ (1,5,2)(3,8,4)}+u_{ (1,4,3)(5,7,8)}+u_{ 
  (1,6,2)(3,7,4)}+u_{ (1,2,6)(3,8,4)}+\\
  \\
  \:\:\:\:\:\:\:\:u_{ (1,3,4)(6,7,8)} +u_{  (2,3,4)(5,7,8)}+u_{ (2,6,5)(3,8,7)}+u_{(1,5,2)(3,4,7)}+u_{ (1,2,8)(4,6,5)}+\\
  \\
  \:\:\:\:\:\:\:\:u_{ (1,2,7)(3,6,5)} +u_{ (2,3,4)(6,8,7)} +u_{ (2,5,6)(4,8,7)}+u_{ (1,6,5)(4,8,7)}+u_{ (1,5,6)(3,8,7)} +\\
  \\
  \:\:\:\:\:\:\:\:u_{(1,7,8)(3,5,4),}+u_{(2,7,8)(3,6,4)}+u_{ (2,7,8)(3,4,5)}+u_{ (1,8,7)(3,6,4)} ]\\
       \end{array}
  \end{equation*}
 \begin{equation*}
 \begin{array}{l}
  \\
  +\frac{5}{32}  [ u_{(1,3,5)(2,4,6)}+u_{ (1,7,3)(4,8,5)}+u_{ (1,6,3)(2,8,5)}+u_{ (1,3,8)(2,4,7)}+ u_{(1,4,7)(2,3,8)}+ \\
  \\
 \:\:\:\:\:\:\:\: u_{(1,5,8)(2,6,7)}+u_{ (1,5,3)(2,7,6)}+u_{ (1,8,6)(2,5,3)}+u_{ (1,4,8)(3,5,7)}+u_{ (1,7,5)(2,8,6)}+\\
 \\
 \:\:\:\:\:\:\:\:u_{ (3,6,8)(4,5,7)} +u_{  (2,3,7)(4,6,8)}+u_{ (3,5,8)(4,6,7)}+u_{ (1,5,4)(2,6,3)}+u_{(1,4,6)(2,3,5)}+\\
 \\
 \:\:\:\:\:\:\:\:u_{ (1,6,8)(2,5,7)}+u_{(1,8,4)(2,7,3)}+u_{ (3,7,5)(4,8,6)}+u_{ (1,3,6)(2,4,5)}+u_{ (1,5,8)(2,4,6)}+\\
 \\
 \:\:\:\:\:\:\:\:u_{ (1,8,3)(4,7,6)}+u_{ (2,7,4)(3,8,5)}+u_{ (1,6,4)(2,7,5)} +u_{   (1,7,6)(2,5,4)}+ 
  u_{(2,8,4)(3,7,6)}+\\
  \\
  \:\:\:\:\:\:\:\:u_{ (1,7,5)(2,6,3)}+u_{ (1,4,5)(2,6,8)}+u_{ (2,8,3)(4,7,5)}+u_{ (1,7,3)(2,8,4)}+u_{  (1,4,7)(3,6,8)}+\\
  \\
  \:\:\:\:\:\:\:\: u_{(3,6,7)(4,5,8)}+u_{ (1,6,7)(2,5,8)} ]\\
  \\
 - \frac{5}{32} [ u_{(1,3,8)(2,6,4)}+u_{ (1,5,7)(3,8,6)}+u_{ (1,3,5)(2,7,4)}+u_{ (1,4,5)(3,8,6)}+ u_{(1,7,4)(2,6,8)}+ \\
  \\
 \:\:\:\:\:\:\:\:u_{ (1,3,7)(2,8,5)}+u_{ (1,6,3)(4,5,8)}+u_{ (1,8,5)(4,6,7)}+ 
 u_{ (1,5,4)(2,3,8)}+u_{ (1,6,4)(2,3,7)}+\\
 \\
 \:\:\:\:\:\:\:\:u_{ (2,6,8)(4,7,5)} +u_{  (2,6,7)(3,8,5)}+u_{ (2,5,7)(4,8,6)}+u_{ (1,3,5)(4,7,6)}+u_{(1,5,8)(2,7,4)}+\\
 \\
 \:\:\:\:\:\:\:\:u_{ (2,4,6)(3,8,5)}+ 
 u_{ (1,8,3)(2,6,7)}+u_{(1,4,7)(2,6,3)}+u_{ (2,3,6)(4,7,5)}+u_{ (1,5,7)(2,8,3)}+\\
 \\
 \:\:\:\:\:\:\:\:u_{ (1,6,8)(3,7,5)}+u_{ (1,8,6)(2,3,7)}+u_{ (1,6,3)(2,4,8)} +  
   u_{(2,5,3)(4,6,8)}+
   u_{(2,8,5)(3,6,7)}+\\
   \\
   \:\:\:\:\:\:\:\:u_{ (1,7,3)(2,4,5)}+u_{ (1,6,7)(4,8,5)}+u_{ (1,6,4)(3,5,7)}+u_{ (1,4,8)(2,7,5)} +u_{(2,4,5)(3,7,6)}+\\
   \\
   \:\:\:\:\:\:\:\:u_{ (1,7,6)(2,4,8)}+u_{ (1,8,4)(2,3,5)} ]\\
  \\
  +\frac{25}{64} [ u_{(1,5,6)(3,4,8)}+u_{ (2,5,6)(3,4,7)}+u_{ (1,2,5)(3,8,7)}+u_{ (1,2,6)(4,8,7)}+ u_{(1,6,2)(3,8,7)}+\\
  \\
\:\:\:\:\:\:\:\:   u_{(2,6,5)(3,4,8)}+u_{ (1,5,6)(3,7,4)}+u_{ (1,2,5)(4,7,8)} ]. 
 \end{array}
\end{equation*}
} 

%%%%%%%%%%%%%%%%%%%%%%%%%%%%%%%%%%%%%%%%%%%%%%%%%%

\end{document}